\documentclass{article}
\usepackage{amssymb}
\usepackage{amsmath}
\usepackage{amsthm}
\usepackage{dsfont}
\usepackage{color}
\usepackage{hyperref}
\usepackage{fullpage} 
\usepackage{xcolor}
\usepackage{empheq}
\usepackage{caption}
\usepackage{subcaption} 
\usepackage{graphics} 
\usepackage{epsfig} 
\usepackage{graphicx}
\usepackage{epstopdf}
\usepackage{mathabx}
\usepackage{capt-of} 

\newcommand{\CC}{\mathbb C}
\newcommand{\EE}{\mathbb E}
\newcommand{\FF}{\mathbb F}
\newcommand{\GG}{\mathbb G}

\newcommand{\PP}{\mathbb P}

\newcommand{\RR}{\mathbb R}

\newcommand{\cB}{\mathcal B}
\newcommand{\cF}{\mathcal F}
\newcommand{\cG}{\mathcal G}

\newcommand{\cL}{\mathcal L}

\newcommand{\bfS}{\mathbf S}
\newcommand{\bfT}{\mathbf T}
\newcommand{\bfU}{\mathbf U}

\newcommand{\diver}{\mathrm{div}}
\newcommand{\grad}{\nabla}
\newcommand{\ds}{\displaystyle}
\newcommand{\indic}{{\mathds 1}}

 \newtheorem{theorem}{\textbf{Theorem}}[section]
 \newtheorem{remark}[theorem]{\textbf{Remark}}
 \newtheorem{lemma}[theorem]{\textbf{Lemma}}
 \newtheorem{proposition}[theorem]{\textbf{Proposition}}
 
\makeatletter
\newcommand{\pushright}[1]{\ifmeasuring@#1\else\omit\hfill$\displaystyle#1$\fi\ignorespaces}
\newcommand{\pushleft}[1]{\ifmeasuring@#1\else\omit$\displaystyle#1$\hfill\fi\ignorespaces}
\makeatother

\title{Optimal control of conditioned processes\\ with feedback controls}
\author{Yves Achdou\thanks { Univ. Paris Diderot, Sorbonne Paris Cit{\'e}, Laboratoire Jacques-Louis Lions, UMR 7598, UPMC, CNRS, F-75205 Paris, France.
 achdou@ljll.univ-paris-diderot.fr}, Mathieu Lauri{\`e}re\thanks {Department of Operations Research and Financial Engineering, Princeton University, Sherrerd Hall, Charlton Street, Princeton, NJ 08544. lauriere@princeton.edu}, Pierre-Louis Lions\thanks {Coll{\`e}ge de France, Paris, France}}

\begin{document}

\maketitle

\begin{abstract}
  We consider a class of closed loop  stochastic optimal control problems  in finite time horizon, in which the cost 
is an expectation conditional on the event that the process has not 
exited a given bounded domain. An important difficulty is that the probability of the  event that conditionates the strategy
 decays as time grows. The optimality conditions consist of  a system of partial differential equations, including  a Hamilton-Jacobi-Bellman equation (backward w.r.t. time) and a (forward w.r.t. time) Fokker-Planck equation  for the law of the conditioned process.  The two equations are supplemented with Dirichlet conditions.
Next, we discuss  the asymptotic behavior as the time horizon tends to $+\infty$. This leads to a new kind of optimal control problem driven by an eigenvalue problem related to a continuity equation with Dirichlet conditions on the boundary. We prove existence for the latter.
We also propose numerical methods and supplement the various theoretical aspects with numerical simulations.
\end{abstract}

\section{Introduction}


In optimal control, the goal is generally to control the dynamics of a system so as to minimize a certain cost or, equivalently, to maximize a certain profit.
If the evolution of the system is stochastic, the cost is computed in expectation over the possible realizations of the randomness.
In this work, we consider a stochastic optimal control problem in which the optimization is performed conditionally on the occurrence or the non occurrence  of a certain event.

Hence, what follows can be seen as an attempt to model limited rationality: the agent determines her strategy but disregards the possibility of  some  event.
Such a situation unfortunately occurs in politics,  for example when the effects of some decisions on the global climate change are not taken into account.
There may even be situations when the probability of the  event that conditionates the strategy becomes smaller and smaller,
i.e. in the example above, the 
 leader keeps thinking that the situation is safe whereas the probability of disastrous events becomes larger and larger.


In what follows, we focus on the case when the event of interest is that the process stays inside a given bounded domain, but of course one may think about many other situations.


A quite important part of the material contained in the present work directly comes 
from the lectures given by the third author in  Coll{\`e}ge de France in November and December 2016, see \cite{PLL}. 
However, we will also discuss theoretical aspects including some technical details that were not dealt with in  \cite{PLL} for lack of time, 
propose numerical methods and illustrate the main ideas by numerical simulations. In particular, the latter  will shed some light on  some open problems related to the long time behavior.

	In the sequel, we consider a bounded domain $\Omega \subset \RR^d$ with a smooth boundary denoted by $\partial \Omega$. The closure of $\Omega$ in $\RR^d$ is denoted by $\overline \Omega$. For a time horizon $T>0$, $Q_T$ stands for the time-space cylinder $(0,T) \times \Omega$. Let $\indic_\Omega$ be the indicator function of $\Omega$, taking values in $\{0,1\}$ such that $\indic_\Omega(x) = 1$ if and only if $x \in \Omega$.
	Let $W = (W_t)_{t \geq 0}$ be a standard $d$-dimensional Brownian motion. We consider feedback controls (also called  closed loop or Markovian controls) in $L^\infty((0,+\infty) \times \RR^{d}; \RR^d)$, which are deterministic functions of $t$ and $x$. For such a control $b \in L^\infty((0,+\infty) \times \RR^{d}; \RR^d)$, let $X^b = (X^b_t)_{t\geq 0}$ be a solution to the SDE
	\begin{equation}
	\label{eq:sde-X}
		d X^b_t = -b(t,X^b_t) dt + \sigma d W_t.
	\end{equation}
	We assume that the distribution of $X^b_0$ is absolutely continuous with respect to Lebesgue measure with density $p_0$, and that $p_0$ is a smooth nonnegative function supported in $\Omega$. 
	The first time $X^b$ exits $\Omega$ is denoted by $\tau^b$:  $\tau^b = \inf\{t > 0 \,:\, X^b_t \notin \Omega\}$.
	
To motivate the problem that will be studied below, let us first focus on a simple case. 
	Let  $g: \overline \Omega \to \RR^d$ be a continuous function and $L: \overline \Omega \times \RR^d \to \RR$ be a  strictly convex function, continuously differentiable in its second argument.
	The problem is to minimize over feedback controls $b \in L^\infty((0,+\infty) \times \RR^{d}; \RR^d)$ the total cost
\begin{equation}\label{eq:intro-cost}
	\int_0^T \EE\left[  L(X^b_t, b(t,X^b_t)) \, | \, \tau^b >t \right] dt
	+
	\EE\left[ g(X^b_T)\, | \, \tau^b >T \right]
\end{equation}
where $X^b$ has been defined above.	
	
We can rewrite the above problem as an optimal control problem driven by a Kolmogorov-Fokker-Planck (KFP) equation. The infinitesimal generator associated to $X^b$ is 
 \begin{equation}
 \label{eq:def_Lb}
 	\cL_b \varphi = -\frac{\sigma^2}{2}\Delta \varphi + b \cdot D \varphi.
\end{equation}
 The distribution of $X^b_t$ has a density $m_b(t)$, and $m_b$ solves the KFP equation:
$$
	\partial_t m_b + \cL^*_b m_b = 0, \qquad \hbox{ in } (0,T) \times \RR^d, 
$$
with initial condition $m_b(0) = p_0$ on $\Omega$, where
\begin{equation}
 \label{eq:def_Lb_star}
	\cL_b^* \varphi = -\frac{\sigma^2}{2}\Delta \varphi - \diver( b \varphi)
\end{equation}
is the dual operator associated to $\cL_b$.
 
 Let us now introduce $p_b$ satisfying the Dirichlet problem
 \begin{equation}
\label{eq:FP-Dirichlet}
\left\{\quad
\begin{aligned}
	\partial_t p_b + \cL_b^* p_b &= 0, \qquad \hbox{ in } Q_T,
	\\
	p_b &= 0, \qquad \hbox{ on } (0,T) \times \partial \Omega,
	\\
	p_b(0,\cdot) &= p_0, \qquad \hbox{ on } \Omega.
\end{aligned}
\right.
\end{equation}
It can be checked that  for any smooth $f: \RR^d \to \RR$,
$$
	\EE \left[ f(X^b_t) \indic_{\{\tau^b > t\}} \right] = \int_\Omega f(x) p_b(t,x) dx.
$$ 
In particular, $\int_\Omega p_b(t,x) dx = \PP(\tau^b > t)$ represents the probability of ``survival'' until time $t$. To alleviate the notations, this quantity will sometimes be denoted simply by $\int_\Omega p_b(t)$.
Hence the running cost in~\eqref{eq:intro-cost} can be written as
$$
	\EE\left[  L(X^b_t, b(t,X^b_t)) \, | \, \tau^b >t \right]
	=
	\frac{\EE\left[  L(X^b_t, b(t,X^b_t)) \indic_{\{\tau^b > t\}} \right]}{ \PP(\tau^b > t) }
	=
	\frac{\int_\Omega L(x, b(t,x)) p_b(t,x) dx}{\int_\Omega p_b(t)}.
$$
Similarly, the final cost can be written as
$$
	\EE\left[ g(X^b_T)\, | \, \tau^b >T \right]
	=
	\frac{\int_\Omega g(x) p_b(T,x) dx}{ \int_\Omega p_b(T)}.
$$
Overall, the problem can be recast as the following deterministic control problem driven by a KFP equation: minimize 
$$
	\int_0^T \frac{\int_\Omega L(x, b(t,x)) p_b(t,x) dx}{\int_\Omega p_b(t)} dt + \frac{\int_\Omega g(x) p_b(T,x) dx}{\int_\Omega p_b(T)}
$$
subject to~\eqref{eq:FP-Dirichlet}.

 We stress that to obtain the latter formulation, it is important to consider only controls that are in feedback form.
In the sequel  we will consider a slightly more general class of problems, in which the costs are given by functionals acting on $\frac{p_b(t,\cdot)}{\int_\Omega p(t)}$. 

\begin{remark}
\label{rem:b-res-QT}
	From~\eqref{eq:FP-Dirichlet} it is clear that only the restriction of $b$ to $Q_T$ matters in the latter optimal control problem.
\end{remark}
We will see below that ${\int_\Omega p_b(t)}$ decays as $t$ grows. This represents an important difficulty at least for two aspects:
\begin{itemize}
\item in numerical simulations, the division by very small quantities is a problem
\item the asymptotic behavior of the problem as the horizon tends to $+\infty$.
\end{itemize}

The paper is organized as follows. In Section  \ref{sec:finite-horizon-case}, we  present the theory on finite horizon conditional control problems (as it was discussed in \cite{PLL}), focusing on the case when the control is bounded in $L^\infty$ by a fixed constant $M$: the existence of a minimizer as well as the optimality conditions are discussed; the  latter have the form of a forward-backward system coupling a Dirichlet problem involving a forward Fokker-Planck equation and a Dirichlet problem involving a Bellman equation, as in the theory of mean field games, see \cite{MR2271747,MR2269875,MR2295621}, and in  the theory of mean field type  control see \cite{MR3091726,MR3752669,MR3134900}.
We prove existence of solutions for this system. Next, we describe a possible asymptotic behavior when the time horizon becomes large: to the best of our knowledge, proving these asymptotics is a difficult open problem. Plugging the ansatz into the previously mentioned system of PDEs leads to an interesting new system, coupling a principal eigenvalue problem involving a non symmetric second order elliptic operator and  a stationnary  Belmann equation. In order to study the latter problem, we give useful facts on principal eigenvalue problems for non symmetric operators in Section \ref{sec:facts-about-princ}. In Section~\ref{sec:long-time-behavior}, we prove existence for the system that has been introduced at the end of Section   \ref{sec:finite-horizon-case}, by showing first  that this system can be seen as the first order optimality conditions of an optimal control problem driven by a principal eigenvalue problem (this material was also discussed in \cite{PLL}). 
Section~\ref{sec:M-to-infty} is independent from Sections \ref{sec:facts-about-princ} and \ref{sec:long-time-behavior}: it deals with the passage to the limit  $M\to+\infty$ ($M$ is the parameter introduced in Section \ref{sec:finite-horizon-case}), and relies very much on stability results for weak solutions of Fokker-Planck and Hamilton-Jacobi equations.
For the finite horizon problem, a  finite difference  method, reminiscent of that introduced in \cite{ MR3135339,MR2679575} is proposed in Section~\ref{sec:numerics-time} and numerical simulations are performed for
one and two dimensional examples. In particular, the results are in agreement with  the ansatz made at the end of  Section  \ref{sec:finite-horizon-case}. 
Finally, in Section \ref{sec:numer-meth-stat},  a numerical method is proposed for the stationary problem discussed in Section \ref{sec:long-time-behavior}, and the results 
include the asymptotics  when the horizon tends to infinity.

\section{The finite horizon case}
\label{sec:finite-horizon-case}

\subsection{Existence of a minimizer and necessary optimality conditions}

As above, let $L: \overline \Omega \times \RR^d \to \RR$ be a continuous function,  strictly convex and continuously differentiable in its second argument.
Let $\Phi,\Psi: L^2(\Omega) \to \RR$ be Fr{\'e}chet differentiable functionals which are bounded from below. The gradients of $\Phi$ and $\Psi$ at $p \in L^2(\Omega)$ are respectively denoted by $F[p]$ and $G[p]$, that is, $F[p],G[p] \in L^2(\Omega)$ and $(F[p], q)_{L^2(\Omega)} = D \Phi[p](q)$, $(G[p], q)_{L^2(\Omega)} = D \Psi[p](q)$ for all $q \in L^2(\Omega)$. Let $\epsilon \ge 0$ be a fixed constant (we stress that unless otherwise specified, $\epsilon$ can be $0$). 
In view of Remark~\ref{rem:b-res-QT}, 
we introduce the following notations for the set of controls:
$
	\cB = L^\infty(Q_T; \RR^d),
$
and for every $M>0$, $\cB_{M} = \{b \in L^\infty(Q_T; \RR^d) \,:\, \|b\|_{L^\infty} \leq M   \}$ is the subset of controls bounded by $M$.
Then, for $b \in \cB$, we consider the cost functional
\begin{align}
\label{eq:def-time-J}
	J(b) 
	&= \int_0^T\left(\frac{\int_\Omega p_b(t,x) L(x,b(t,x)) dx}{\int_\Omega p_b(t)} + \Phi\left[ \frac{p_b(t,x)}{\int_\Omega p_b(t)} \right]\right) dt 
	+ \Psi\left[\frac{p_b(T,\cdot)}{\int_\Omega p_b(T)}\right] - \epsilon \ln\left( \int_\Omega p_b(T) \right) ,
\end{align}
where $p_b$ solves~\eqref{eq:FP-Dirichlet}. We are going to address the following optimal control problem:
\begin{equation}
\label{eq:def-pb-J-BM}
	\text{ minimize  $J$ on $\cB_M$. } 
\end{equation}

\begin{remark}
	The last term in~\eqref{eq:def-time-J} has a different nature from the other ones because it does not depend on the conditional probability $\frac{p_b(T,\cdot)}{\int_\Omega p_b(T)}$. If $\epsilon>0$, it penalizes the decay of the probability of survival. This term will be helpful in Section~\ref{sec:M-to-infty}.
\end{remark}

We start with the following result.
\begin{theorem}
\label{thm:existence-minimizer-J}
	For each $b \in \cB$, the unique weak solution $p_b$ to~\eqref{eq:FP-Dirichlet} is continuous in $\overline Q_T$ and positive in $Q_T$. The cost $J(b)$ given by~\eqref{eq:def-time-J} is well defined for each $b \in \cB$. Moreover, for every $M>0$, there exists $b_{\rm{opt}} \in \cB_M$ which minimizes $J$ over $\cB_M$.
\end{theorem}
\begin{remark}
	Due to the lack of convexity of $J$, we do not know if the minimizer is unique.
\end{remark}

\begin{proof}
Let us first give some relevant information on~\eqref{eq:FP-Dirichlet} and prove that $J(b)$ is well defined for all $b\in \cB$.

  Standard arguments ensure that for all $b\in \cB$, there exists a unique weak solution  $p\in  L^2(0,T; H^1_0(\Omega))\cap W^{1,2}(0,T;H^{-1}(\Omega)) $ to~\eqref{eq:FP-Dirichlet}.
 Maximum norms estimates  for equations in divergence form (see  \cite[Corollary 9.10]{MR1465184})) tell us  
that there exists a positive constant $\bar p$ which depends on $p_0$, $\Omega$, $T$ and $\| b\|_{L^\infty}$ such that 
\begin{equation}
\label{eq:p-pbar}
  0\le p \le \bar p, \quad \hbox{a.e. in }  Q_T.
\end{equation}
Therefore,  if $\|b\|_\infty\le M$,  $\|p\|_{L^\infty}$ is bounded by a constant which depends only on $p_0$, $\Omega$, $T$ and $M$. 

From~\eqref{eq:p-pbar}, \eqref{eq:FP-Dirichlet} can be written
\begin{equation}
\label{eq:57}
\frac {\partial  p}{\partial t}- \Delta  p +\diver\left(  B\right)  =0,\quad \hbox{in }  Q_T,
\end{equation}
where $B\in L^\infty(Q_T; \RR^d)$.  H{\"o}lder  estimates for  the heat equation with a right hand side in divergence form 
(see  \cite[Theorem 6.33]{MR1465184})
 tell us that $p\in C^{\alpha/2,  \alpha} (\overline Q_T)$, for some $0<\alpha<1$,
and that $\|p\|_{ C^{\alpha/2,  \alpha} (\overline Q_T)}$ is bounded  by a constant which depends only on $p_0$, $\Omega$, $T$ and $\|b\|_\infty$.

Moreover, using the nonnegativity of $p$, a parabolic version of Harnack inequality, see \cite[Theorem 6.27]{MR1465184}, and the fact that $\int_\Omega p_0=1$, we can prove by contradiction  that $p>0$ in $(0,T]\times \Omega$. 
This allows one to define $J(b)$ for all $b\in \cB$.

Consider now a minimizing sequence $(b_n)_{n}$ in $\cB_M$ and 
$p_n$ the solution of~\eqref{eq:FP-Dirichlet} corresponding to $b_n$. 
From the above estimates, we may assume that, up to the extraction of a subsequence,  $b_n\rightharpoonup b$ in $L^\infty(Q_T;\RR^d)$ weak * and that 
$p_n$ tends to $p$ in $C(\overline Q_T)$. It is easy to prove that $p$ is the unique  solution  to~\eqref{eq:FP-Dirichlet} corresponding to $b$.
Using the assumptions on $L, \Phi$ and $\Psi$,
we can also see from the latter convergence  that
\begin{displaymath}
   J(b)\le \liminf  J(b_n).
\end{displaymath}
 Hence, $b$ achieves the minimum of $J$.
\end{proof}

\begin{remark}
  \label{sec:conn-with-finite-2}
The map $b\mapsto p$, where $p$ is the solution of~\eqref{eq:FP-Dirichlet}, is locally Lipschitz continuous from  $L^\infty(Q_T;\RR^d)$ to $ L^2(0,T; H^1_0(\Omega))\cap W^{1,2}(0,T;H^{-1}(\Omega))$.
Hence, $b\mapsto J(b)$ is  locally Lipschitz continuous using the assumptions on $L, \Phi$ and $\Psi$.
\end{remark}

Next, for $M>0$, we obtain first order necessary optimality conditions for~\eqref{eq:def-pb-J-BM} in the form of a forward-backward PDE system. Let us introduce the Hamiltonian $H: \overline \Omega \times \RR^d \to \RR$,
\begin{equation}
\label{eq:def-H}
	H(x,\xi)= \max_{b \in \RR^d;\,  |b|\le M } \left( \xi \cdot b - L(x,b) \right).
\end{equation}

\begin{remark}
	In this section $M$ is fixed. Later on, when necessary, we will change the notation to make the dependency of the Hamiltonian on $M$ explicit. 
\end{remark}

\begin{theorem}
\label{thm:time-optcond}
	Let $b_{\rm{opt}} \in \cB_M$ be a minimizer of $J$ over $\cB_M$. Then  for almost every $(t,x) \in Q_T$, $b_{\rm{opt}}(t,x)$ achieves the maximum in~\eqref{eq:def-H} with $\xi$ replaced by $\left(\int_\Omega p(t,y) dy\right) Du(t,x)$, namely,
	$$
		b_{\rm{opt}}(t,x) = H_\xi\left (x, \left(\int_\Omega  p(t) \right)  Du(t,x)\right),
	$$
	where $(p,u)$ solves the forward-backward PDE system
	\begin{equation}
	\label{eq:time-HJB}
	\left\{ \quad
	\begin{aligned}
	&-\partial_t u(t,x) -\frac {\sigma^2} 2 \Delta u(t,x) +\frac { H\left(x, \left(\int_\Omega  p(t) \right) Du(t,x)\right) }  {\int_\Omega  p(t)}=   \ds \frac {F\left[ \frac {p(t,\cdot)} {\int_\Omega p(t) } \right](x)} {\int_\Omega  p(t)}  +c_1(t)  ,
	 \quad \hbox{ in } Q_T,
	 \\
 	& u = 0, \qquad \hbox{ on } (0,T)\times\partial \Omega,
	\\ 
	& u(T,x)=  \frac 1  {\int_\Omega  p(T)} G\left[ \frac {p(T,\cdot)} {\int_\Omega p(T) } \right](x)  +  c_2(T),
	\qquad \hbox{ in } \Omega ,
    \end{aligned}
	\right.
	\end{equation}

\begin{equation}
\label{eq:time-FP}
 \left\{ \quad 
 \begin{aligned}
	& \partial_t p(t,x) -\frac {\sigma^2} 2 \Delta p(t,x) - \diver\left( p(t,\cdot) H_\xi \left(\cdot,  \left(\int_\Omega  p(t)\right)   Du(t,\cdot) \right) \right)(x)=0, \quad  \hbox{in } Q_T,
	\\
	& p = 0,\quad\quad \qquad \hbox{ on } (0,T)\times\partial \Omega,
	\\
	& p(0,\cdot) = p_0, \qquad \hbox{ in } \Omega, 
  \end{aligned}
\right.
\end{equation}
and $c_1,c_2: [0,T] \to \RR$ are defined by

\begin{equation}
  \label{eq:62}
  \begin{split}
 c_1(t) &=
  -  \frac { {\ds \int_\Omega} p(t,x)  \left( L \left(x,  H_\xi (x, \left(\int_\Omega  p(t) \right)  Du(t,x) ) \right)    + 
    F\left[ \frac {p (t,\cdot)} {\int_\Omega p(t) } \right](x)  \right)  dx }  {\left( \int_\Omega  p(t)\right)^2 },
	\\
     c_2(T) &= - \frac { \epsilon}  {\int_\Omega  p(T) } - \frac {\int_\Omega p(T,x) G\left[ \frac {p(T,\cdot)} {\int_\Omega p(T) } \right](x)  dx}  {\left(\int_\Omega  p(T)\right)^2}  .
  \end{split}
\end{equation}
\end{theorem}
\begin{proof}
Let $b \in\cB_M$ and let $p_b$ be the corresponding solution of~\eqref{eq:FP-Dirichlet}. 
Let $\delta b \in L^\infty(Q_T; \RR^d)$ be a small variation of $b$, and $\delta p$ be the corresponding variation of $p$ ($\delta p$ depends on $b$ and $\delta b$ but we do not write explicitly this dependence to save notations). Neglecting higher order terms thanks to Remark~\ref{sec:conn-with-finite-2},
 \begin{equation*}
 \left\{ \quad
 \begin{aligned}
	& \partial_t \delta p   -\Delta \delta p -\diver(  \delta p b) =   \diver(  p_b \delta b) ,\quad \hbox{in } Q_T,
	\\
	& \delta p = 0,\quad \hbox{on } (0,T)\times\partial \Omega,
	\\
	& \delta p(0,\cdot) = 0,\quad \hbox{in }  \Omega.
\end{aligned}
\right.
 \end{equation*}
At leading order,  the variation of the cost is
\begin{equation*}
\begin{split}
	\delta J 
	=
	& -\epsilon \frac {\int_\Omega \delta p (T,x) dx} {\int_\Omega  p (T,x) dx} 
	 \\
	 &
	\quad + \frac{1}{\int_\Omega p_b(T)} \int_\Omega G\left[\frac{p_b(T,\cdot)}{\int_\Omega p_b(T)}\right](x) \delta p(T,x) dx
	- \frac{\int_\Omega \delta p(T)}{\left(\int_\Omega p_b(T)\right)^2} \int_\Omega p_b(T,y) G\left[\frac{p_b(T,\cdot)}{\int_\Omega p_b(T)}\right](y) dy
	\\ 
	&
	\quad + \int_0^T \left\{    \frac {\int_\Omega p_b(t,x) L_b(x, b(t,x)) \cdot \delta b(t,x) dx    + \int_\Omega  \delta  p(t,x)  \left( L(x, b(t,x)) +  F\left[ \frac {p_b(t,\cdot)} {\int_\Omega p_b(t) } \right]\right) dx  }
 {\int_\Omega  p_b(t) }
 	\right.
	\\
	 & \qquad\qquad\qquad 
	\left.-   \frac {\left(\int_\Omega \delta p(t)\right)  \int_\Omega p_b(t,x) \left( L(x, b(t,x)) +  F\left[ \frac {p_b(t,\cdot)} {\int_\Omega p_b(t) } \right]\right) dx } {\left(\int_\Omega p_b(t)\right)^2 }    \right\} dt.  
\end{split}
\end{equation*}

Let us consider the adjoint problem (for the same control $b$):
\begin{equation*}
\left\{ \quad
\begin{array}{l l}
	-\partial_t u -\Delta u + b\cdot Du 
	=
	\ds \frac {  L(\cdot, b) +  F\left[ \frac {p_b} {\int_\Omega p_b } \right] }  {\int_\Omega  p_b }   -  \frac {\int_\Omega p_b  \left( L(\cdot, b) +  F\left[ \frac {p_b} {\int_\Omega p_b } \right]\right) }  {\left(\int_\Omega  p_b\right)^2 }  \mathds{1}_{\Omega},     \quad  & \hbox{in } Q_T,
	\\
	u = 0, \quad & \hbox{on } (0,T)\times\partial \Omega,
	\\ 
	u(T,x) = \ds  -\epsilon \frac { \mathds{1}_{\Omega}(x)}  {\int_\Omega  p_b(T) } 
		+ \frac{G\left[\frac{p_b(T,\cdot)}{\int_\Omega p_b(T)}\right](x)}{\int_\Omega p_b(T)}  
	- \frac{\int_\Omega p_b(T,y) G\left[\frac{p_b(T,\cdot)}{\int_\Omega p_b(T)}\right](y) dy}{\left(\int_\Omega p_b(T)\right)^2}\mathds{1}_{\Omega}(x), \quad & \hbox{in } \Omega .
\end{array}
\right.
\end{equation*}
Let us denote by $u_b$ the solution to the above problem.
Using the equation on $u_b$, integration by parts, and the equation on $\delta p$, we see that, at leading order
\begin{displaymath}
  \begin{split}
	  \delta J
	  &=  \int_0^T \left\{    \frac {\int_\Omega p_b L_b(\cdot, b)\cdot \delta b }  {\int_\Omega  p_b } + \int_\Omega    \delta  p  \left(-\partial_t u_b -\Delta u_b +b\cdot Du_b \right) \right\} dt 
	  + \int_\Omega \delta p (T,x) u_b(T,x)dx     
	\\
	&= \int_0^T \left\{    \frac {\int_\Omega p_b L_b(\cdot, b) \cdot \delta b }  {\int_\Omega  p_b } +\left \langle  \partial_t \delta p   -\Delta \delta p -\diver(  \delta p b) ,  u_b \right \rangle \right\}dt  
	\\
	&= \int_0^T \left\{    \frac {\int_\Omega p_b L_b(\cdot, b) \cdot \delta b }  {\int_\Omega  p_b } - \int_\Omega   p_b \delta b\cdot Du_b \right\} dt.
  \end{split}
\end{displaymath}
Hence a necessary condition for $b = b_{\rm{opt}}$ to be a minimizer of $J$ over $\cB_M$ is that
\begin{displaymath}
 b_{\rm{opt}}(t,x) = H_\xi\left (x, \left(\int_\Omega  p_{b_{\rm{opt}}}(t) \right)  Du_{b_{\rm{opt}}}(t,x)\right).
\end{displaymath}
Plugging this control back into the PDEs for $p_{b_{\rm{opt}}}$ and $u_{p_{b_{\rm{opt}}}}$, we obtain the forward-backward KFP-HJB system~\eqref{eq:time-HJB}--\eqref{eq:time-FP}.
\end{proof}
\begin{remark}
	As a consequence of Theorems~\ref{thm:existence-minimizer-J} and~\eqref{thm:time-optcond}, we get the existence of weak solutions for system~\eqref{eq:time-HJB}--\eqref{eq:time-FP} with $p \in C^0(\overline Q_T) \cap L^2(0,T; H^1_0(\Omega))\cap W^{1,2}(0,T;H^{-1}(\Omega))$ and $u \in  C^0([0,T]; L^2(\Omega))\cap L^2(0,T; H^1_0(\Omega))$.
\end{remark}

\begin{remark}
Setting, for $\mu \in (0,+\infty)$,
\begin{equation}
\label{eq:def-checkH}
\widecheck H(x,\mu,\xi)=  \frac { H\left(x, \mu\xi \right) }  {\mu} = \max_{b \in \RR^d;\,  |b|\le M } \left( \xi \cdot b - \frac{L(x,b)}{\mu} \right),  
\end{equation}
the HJB equation~\eqref{eq:time-HJB} can be written as follows:
\begin{equation}
\label{eq:HJB-checkH}
  \begin{split}
 &-\partial_t u -\Delta u +\widecheck H \left(\cdot, \int_\Omega p(t), Du\right) + \int_\Omega  p(t,y) \widecheck H_\mu \left(y, \int_\Omega p(t) ,Du(t,y)\right) dy \mathds{1}_{\Omega} \\
 =   &  \frac {F\left[ \frac {p(t)} {\int_\Omega p(t) } \right]} {\int_\Omega  p(t)} 
  -  \frac {\ds \int_\Omega\left( p(t)    F\left[ \frac {p(t)} {\int_\Omega p(t) } \right]\right)   }  {\left(\int_\Omega  p(t)\right)^2 }  \mathds{1}_{\Omega}
,    
  \end{split}
\end{equation}
because, if $b_{\rm{opt}}$ is optimal for $J$ over $\cB_M$, then
\begin{displaymath}
  \widecheck H_\mu \left(y, \int_\Omega p(t), Du(t,y)\right) = \frac{L(y, b_{\rm{opt}}(t,y))}{ \left(\int_\Omega p(t)\right)^2}  \mathds{1}_{\Omega}.
\end{displaymath}
\end{remark}

\begin{remark}
\label{eq:Lquadra}
	If, for example, $L(x,b) = f(x) + \frac{1}{2}|b|^2$, then the maximizer in~\eqref{eq:def-checkH} is $\mu\xi$ if $\mu\|\xi\| \leq M$ and $M \frac{\xi}{\|\xi\|}$ otherwise. Thus
	$$
		\widecheck H(x,\mu,\xi) = 
		\begin{cases}
			\frac{1}{2} \mu|\xi|^2 - \frac{f(x)}{\mu}, &\hbox{ if } \mu|\xi| \leq M,
			\\
			M |\xi| - \frac{f(x)}{\mu} - \frac{M^2}{2\mu}, &\hbox{ otherwise.}
		\end{cases}
	$$
\end{remark}

\subsection{Heuristics about long time behavior}

 For $T>0$, let $p^{(T)}$ and $u^{(T)}$ denote respectively the solution of the FP equation~\eqref{eq:time-FP} and the HJB equation~\eqref{eq:time-HJB} on the time interval $[0,T]$, and let $b^{(T)}_{\rm{opt}}$ denote the associated optimal control.
Numerical simulations provided in Section~\ref{sec:statio-nonstatio-turnpike} below indicate that,  at least in some cases, the following 
asymptotic behavior   as $T\to\infty$ may be expected:
for some $\eta \in (0,1/2)$,  for all $t$ s.t. $t/T \in (\eta, 1-\eta)$, 
\begin{displaymath}
  \begin{array}[c]{llcl}
& \ds \frac {p^{(T)}(t,x)}{\int_\Omega p^{(T)}(t,y) dy }&\to& \tilde p(x), \\ \\
& \ds \left( \int_\Omega p^{(T)}(t,y) dy\right)  u^{(T)}(t,x) &\to& \tilde u(x),\\ \\
&\ds \frac 1 {\int_\Omega p^{(T)}(t,y) dy }  \frac d {d t}  \left(\int_\Omega p^{(T)}(t,y) dy\right) & \to  & -\lambda, \\ \\
& \ds  b^{(T)}_{\rm{opt}}(t,x) & \to & \tilde b(x).
  \end{array}
\end{displaymath}

Notice that, if we multiply respectively the HJB equation~\eqref{eq:time-HJB} by $p$ and the Fokker-Planck equation~\eqref{eq:time-FP} by $u$, we integrate in $(t,T)\times \Omega$, and we sum the resulting equations, then, using the definitions of $c_1(t)$ and $c_2(T)$, all the terms but one cancel, and we obtain 
\begin{equation*}
	\int_\Omega u(t,x)p(t,x) dx = -\epsilon, \quad \hbox{for all } t\in [0,T].
\end{equation*}

Plugging the ansatz into the FP equation~\eqref{eq:time-FP}, we obtain:
\begin{equation}
  \label{eq:S-FP} 
  \left\{
    \begin{array}[c]{rcll}
 -\frac{\sigma^2 } 2 \Delta  \tilde p -\diver\left( \tilde p    H_\xi\left (\cdot ,  D\tilde u \right)   \right)   &=&\lambda \tilde p,    \quad  \quad  \quad & \hbox{ in } \Omega,\\
\tilde p&=&0    \quad  & \hbox{ on } \partial \Omega,\\
 \tilde p &\ge& 0   \quad &\hbox{ in } \Omega,\\
\ds \int_\Omega \tilde p &=&1.&
    \end{array}
\right.
\end{equation}

This means that $\lambda$ is the principal eigenvalue of the Dirichet problem associated with the stationary FP equation with drift  $-b=-H_\xi\left (\cdot,   D\tilde u\right)$.

Moreover, plugging the ansatz in the HJB equation~\eqref{eq:time-HJB}, we find:
\begin{equation}
\label{eq:S-HJB}
 \left\{
    \begin{array}[c]{rcl}
 -\frac{\sigma^2 } 2 \Delta  \tilde u + H(\cdot, D\tilde u)
 &=&  \ds  \lambda \tilde u +F[\tilde p]    +c_1 \mathds{1}_{\Omega} , \hfill \hbox{in } \Omega,\\
\tilde u&=&0 \hfill \hbox{on }\partial \Omega,
\\
\ds \int_\Omega \tilde u(x)\tilde p(x) dx & = &  -\epsilon,\\
  c_1 &  = & -   \left( \int_\Omega \tilde p   \left( L\left (\cdot,  H_\xi\left (\cdot,   D\tilde u  \right)  \right)  +F[\tilde p]\right) dx\right).
    \end{array}
    \right.
  \end{equation}
We show in the sequel that \eqref{eq:S-FP}--\eqref{eq:S-HJB} are the necessary optimality conditions of an optimization problem driven by an eigenvalue problem. This will yield existence for  \eqref{eq:S-FP}--\eqref{eq:S-HJB}.

 In  \cite{PLL}, the third author  discussed some mathematical tools that may be used 
in order to prove that the ansatz given above actually describes the asymptotic behavior as $T\to \infty$, such as suitable generalizations to the notion of ergodic problems and to principal eigenvalue problems. However, we stress the fact that, to the best of our knowledge, proving rigorously the ansatz is an open problem. The main reason for that is that other asymptotic behaviors, involving for example time-periodic solutions, may be possible.

\section{Facts about a principal eigenvalue problem}
\label{sec:facts-about-princ}

\subsection{The principal eigenvalue related to an elliptic equation: known facts}\label{sec:princ-eigenv-relat-4}
We recall that $\Omega$ is a bounded domain of $\RR^d$ with a smooth boundary. Let $b$ be a vector field: $b\in L^\infty(\Omega; \RR^d)$.
Recall that the elliptic operator $\cL_b$ is defined by~\eqref{eq:def_Lb}. The weak maximum principle holds for $\cL_b$, see \cite[Theorem 8.1]{MR1814364}:
{\sl
If $v\in H^1(\Omega)$ is a weak subsolution of the Dirichlet problem 
\begin{displaymath}
  \begin{split}
    \cL_b v &=0\quad \hbox{ in }\Omega\\
    v &=0\quad \hbox{ on }\partial\Omega\
  \end{split}
\end{displaymath}
i.e. is such that 
\begin{displaymath}
\int_\Omega Dv\cdot Dw +  w b\cdot Dv  \le 0,
\end{displaymath}
for all  $w\in H^{1}_0(\Omega)$ such that $w\ge 0$ almost everywhere in $\Omega$, and if $v^+ \in H^1_0(\Omega)$, then $v\le 0$ 
almost everywhere in $\Omega$.}

It is also well known, see \cite[Theorem 8.3]{MR1814364}, that  for any $g\in L^2(\Omega)$, there exists a unique weak solution of the Dirichlet problem:
\begin{eqnarray}
\label{eq:1}
 \cL_b u = g \quad \hbox{ in }\Omega,\\
\label{eq:2}
u=0\quad \hbox{ on }\partial \Omega,
\end{eqnarray}
 i.e. $u\in H^1_0(\Omega)$  such that for all $w\in H^1_0(\Omega)$,
\begin{displaymath}
\int_\Omega Du\cdot Dw +  w b\cdot Du  = \int_\Omega g w.
\end{displaymath}
Moreover, $u\in H^2 (\Omega)\cap H^1_0(\Omega)$, see for example \cite[\S 6.3, Theorem 4]{MR2597943} and  the operator $ g\mapsto u$ is bounded from $L^2(\Omega)$ to $H^2 (\Omega)\cap H^1_0(\Omega)$. The operator $\bfS: L^2(\Omega)\to L^2(\Omega)$ defined by $\bfS(g)=u$ is compact.  Let $L^2_+(\Omega)$ be the closed cone made of the nonnegative functions in $L^2(\Omega)$: from the weak maximum principle, we know that  if $ g\in L^2_+(\Omega) $, then $\bfS g\in L^2_+(\Omega)$.
\\
From Agmon-Douglis-Nirenberg theorems, see \cite{MR0125307} or \cite[Theorem 9.15]{MR1814364}, we know that if  $g\in L^p(\Omega)$, for $2<p<\infty$, then $u\in W^{2,p}(\Omega)\cap W^{1,p}_0(\Omega)$ and the operator  $g\mapsto u$ is bounded from $L^p(\Omega)$ to $W^{2,p} (\Omega)\cap W^{1,p}_0(\Omega)$.
 \\
Let $C_0(\overline\Omega)$ be the space of  real valued continuous functions  defined on $\overline \Omega$ and vanishing on $\partial \Omega$. For any $0<\gamma<1$ that we fix, it will be useful to define the Banach space  $\FF= C^{1, \gamma}(\overline \Omega)\cap C_0(\overline \Omega)$ and the cone $\FF_+= \{ v\in \FF: v\ge 0 \in \overline \Omega\}$. Using the latter regularity results, we know that if  $g \in C(\overline \Omega)$, then 
 $u\in W^{2,p}(\Omega)\cap W^{1,p}_0(\Omega)$ for all $ p<\infty$.  Sobolev embeddings of $W^{2,p}(\Omega)\cap W^{1,p}_0(\Omega)$ in  $\FF$ for $p$ large enough imply that the linear operator  $g\mapsto u$ is compact from $ C(\overline \Omega)$ to $\FF$. Let $\bfT$ be the restriction of the latter operator to $\FF$. 
\\
Finally, as a consequence of Hopf lemma, which may be applied since $b\in L^\infty(\Omega;\RR^d)$, we know that if $0\not\equiv g\ge 0$ in $\overline\Omega$, then $\bfT g\in \rm{Int}(\FF_+)$. 
\\
The previously mentioned facts allowes us to state the following theorem:
\begin{theorem}
\label{sec:princ-eigenv-relat-5}
There exists a positive  real number $\lambda_b>0$, such that
\begin{enumerate}
\item $1/\lambda_b$ is an eigenvalue of $\bfT$
\item  The related eigenspace is one-dimensional and can be written $\RR u$ where $u\in \FF_+$, $u>0$ in $\Omega$ 
\item For all complex number $\mu$ such that $\mu \bfT(v) = v$, for some non identically zero function $v$ whose real and imaginary part belong to $\FF$, 
  \begin{displaymath}
    \rm{Re}(\mu)\ge \lambda_b.
  \end{displaymath}
\end{enumerate}
\end{theorem}

\begin{proof}
  When $b\in C^{\infty}(\bar \Omega;\RR^d)$, Theorem \ref{sec:princ-eigenv-relat-5} is exactly  \cite[Theorem 3, page 340]{MR2597943}. However, in view of the facts that we have recalled above, the proof of  \cite[Theorem 3, page 340]{MR2597943},
which only requires the compactness of $\bfT$, the strong maximum principle and Hopf lemma, can be repeated in the case when $b\in L^\infty(\Omega;\RR^d)$.
\end{proof}
\begin{remark}\label{sec:princ-eigenv-relat-6}
 Theorem \ref{sec:princ-eigenv-relat-5} implies that the spectral radius of $\bfT$ is positive and that it coincides with $1/\lambda_b$.
\end{remark}

The positive number $\lambda_b$  is called the principal eigenvalue related to the Dirichlet problems (\ref{eq:1})-(\ref{eq:2}). Setting $e_b= u /\int_\Omega u dx$, we see that  $e_b\in \FF_+$ is the unique  weak solution of 
\begin{eqnarray}
\label{eq:3}
  -\Delta e_b + b\cdot De_b = \lambda_b e_b \quad \hbox{ in }\Omega,\\
\label{eq:4}
e_b=0\quad \hbox{ on }\partial \Omega,\\
\label{eq:5}
\int_\Omega e_b=1,
\end{eqnarray}
 and $e_b>0$ in $\Omega$.

If $\mu\not =0$ is an eigenvalue of $\bfS$ and $w$ is any related eigenvector,
 then the regularity results for the Dirichlet problem
\begin{eqnarray*}
  \cL_b w = \frac 1 \mu  w \quad \hbox{ in }\Omega,\\
w=0\quad \hbox{ on }\partial \Omega,
\end{eqnarray*}
imply that $w\in \FF$, see for example \cite[Theorem 9.15]{MR1814364}. Hence, the  spectrum and the eigenspaces of $\bfS$ and $\bfT$ coincide.\\

In \cite{MR1258192}, Berestycki, Nirenberg and Varadhan propose a definition of the principal eigenvalue of $\cL_b $ which can be applied even if $\partial \Omega$ is not smooth: 
\begin{equation}
\label{eq:6}
\lambda_b=\sup\{\lambda : \exists  \phi\in W^{2,d}_{\rm{loc}}(\Omega),\;  \phi >0  \hbox{ in } \Omega, \hbox{ such that }   \cL_b \phi -\lambda\phi \ge 0   \hbox{ in } \Omega \}.
\end{equation}
 As stated in \cite{MR1258192} (see also \cite{MR1226957}), if $\partial \Omega$ is smooth, then the definition in (\ref{eq:6}) is equivalent to the former definition 
of the principal eigenvalue of $\cL_b $. Hence, if $\partial \Omega$ is smooth, (\ref{eq:6}) is a characterization of the principal eigenvalue of $\cL_b $. Note that  no restriction is made in (\ref{eq:6}) on the behavior of $\phi$ near $\partial \Omega$.
\\
 The main interest of (\ref{eq:6}) is that it is a variational characterization of the principal eigenvalue
(recall that the standard characterization for self-adjoint operators is irrelevant in the present context):
 indeed, as proved in \cite{MR1258192}, 
 \begin{equation}
   \label{eq:7}
\lambda_b=\sup_{\phi\in W_{\rm{loc}}^{2,d}(\Omega)}
 \inf_\Omega \frac {\cL_b \phi}\phi.
 \end{equation}

Note that in \cite{MR0425380}, Donsker and Varadhan gave a slightly different variational formula for the principal value of $\cL_b $ and 
proved (\ref{eq:7}) in the case when $b$ is continuous and $\partial \Omega $ is smooth.

As a consequence of (\ref{eq:6}), it is proven in \cite{MR1258192} that for any domain $\Omega'\subset \Omega$, $\lambda_b(\Omega)\le \lambda_b(\Omega')$ and that the inequality is strict if the set inclusion is strict (see Theorem~\ref{sec:princ-eigenv-relat-3} below, which is in fact a much stronger result).

An important consequence of (\ref{eq:7}), of the latter observation and of the maximum principle is an upper bound on $\lambda_b$:
\begin{lemma} [Lemma 1.1. in \cite{MR1258192}]\label{sec:princ-eigenv-relat-1}
 Suppose that $\Omega$ contains a ball of radius $R\le 1$. Then
\begin{displaymath}
 0< \lambda_b\le \frac C{R ^2 }
\end{displaymath}
where $C$ depends only on $d$ and $\|b\|_{L^\infty(\Omega;\RR^d)}$.
\end{lemma}
The dependence of $\lambda_b$ with respect to $b$ is also addressed in \cite{MR1258192}:
\begin{proposition}[Proposition 5.1 in \cite{MR1258192}]\label{sec:princ-eigenv-relat-2}
  \label{sec:princ-eigenv-relat}
The map $b\mapsto \lambda_b$  is  locally Lipschitz continuous from $L^\infty(\Omega;\RR^d)$ to $\RR$. The Lipschitz constant 
of the restriction of the latter map to $\cB_M = \{ b: \|b\|_{L^\infty(\Omega;\RR^d)} \le M\}$ can be chosen to depend only on $\Omega$ and $M$.
\end{proposition}

\begin{theorem}[Theorem 2.4 in \cite{MR1258192}] \label{sec:princ-eigenv-relat-3}
  Consider a bounded domain $\Omega$ of $\RR^d$, $b\in L^\infty(\Omega; \RR^d)$ and a compact set $K$ contained in $\Omega$  such that 
$|K|\ge \delta_0>0$ (where $|K|$ denotes the Lebesgue measure of $K$). Let $\lambda_b(\Omega)$ and $\lambda_b(\Omega\setminus K)$ be the principal eigenvalues associated respectively with $\Omega$ and $\Omega\setminus K$.
Then
\begin{displaymath}
  \lambda_b(\Omega\setminus K) -\lambda_b(\Omega)\ge \bar \delta>0,
\end{displaymath}
where $\bar \delta$ depends only on $\Omega$, $\|b\|_{L^\infty(\Omega;\RR^d)}$ and $\delta_0$.
\end{theorem}

\subsection{Some spectral properties of the adjoint operator}

From Fredholm alternative, see \cite[Theorem 5.11]{MR1814364}, we know that $\bfS$ and $\bfS^*$ have the same spectra, and that if $\mu \in \CC$,
\begin{displaymath}
  \dim{(\ker( I -\mu \bfS))}=\dim{(\ker( I -\bar \mu \bfS^*))}.
\end{displaymath}
 Therefore $1/\lambda_b$ is the spectral radius of $\bfS^*$, and $\lambda_b$ can be seen as the principal eigenvalue related
 to the dual boundary value problem with a Fokker-Planck equation: there exists a unique $p_b\in H^{1}_0(\Omega)$ such that
 \begin{eqnarray}
\label{eq:8}
  \cL_b^* p_b = \lambda_b p_b \quad \hbox{ in }\Omega,\\
\label{eq:9}
p_b=0\quad \hbox{ on }\partial \Omega,\\
\label{eq:10}
\int_\Omega p_b=1,
 \end{eqnarray}
 where $\cL_b^*$ is the adjoint elliptic operator defined by~\eqref{eq:def_Lb_star}.
Moreover, from 
 Krein-Rutman theorem (see~\cite{MR0038008} or e.g.~\cite[Theorem 1]{MR2205529}   applied to $\bfS$, the eigenvector $p_b$  belongs to $L^2_+(\Omega)$.
\begin{proposition}\label{sec:stab-princ-eigenv}
  The normalized eigenvector $p_b$ characterized by (\ref{eq:8})--(\ref{eq:10}) is positive in $\Omega$ and continuous in $\bar \Omega$.
\end{proposition}
\begin{proof}
For $b\in L^\infty(\Omega; \RR^d)$, we now consider the Dirichlet problem related to the adjoint elliptic operator $\cL_b^*$
\begin{eqnarray}
\label{eq:59}
  \cL_b^* q = g \quad \hbox{ in }\Omega,\\
\label{eq:60}
q=0\quad \hbox{ on }\partial \Omega.
\end{eqnarray}
We say that $\cL_b$ (respectively $\cL_b^*$) satisfies  property \textbf{(M)} if for every $v\in H^1_0(\Omega)$, the fact that $\langle \cL_b v, w\rangle$ is $\ge 0$   (respectively  $\langle  \cL_b^* v, w\rangle \ge 0$)
 for all nonnegative functions $w\in H^1_0(\Omega)$, implies that $v$ is nonnegative in $\Omega$. Property \textbf{(M)} for $\cL_b$ is a particular version of the weak maximum principle, so we have seen in
paragraph \ref{sec:princ-eigenv-relat-4} that it is true. On the other hand, since $\Omega$ has a smooth boundary,  using Theorem 2.2.1. in \cite[page 60]{MR3443169}, we see  that property \textbf{(M)} is true for $\cL_b^*$ if and only if it is true for $\cL_b$.
Hence, $\cL_b^*$ enjoys property \textbf{(M)}. Then, from Proposition 2.1.4. in  \cite[page 57]{MR3443169}, we know that
 for any $g\in H^{-1}(\Omega)$, there exists a unique weak solution of the Dirichlet problem (\ref{eq:59})--(\ref{eq:60}), i.e. a unique function $q\in  H^1_0(\Omega)$  such that for all $w\in H^1_0(\Omega)$,
\begin{displaymath}
\int_\Omega Dq\cdot Dw +  q b\cdot Dw  = \langle g, w\rangle.
\end{displaymath}
         
Let us recall H{\"o}lder regularity results for the Dirichlet problem (\ref{eq:59})--(\ref{eq:60}). First, using Theorem 8.15 in \cite{MR1814364}, we know that  for $s>d/2$, there exists a constant $C$ which depends on $d$, $\Omega$, $\|b\|_{\infty}$ and $s$ such that  if $g\in L^s(\Omega)$, then $q\in L^\infty(\Omega)$ and
\begin{displaymath}
  \|q\|_{L^\infty(\Omega)}\le C \|g\|_{L^s(\Omega)}.
\end{displaymath}
Next, from Theorem 8.29 in  \cite{MR1814364} and the latter $L^\infty$ bound, we see that
if $g\in L^s(\Omega)$ for $s>d/2$, then there exists $\alpha\in (0,1)$ such that $p\in C^\alpha(\bar \Omega)$ and
\begin{displaymath}
\|q\|_{C^\alpha(\bar \Omega)}\le C \|g\|_{L^s(\Omega)},
\end{displaymath}
for a given constant $C$. 

Let us  now focus on the case when   for $s>d/2$, the function $g$ belongs to $ L^s(\Omega)$  and is nonnegative.
 The solution $q$ of (\ref{eq:59})--(\ref{eq:60}) is  nonnegative in $\Omega$ by property \textbf{(M)}, and continuous. It is a weak supersolution of $\cL_b^* q=0$ in $\Omega$. Hence, we can apply the weak Harnack inequality stated in Theorem 8.18 in  \cite{MR1814364}: for $1\le t< d/(d-2)$, for all $y\in \Omega$ and $R>0$ such that $B_{4R}(y)\subset \Omega$, 
 \begin{displaymath}
   R^{-d/t} \|q\|_{L^t(B_{2R}(y))}\le C \inf_{B_R(y) } q,
 \end{displaymath}
for a constant $C$ which depends on $d$, $\|b\|_{\infty}$, $R$, $s$, and $t$.
This implies that if $q$ vanishes in $\Omega$, then $q$ must be identically $0$ in $\Omega$. Hence, if $g$ is non identically $0$, $q$ must be positive in $\Omega$.

 Fix one of the H{\"o}lder exponents $\alpha$ obtained above  and define the Banach space $\GG=C^\alpha_0(\Omega)$ 
and the cone $\GG_+=\{g\in \GG; g\ge 0\in \Omega\}$. The linear operator $\bfU: g\mapsto q$ is compact from $\GG$ to $\GG$.
Krein-Rutman theorem (the first version stated in paragraph \ref{sec:princ-eigenv-relat-4}) can be applied to $\bfU$, which yields that there exists an eigenvector of $\bfU$ in $\GG_+$ associated with the spectral radius of $\bfU$. Since the eigenvectors of $\bfU$ are eigenvectors of $\bfS$, the latter must coincide with $p_b$. Hence $p_b\in C(\bar \Omega)$.

Finally, we can apply Harnack inequality to $p_b$: $p_b$ is positive in $\Omega$.

\end{proof}

\subsection{A stability result}
\label{sec:anoth-stab-result}
\begin{lemma}\label{sec:stability-1}
  Let the sequence $b_n$ tend to $b$ in $L^\infty(\Omega;\RR^d)$ weak *. Let $\lambda_n$ be the principal eigenvalue of $\cL_n$ (related to $b_n$). Let $p_{b_n}$ be the principal eigenvector of $\cL _n^*$ (non negative and normalized in $L^1$ norm). 
The following convergence results hold: $\lambda_n$ tends to $\lambda_b$ and $p_{b_n}$ tends to $p_b$  in $L^2(\Omega)$ strongly and 
in $H^1 _0(\Omega)$ weakly as $n\to \infty$.
\end{lemma}
\begin{proof}
We consider  $p_n$ the eigenvector of $\cL _n^*$ related to $\lambda_n$ and  normalized in $L^2 (\Omega)$  norm: it solves
 \begin{eqnarray}
\label{eq:12}
  -\Delta p_n - \diver(  p_n b_n) = \lambda_n p_n \quad \hbox{ in }\Omega,\\
\label{eq:13}
p_n=0\quad \hbox{ on }\partial \Omega,\\ \label{eq:14}
\int_\Omega p_n^2=1.
 \end{eqnarray}
We know that $p_n> 0$ in $\Omega$. We know that $p_{b_n}= p_n/\int_\Omega p_n$.
From Lemma \ref{sec:princ-eigenv-relat-1}, we know that $0<\lambda_n$ is bounded from above by a constant independent of $n$.
From this observation, it is easy to see that $p_n$ is bounded in $H^1_0(\Omega)$ by a constant independent of $n$.
Extracting a subsequence, we may assume that $p_n$ tends to $q$ weakly in $H^1_0(\Omega)$, strongly in $L^2(\Omega)$  and that $\lambda_n$ tends to $\mu$ in $\RR$.

 We may pass to the limit and we get that
 \begin{eqnarray*}
  -\Delta q - \diver(  q b) = \mu q \quad \hbox{ in }\Omega,\\
q=0\quad \hbox{ on }\partial \Omega,\\ 
\int_\Omega q^2 =1,\\
 q\ge 0,\quad \hbox{in }\Omega.
 \end{eqnarray*}
We also see that $\int_\Omega p_n$ tends to $\int_\Omega q$. The latter cannot be zero, otherwise $q$ would be identically $0$ in contradiction with $\int_\Omega q^2 =1$. Hence $\int_\Omega p_n$ is bounded from below by a positive constant. Therefore 
$p_{b_n}= p_n/\int p_n \to  \tilde q= q/\int_\Omega q$  weakly in $H^1_0(\Omega)$ and strongly in $L^2(\Omega)$.
\\
 If $\mu\not=\lambda_b$, then, using $e_b$ as a test function in the latter boundary value problem, we get $\int_\Omega q e_b =0$ which is not possible since $0\not\equiv q\ge 0$ and $e_b>0$.
Therefore $\mu=\lambda_b$ and $\tilde q=p_b$. 
   Since the two cluster points are unique,  the whole sequences $(\lambda_n)$ and $(p_{b_n})$ tend to $\lambda_b$ and $p_b$ respectively.
\end{proof}

\subsection{Dependence of $p_b$ with respect to $b$}
\begin{lemma}\label{sec:stability-2}
The map $b\mapsto p_b$ is locally Lipschitz continuous from $L^\infty(\Omega; \RR^d)$ to $H^1_0(\Omega)$.
\end{lemma}

\begin{proof}
First, using a similar argument as in the proof of Lemma~\ref{sec:stability-1}, we see that
the map $b\mapsto p_b$ is locally bounded from $L^\infty(\Omega; \RR^d)$ to $H^1_0(\Omega)$.

Take $M>0$ and two vector fields $b_1$ and $b_2$ 
such that $\|b_1\|_{L^\infty(\Omega;\RR^d)}\le M$ and $\|b_2\|_{L^\infty(\Omega;\RR^d)}\le M$. 
Let us set $\delta b=b_1-b_2$,  $\delta p=p_{b_1}-p_{b_2}$, $\delta \lambda= \lambda_{b_1}  -\lambda_{b_2}$:
   \begin{eqnarray}
\label{eq:15}
  -\Delta \delta p - \diver(  b_1 \delta p) = \lambda_{b_1} \delta p     + \diver(  p_{b_2} \delta b)+\delta \lambda p_{b_2} 
    \quad \hbox{ in }\Omega,\\
\label{eq:16}
\delta p=0\quad \hbox{ on }\partial \Omega,\\
\label{eq:17}
\int_\Omega \delta p=0.
 \end{eqnarray}
Testing (\ref{eq:15}) by $e_{b_1}$, we see that
\begin{equation}
  \label{eq:18}
   \langle \diver(  p_{b_2} \delta b)+\delta \lambda p_{b_2} , e_{b_1} \rangle_{H^{-1}(\Omega), H^{1}_0(\Omega)} = -\int_\Omega     p_{b_2} \delta b  \cdot D e_{b_1} dx + \delta \lambda \int_\Omega 
p_{b_2} e_{b_1} dx =0.
\end{equation}
On the other hand, from  Fredholm's theory, we know that for all $g\in  H^{-1}(\Omega)$ such that $\langle g, e_{b_1}\rangle _{H^{-1}(\Omega), H^{1}_0(\Omega)}=0$,   there exists $z\in H^1_0(\Omega)$ a weak solution of  
\begin{equation}
\label{eq:19}
  -\Delta z - \diver(  b_1 z) -\lambda_{b_1} z   = g
    \quad \hbox{ in }\Omega,
\end{equation}
and that  $z$ is unique up to the addition of a multiple of $p_{b_1}$. Moreover, the operator $ A: g\mapsto z$ is bounded from $\{g\in  H^{-1}(\Omega),\;\langle g, e_{b_1}\rangle _{H^{-1}(\Omega), H^{1}_0(\Omega)}=0 \}$ to $H^{1}_0(\Omega) /  (\RR p_{b_1})$.
Therefore, for all $g\in  H^{-1}(\Omega)$ such that $\langle g, e_{b_1}\rangle _{H^{-1}(\Omega), H^{1}_0(\Omega)}=0$, there exists a weak solution $\tilde z \in H^{1}_0(\Omega)$ of (\ref{eq:19}) such that $\|\tilde z\|_{H^{1}_0(\Omega)}\le 2 \|A\|\, \|g\|_{ H^{-1}(\Omega)}$.
Then  $q\equiv \tilde z - (\int_\Omega \tilde z) p_{b_1}\in H^1_0(\Omega)$ is the unique weak solution of (\ref{eq:19}) such that $\int_\Omega q=0$. We deduce that $\|q\|_{H_0^1 (\Omega)}\le C  \|g\|_{ H^{-1}(\Omega)}$ for a positive constant $C$ (independent of $g$). 

In what follows, the positive constant $C$ (independent of $g$) will vary from one line to the other. From (\ref{eq:15})--(\ref{eq:17}),  (\ref{eq:18}) and the previous arguments, we deduce  that
\begin{displaymath}
  \begin{split}
\|\delta p\|_{H_0^1 (\Omega)}& \le C  \|\diver(  p_{b_2} \delta b)+\delta \lambda p_{b_2}  \|_{H^{-1}(\Omega)}      \\
& \le C \left(        \| p_{b_2} \delta b  \|_{L^2 (\Omega; \RR^d)} + |\delta \lambda| \|p_{b_2}\|_{L^2 (\Omega)} \right)\\
& \le C \left(   \|  \delta b  \|_{L^\infty (\Omega; \RR^d)} + |\delta \lambda|  \right)\\
&\le C \|  \delta b  \|_{L^\infty (\Omega; \RR^d)}.
  \end{split}
\end{displaymath}
The third line is obtained by using the bound on $\|p_{b_2}\|_{L^2(\Omega)}$ discussed in the very first lines of the proof. 
The last line is obtained by using Proposition \ref{sec:princ-eigenv-relat}.
\end{proof}

\section{Long time behavior: an optimal control problem driven by a principal eigenvalue problem}
\label{sec:long-time-behavior}

\subsection{Definition of the problem}

Consider the following problem 
\begin{equation}
\label{eq:20}
\inf_{\tilde b \in L^\infty(\Omega;\RR^d); \|\tilde b \|_{\infty}\le M} \widetilde J(\tilde b )
\end{equation}
where
\begin{equation}
  \label{eq:56}
 \widetilde J(\tilde b )= \int_{\Omega} p_{\tilde b}(x)L(x,\tilde b (x))  dx + \epsilon \lambda_{\tilde b } + \Phi[p_{\tilde b }],
\end{equation}
and
\begin{eqnarray}
\label{eq:21}
	\lambda_{\tilde b} \hbox{ is the principal eigenvalue related to ~(\ref{eq:3})--(\ref{eq:5})},\\
\label{eq:22}
p_{\tilde b } \hbox{ is the unique solution of (\ref{eq:8})--(\ref{eq:10})}.
\end{eqnarray}
It will be convenient to use the following notations for the set of controls:
$
	 \widetilde \cB = L^\infty(\RR^{d}; \RR^d),
$
and for every $M>0$, $ \widetilde \cB_{M} = \{\tilde b  \in L^\infty(\RR^{d}; \RR^d) \,:\, \|\tilde b \|_{L^\infty} \leq M   \}$ is the subset of controls uniformly bounded by $M$.

\subsection{Existence of an optimal solution}
\label{sec:existence-an-optimal}

\begin{lemma}
  For every $M>0$, there exists a minimizer of \eqref{eq:20} subject to \eqref{eq:21}--\eqref{eq:22} over $\widetilde \cB_M$.
\end{lemma}
\begin{proof}
  
We consider a minimizing sequence $(\tilde b _m)_{m>0}$, $\|\tilde b _m\|_{\infty}\le M$ and $\cL_m v = -\Delta v+ \tilde b _m \cdot Dv$. Let $\lambda_{\tilde b _m}$ 
 be the principal eigenvalue of $\cL_m$ (with Dirichlet conditions)
and let $p_{\tilde b _m}$ be   the unique solution of (\ref{eq:8})--(\ref{eq:10}) where $\tilde b $ is replaced by  $\tilde b _m$. 

We can extract a subsequence, still indexed by $m$, such that $\tilde b _m \to \tilde b $ in $L^\infty(\Omega;\RR^d)$ weakly~*.
Then from Lemma \ref{sec:stability-1}, $\lambda_{\tilde b _m}$ tends to $\lambda_{\tilde b }>0$ and $p_{\tilde b _m}$ tends to $p_{\tilde b }$ in $L^2 (\Omega)$ strongly and in 
$H^1_0(\Omega)$ weakly.

Moreover since $L$ is convex and continuous in its second argument and since $\tilde b _m\rightharpoonup \tilde b $ in $L^\infty(\Omega;\RR^d)$ weak * and  $p_{\tilde b _m}$ converges to $p_{\tilde b }$ in $L^2(\Omega)$, we know that
\begin{displaymath}
 \int_{\Omega} p_{\tilde b }(x)L(x,\tilde b (x)) dx\le \liminf_{m\to \infty} 
\int_{\Omega} p_{\tilde b _m}(x)L(x,\tilde b _m(x)) dx  
\end{displaymath}
 and $\Phi[p_{\tilde b }]=\lim_{m\to\infty} \Phi[p_{\tilde b _m}]$.
 Therefore, $\tilde b $ is a minimizer of
 (\ref{eq:20}) subject to (\ref{eq:21})--(\ref{eq:22}).
\end{proof}

\subsection{Necessary optimality conditions}
\label{sec:necess-ptim-cond}
We now derive a necessary optimality condition for the above problem, as we did in Theorem~\ref{thm:time-optcond} in the finite time horizon problem.
\begin{theorem}
\label{thm:ergo-optcond}
	Let $M>0$. If $\tilde b _{\rm{opt}} \in \widetilde \cB_M$ is a minimizer of $\widetilde J$ over $\widetilde \cB_M$, then necessarily, for every $x \in \RR^d$, $\tilde b _{\rm{opt}}(x)$ achieves the maximum in~\eqref{eq:def-H} with $\xi$ replaced by $Du(x)$
	where $(p,u)$ solves the PDE system~\eqref{eq:S-FP}--\eqref{eq:S-HJB}. The necessary optimality condition can be written
	$$
		\tilde b _{\rm{opt}}(x) = H_\xi\left (x, Du(x)\right).
	$$
\end{theorem}
\begin{proof}

Assume that $\tilde b = \tilde b _{\rm{opt}}$ is a  minimizer of (\ref{eq:20}) subject to (\ref{eq:21})--(\ref{eq:22}).
Let $\delta \tilde b \in L^\infty( \Omega;\RR^d)$ be a small variation of $\tilde b $, and $\delta \lambda_{\tilde b }$, $\delta p_{\tilde b }$ be the corresponding variations of $\lambda_{\tilde b }$ and $p_{\tilde b }$. We have already seen that 
 \begin{eqnarray}
\label{eq:23}
  -\Delta \delta p_{\tilde b } - \diver(  \delta p_{\tilde b } \tilde b ) =  \delta\lambda_{\tilde b } p_{\tilde b } + \lambda_{\tilde b } \delta p_{\tilde b }  + \diver(  p_{\tilde b } \delta \tilde b )  \quad \hbox{ in }\Omega,\\
\label{eq:24}
\delta p_{\tilde b }=0\quad \hbox{ on }\partial \Omega,\\
\label{eq:25}
\int_\Omega \delta p_{\tilde b }=0.
 \end{eqnarray}
Up to higher degree terms which can be neglected in view of Proposition \ref{sec:princ-eigenv-relat} and Lemma \ref{sec:stability-2}, 
 the variation of the cost is
\begin{equation}
\label{eq:26}
  \begin{split}
\delta \widetilde J  = &  \int_{\Omega} \delta p_{\tilde b } (x) L(x,\tilde b (x))  dx  +  \int_{\Omega} p_{\tilde b } (x)  L_{\tilde b }(x, \tilde b (x))  \cdot \delta \tilde b (x) dx \\ &  + \epsilon \delta \lambda_{\tilde b }  +  \int_{\Omega} F[p_{\tilde b }](x) \delta p_{\tilde b }(x) dx.
\end{split}
\end{equation}
Let us consider the adjoint problem: find $u\in H^{1}_0(\Omega)$ such that
\begin{equation}
\label{eq:27}
\begin{split}
   &\int_{\Omega} \left(Du \cdot D q +    (\tilde b 
\cdot D u) q -  \lambda_{\tilde b } u q\right) dx \\
 =&  \int_{\Omega} q(x)(  L(x,\tilde b (x))  + F[p_{\tilde b }](x) ) dx  \\
&+ \left( - \int_{\Omega} p_{\tilde b }(x) \left( L(x,\tilde b (x))  + F[p_{\tilde b }](x) \right) dx \right)\int_\Omega q(x) dx ,
\end{split}
 \end{equation}
for all $q\in H^1_0(\Omega)$.
We see that the right hand side of (\ref{eq:27}) vanishes for $q=p_{\tilde b }$. From Fredholm's theory, this implies the existence of a solution $u$ and its uniqueness up to the addition of a multiple of $e_{\tilde b }$.

Hence, at leading order,
\begin{displaymath}
  \delta \widetilde J=  \int_{\Omega} (Du \cdot D \delta p_{\tilde b } +  \delta p_{\tilde b }  \tilde b \cdot Du  -  \lambda_{\tilde b } u \delta p_{\tilde b }) dx   +  \int_{\Omega} p_{\tilde b }  L_{\tilde b }(x,\tilde b ) \cdot \delta \tilde b \, dx    + \epsilon \delta \lambda_{\tilde b } .
\end{displaymath}
Since $\int_\Omega p_{\tilde b } e_{\tilde b } >0$, there exists a unique $u_{\tilde b}$ satisfying (\ref{eq:27}) and
\begin{equation}
\label{eq:28}
 \int_{\Omega}  p_{\tilde b } u_{\tilde b} dx = -\epsilon.
\end{equation}
For this choice of $u_{\tilde b}$, using (\ref{eq:23})--(\ref{eq:25}) yields
\begin{equation}
\label{eq:29}
  \delta \widetilde J=\int_{\Omega}  \delta \tilde b (x) p_{\tilde b }(x)  \cdot \left(   L_{\tilde b }(x,\tilde b (x)) - Du_{\tilde b}(x) \right) dx,
\end{equation}
and since $p_{\tilde b }>0$ in $\Omega$ from Proposition \ref{sec:stab-princ-eigenv}, the first order optimality condition  is that almost everywhere, $\delta \tilde b (x)  \cdot \left( L_{\tilde b }(x,\tilde b (x))-Du_{\tilde b}(x) \right)$ for all 
admissible variation $\delta \tilde b $ of $\tilde b $. 

Let us recall that the Hamiltonian $H$ is defined by~\eqref{eq:def-H}. 

The first order optimality condition is equivalent to saying that almost everywhere in $\Omega$,  $\tilde b (x)$ achieves the maximum in
\begin{displaymath}
	H(x,Du_{\tilde b}(x))= \max_{\eta \in \RR^d;\,  |\eta|\le M} \left(  Du_{\tilde b}(x) \cdot \eta - L(x,\eta) \right),  
\end{displaymath}
i.e., assuming that $H$ is differentiable with respect to its second argument,
\begin{equation}
  \label{eq:30}
	\tilde b (x)= H_\xi(x,Du_{\tilde b}(x)).
\end{equation}

We thus obtain the PDE system~\eqref{eq:S-FP}--\eqref{eq:S-HJB}.
\end{proof}

\section{Back to the finite horizon problem, letting $M \to \infty$ }
\label{sec:M-to-infty}
For simplicity only, we restrict ourselves to the following minimization problem
\begin{equation}
\label{eq:11}
	\text{ minimize  $J$ on $\cB_M$, } 
\end{equation}
where 
\begin{align}
\label{eq:38}
	J(b) 
	&= \int_0^T\left(\frac{\int_\Omega p_b(t,x) \left(\frac 1 2 |b(t,x)|^2 +f(x)  \right) dx}
{\int_\Omega p_b(t)} \right) dt  - \epsilon \ln\left( \int_\Omega p_b(T) \right) ,
\end{align}
and $p_b$ solves~\eqref{eq:FP-Dirichlet}. Here, $f$  is a smooth nonnegative function defined in $\overline \Omega$ and the distribution $p_0$ of $X_0$ satisfies $ p_0\ln(p_0)\in L^1 (\Omega)$.
We also suppose that $\epsilon>0$: this assumption will ensure a lower bound for $\int_\Omega p_b(t)$ uniform with respect to $M$.
\\
Let us introduce the truncation   $T_M$: $\RR^d\to \RR^d$,
\begin{displaymath}
  T_M(\xi)=\left\{
    \begin{array}[c]{ll}
      \xi \quad &\hbox{if }|\xi| \le M,\\
      M \frac {\xi}{|\xi|}   \quad &\hbox{if }|\xi|\ge M.
    \end{array}
\right.
\end{displaymath}
and the Hamiltonians 
\begin{equation}
  \label{eq:39}
H_M(\xi)=
  T_M(\xi)\cdot \left(\xi -\frac 1 2 T_M(\xi)\right), \quad\quad H(\xi)=\frac 1 2  |\xi|^2.
\end{equation}
Note that  by contrast with (\ref{eq:def-H}), the Hamiltonians in (\ref{eq:39}) do not involve $f$ because it is more convenient for what follows.

The optimality conditions for~(\ref{eq:11}) consist of a forward-backward system of PDEs:
	\begin{equation}
\label{eq:40}
	\left\{ \quad
	\begin{aligned}
	&-\partial_t u_M(t,x) -\frac {\sigma^2} 2 \Delta u_M(t,x) +\frac { H_M\left(\left(\int_\Omega  p_M(t) \right) Du_M(t,x)\right) }  {\int_\Omega  p_M(t)}=   \ds \frac{f(x)}{\int_\Omega  p_M(t)} -\gamma_M(t)  ,
	 \quad \hbox{ in } Q_T,
	 \\
 	& u_M = 0, \qquad \hbox{ on } (0,T)\times\partial \Omega,
	\\ 
	& u_M(T,x)=  -\epsilon \frac { \mathds{1}_{\Omega}(x)}  {\int_\Omega  p_M(T,y)dy }  ,	\qquad \hbox{ in } \Omega ,
     \end{aligned}
	\right.
	\end{equation}
\begin{equation}
\label{eq:47}
 \left\{ \quad 
 \begin{aligned}
	& \partial_t p_M(t,x) -\frac {\sigma^2} 2 \Delta p_M(t,x) - \diver\left( p_M  b_M \right) (t,x)
=0, \quad  \hbox{in } Q_T,
	\\
	& p_M = 0,\quad\quad \qquad \hbox{ on } (0,T)\times\partial \Omega,
	\\
	& p_M(0,\cdot) = p_0, \qquad \hbox{ in } \Omega, 
  \end{aligned}
\right.
\end{equation}
and 
\begin{equation}
  \label{eq:48}
b_M(t,x)= D_\xi H_M \left( \left(\int_\Omega  p_M(t)\right)   Du_M(t,x) \right) = T_M\left(  \left(\int_\Omega  p_M(t)\right)   Du_M(t,x) \right),
\end{equation}
and  $\gamma_M :[0,T] \to \RR$ is defined by
\begin{displaymath}
  \begin{split}
0\le  \gamma_M(t) &=
   \frac { {\ds \int_\Omega} p_M(t,x)  \left(   \frac 1 2 | b_M(t,x)|^2 
 +f(x)   \right) dx }  {\left( \int_\Omega  p_M(t)\right)^2 }.
   \end{split}
 \end{displaymath}
 \begin{remark}
   \label{sec:letting-m-to-2}
For any smooth function $\phi: [0,T]\to \RR$ compactly supported in $(0,T)$,
we test (\ref{eq:40}) by $p_M\phi$ and (\ref{eq:47}) by $u_M\phi$, subtract the resulting identities and  integrate by part:
 we obtain that  
 \begin{equation}
   \label{eq:51}
  \int_0^T \phi'(t)\int_\Omega u_M(t,x)p_M(t,x) dxdt=0.
 \end{equation}
This means that  $t\mapsto \int_\Omega u_M(t,x)p_M(t,x) dx$ is  constant. On the other hand, it is easy to see 
 from  the last equation in (\ref{eq:40}) 
that $\int_\Omega u_M(T,x)p_M(T,x) dx=-\epsilon$. 
Therefore, 
\begin{equation}
  \label{eq:50}
\int_\Omega u_M(t,x)p_M(t,x)=-\epsilon,\quad \hbox{for all } t\in [0,T].
\end{equation}
 \end{remark}

We wish to study the behaviour of a family of solutions $(b_M, p_M, u_M)_{M}$ as $M\to +\infty$.

\begin{theorem}
  \label{sec:letting-m-to}
There exist 
\begin{enumerate}
\item $u\in L^2(0,T; H_0^1(\Omega))\cap L^\infty(Q_T)$ 
\item $p\in  C^0([0,T]; L^1 (\Omega)) \cap L^\infty(0,T; L{\rm log}L (\Omega))\cap L^r(Q_T) \cap W^{1,q} (0,T; W^{-1,q}(\Omega))$ for all $1\le q \le \frac {d+2}{d+1}$, $1\le r\le  \frac {d+2}{d}$
\item A positive Radon measure $\gamma$ on $[0,T]$,
\end{enumerate}
such that,  setting $\ds b(t,x) = \left(\int_\Omega p(t)\right)  Du(t,x)$
\begin{enumerate}
\item 
\begin{eqnarray}
  \label{eq:63}
  \int_\Omega p(t,x) dx> 0, \quad \forall t\in [0,T], \\
\label{eq:64} \int_{Q_T} |b(t,x)|^2 dxdt +\int_{Q_T} p(t,x)|b(t,x)|^2 dxdt <\infty,\\
\label{eq:65} \int_\Omega p(t,x) u(t,x)= -\epsilon \quad  \hbox{for almost all } t\in (0,T)
\end{eqnarray}
\item up to the extraction of a sequence,
  \begin{enumerate}
  \item  $p_M\to p$ in $L^1(Q_T)$, almost everywhere, in $W^{1,q} (0,T; W^{-1,q}(\Omega))$ weakly for all $1\le q \le \frac {d+2}{d+1}$,
  \item  $u_M\to u$ in $L^1(Q_T)$, almost everywhere in $Q_T$, and in  $L^2(0,T; H^1_0(\Omega))$ weakly,
  \end{enumerate}
\item
\begin{enumerate}
\item $p$ is a distributional solution of  (\ref{eq:FP-Dirichlet}), i.e for all $\psi\in C^\infty(Q_T)$, such that $\psi=0$ on 
$[0,T]\times \partial \Omega$ and at $t=T$,
\begin{equation}
\int_{Q_T}  p_b\left( \partial_t \psi(t,x) -  \frac {\sigma^2} 2 \Delta \psi (t,x) +b(t,x)\cdot D\psi(t,x)\right)dxdt
= \int_\Omega p_0(x)\psi(0,x) dx
\end{equation}

\item $u$ is a distributional solution of
	\begin{equation}
\label{eq:61}
	\left\{ \quad
	\begin{aligned}
	&-\partial_t u(t,x) -\frac {\sigma^2} 2 \Delta u(t,x) +\frac { H\left(\left(\int_\Omega  p(t) \right) Du(t,x)\right) }  {\int_\Omega  p(t)}=   \ds \frac{f(x)}{\int_\Omega  p(t)}   -\gamma(t) ,
	 \quad \hbox{ in } Q_T,
	\\ 
	& u(T,x)=  -\epsilon \frac { \mathds{1}_{\Omega}(x)}  {\int_\Omega  p(T,y)dy }  ,	\qquad \hbox{ in } \Omega 
     \end{aligned}
	\right.
	\end{equation}
      \end{enumerate}
      
    \item
      \begin{equation}\label{eq:66}
        J(b)\le \lim_{M\to +\infty} \inf_{b'\in \cB_M} J(b') =\inf_{M>0} \inf_{b'\in \cB_M} J(b') . 
      \end{equation}
    \end{enumerate}
    
  \end{theorem}  
  \begin{remark}
    Equation (\ref{eq:50})  plays the role of the energy identity that is usually found in the theory of mean field games, see for example \cite{MR3305653}.
\\
Besides, it seems possible to prove that 
\begin{displaymath}
   \gamma(t) =
   \frac { {\ds \int_\Omega} p(t,x)  \left(   \frac 1 2 | b(t,x)|^2 
 +f(x)   \right) dx }  {\left( \int_\Omega  p(t)\right)^2 },
\end{displaymath}
by using the  crossed regularity results in the spirit of what did A. Porretta for weak solutions of mean field games, but this would be too long for the present paper.
  \end{remark}

\begin{proof}

 We start by looking for estimates uniform in $M$ and compactness results.

For all what follows, it will be important to have a bound from below on $\int_\Omega p_M(t)$ uniform in $M$;
first, since $f$ is bounded from below, and since  $-\epsilon \ln \left(\int_\Omega p_M (T) \right)\le J(b_M)\le J(0)$, we see that  $\int_\Omega p_M(T)$ is bounded from below by a positive constant independent of $M$. Second, from  known results on weak solutions of Fokker-Planck equations, \cite[Proposition 3.10]{MR3305653}, we know that $t\mapsto \int_\Omega p_M(t)$ is a nonincreasing function from $[0,T]$ to $(0,+\infty)$. Therefore $\int_\Omega p_M(t)$ is bounded from below by a positive constant uniform in $t$ and $M$ (but not in $\epsilon$), and from above by $\int_\Omega p_0$.
\\
Using this information and again the fact that $J(b_M)\le J(0)$, we see that $\|\gamma_M\|_{L^1 (0,T)}$
 is bounded  uniformly with respect to  $M$. 
\\
From the latter fact, we immediately deduce  a bound on $ \|u_M\|_{L^\infty(Q_T)}$ uniform in $M$, by constructing subsolutions and supersolutions of the form $(t,x)\mapsto \theta (t)$.  Then  testing (\ref{eq:40}) by $u_M e^{\lambda u_M^2}$ with $\lambda$ large enough,  we deduce that $\|u_M\|_{L^2 (0,T; H^1_0(\Omega))}$ is bounded uniformly with respect to $M$.
This implies that  $  -\partial_t u_M -\Delta u_M $ is bounded in $L^1(Q_T)$ uniformly with respect to $M$.
From classical results on the heat equation with $L^1$ data, see e.g. \cite{MR1025884,MR1453181},  we deduce that $(u_M)_M$ is relatively compact in $L^1(Q_T)$ and that  there exists $u\in  L^2 (0,T; H^1_0(\Omega))\cap L^\infty(Q_T)$ such that, up to the extraction of a sequence,   $u_M\to u$ in $L^1(Q_T)$ and almost everywhere in $Q_T$, 
even in $L^q(Q_T)$ for all $1\le q <+\infty$ and in  $L^\infty(Q_T)$ weak * due to the $L^\infty (Q_T)$ bound,
and $Du_M\to Du$ almost everywhere in $Q_T$. Moreover $t\mapsto u(t,\cdot)$ has a trace at $t=T$ which is a Radon measure.
\\
Let us go back to the Fokker-Planck equation. First, the uniform bounds on $\|\gamma_M\|_{L^1 (0,T)}$  and on $\int_\Omega p_M(t)$ imply that $\int_{Q_T} p_M |b_M|^2$ is uniformly bounded with respect to $M$. Second,
 we deduce from (\ref{eq:48}) and  the bound on $\|u_M\|_{L^2 (0,T; H^1_0(\Omega))}$ that $\|b_M\|_{L^2 (Q_T)}$ is bounded uniformly in $M$. The latter two points  enable us to use \cite[Proposition 3.10]{MR3305653},
 which  ensures that for all $1\le q < \frac {d+2}{d+1}$ and $1\le r< \frac {d+2}d$, there exists a constant $C>0$ uniform in $M$ such that
\begin{displaymath}
   \|p_M\|_{L^\infty (0,T; L^1 (\Omega) )}
+\|p_M\|_{L^r(Q_T)}+ \|Dp_M\|_{L^q(Q_T)}+  \|\partial_t p_M\|_{L^q(0,T; W^{-1,q}(\Omega))}\le C,
\end{displaymath}
where $C$ depends on the bound on  $\int _{Q_T} p_M|b_M|^2$ and on $\|p_0\|_{L^1(\Omega)}$.

Moreover, since $p_0\in L {\rm{log}} L (\Omega)$  then for  $1\le q \le \frac {d+2}{d+1}$ and $1\le r \le \frac {d+2}d$,
\begin{displaymath}
  \|p_M\|_{L^\infty (0,T; L{\rm{log}}L (\Omega) )}
+\|p_M\|_{L^r(Q_T)}+ \|Dp_M\|_{L^q(Q_T)}+  \|\partial_t p_M\|_{L^q(0,T; W^{-1,q}(\Omega))}\le C,
\end{displaymath}
where $C$ depends on the bound on  $\int _{Q_T} p_M|b_M|^2$ 
and $ \|p_0\|_{L{\rm{log}}L (\Omega) }$.

Since $\|b_M\|_{L^2 (Q_T)}$ is bounded uniformly in $M$, the last point in \cite[Proposition 3.10]{MR3305653} ensures also that $p_M$ lies in a relatively compact subset of $L^1(Q_T)$.\\
Therefore, 
there exists a nonnegative function $p\in L^1(Q_T)\cap W^{1,q} (0,T; W^{-1,q}(\Omega))$
for all $1\le q \le \frac {d+2}{d+1}$,
 such that up to an extraction of a subsequence, we may assume that
 $p_M\to p$ in $L^1(Q_T)$, almost everywhere, and in $W^{1,q} (0,T; W^{-1,q}(\Omega)) $ weakly. 
Therefore, up to a further extraction,   $t\mapsto \int_\Omega p_M(t,y)dy$ tends to $t\mapsto \int_\Omega p(t,y)dy$  strongly in $L^1 (0,T)$ and a.e. in $(0,T)$.

 Moreover, we know that $p_M$ belongs to $C([0,T];L^1(\Omega))$ and  is bounded in $ L^\infty(0,T; L{\rm{log}} L(\Omega))$ by a constant $C$ independent of  $M$.
Let us fix $t\in [0,T]$. There exists a sequence $t_n$ which tends to $t$ such that $\|p_M(t_n)\|_{ L{\rm{log}} L(\Omega)}\le C$ and $p_M(t_n,\cdot)\to p_M(t,\cdot)$ in $L^1(\Omega)$.
Up to the extraction of a subsequence, we can assume that $p_M(t_n,x)\to p_M(t,x)$ almost everywhere in $\Omega$. Then Fatou lemma implies that $\int_{\Omega} p_M(t,x)\log(p_M(t,x) )dx \le \liminf_{t_n\to t} \int_{\Omega} p_M(t_n,x)\log(p_M(t_n,x) ) dx\le C$. This implies that $p_M(t)$ is bounded in $L{\rm{log}} L(\Omega)$ uniformly with respect to $M$ and $t$.
Hence,  for all $t$, $(x\mapsto p_M(t,x))_M$ is equi-integrable. By Dunford-Pettis theorem,  $p_M(t,\cdot)$ is relatively weakly compact 
in $L^1(\Omega)$. Now choosing $t=T$,  since we already know that $p_M(T,\cdot)\rightharpoonup p(T,\cdot)$ weakly in  $W^{-1,q}(\Omega)$, we  deduce that $p_M(T,\cdot)\rightharpoonup p(T,\cdot)$ weakly in $L^1(\Omega)$.

Let us now consider the sequence $b_M$: from (\ref{eq:48}) and the almost everywhere convergence of $Du_M$ to $Du$ and of $p_M$ to $p$, we deduce that $b_M $ tends to $b$  almost everywhere in $Q_T$, where
\begin{displaymath}
  b(t,x)=   \left(\int_\Omega  p(t)\right)   Du(t,x).
\end{displaymath}
From Fatou's lemma and the uniform estimate on  $\|b_M\|_{L^2 (Q_T)}$, we see that $b\in L^2 (Q_T)$. Similarly, we see that $\int _{Q_T} p|b|^2$ is finite.
\\
We now use  the uniform bound on $\int_{Q_T} p_M |b_M|^2$ and the strong convergence of $p_M$ to $p$ in $L^1(Q_T)$: for all measurable  subset $E$ of $Q_T$,
\begin{displaymath}
  \int_E p_M |b_M| \le \left(\int_{Q_T}   p_M |b_M|^2 \right)^{\frac 1 2}\left(\int_E p_M\right)^{\frac 1 2},
\end{displaymath}
which implies that $p_M b_M$ is equiintegrable. Therefore $p_Mb_M\rightharpoonup pb $ in $L^1(Q_T)$.\\
Hence, $p$ is a
 distributional solution of (\ref{eq:FP-Dirichlet}), i.e., for all test function $\phi \in C^\infty(Q_T)$ such that $\phi=0$ on $[0,T]\times \partial \Omega$ and on $\{T\}\times \Omega$,
 \begin{displaymath}
   \int_{Q_T}   ( -\partial_t \phi -\Delta \phi + b\cdot D\phi ) p -\int_{\Omega}  p_0(x) \phi(0,x) dx =0.
 \end{displaymath}
Then  using \cite[Theorem 3.6]{MR3305653}, $p\in C([0,T];L^1 (\Omega) )$ and $p$ is a renormalized solution of  (\ref{eq:FP-Dirichlet}), because $b\in L^2(Q_T)$ and  $\int_{Q_T } p |b|^2 $ is finite. 
\\
For $\phi$ as in Remark \ref{sec:letting-m-to-2}, we can pass to the limit in (\ref{eq:51}), because $u_M
\rightharpoonup u$ in $L^\infty(Q_T)$ weak * and $p_M\to p$ in $L^1(Q_T)$, and we get 
  \begin{equation*}
  \int_0^T \phi'(t)\int_\Omega u(t,x)p(t,x) dx=0,
 \end{equation*}
which implies that $t\mapsto \int_\Omega u(t,x)p(t,x) dx$ is constant. 
From (\ref{eq:50}), we know that $\int_{Q_T} p_M(t,x) u_M(t,x) dxdt =-\epsilon T$. Passing to the limit, 
we see that $\int_{Q_T} p(t,x) u(t,x) dxdt =-\epsilon T$. Since $\int_\Omega u(t,x)p(t,x) dx$ does not depend  on $t$, we conclude that
\begin{equation}
  \label{eq:52}
\int_\Omega u(t,x)p(t,x) dx=-\epsilon.
\end{equation}
\\
We are left with passing to the limit in the Bellman equation (\ref{eq:40}). Since $0\le \gamma_M$ is bounded in $L^1 (0,T)$,
$\int_t^T \gamma_M(s)ds$ is bounded in BV, and 
there exists a bounded positive measure $\gamma$ on $[0,T]$ such that 
 up to a further extraction of a subsequence,  $\gamma_M\rightharpoonup \gamma$  in the sense of measures
and $\mu_M:  \mu_M(t)= \int_t^T \gamma_M(s)ds  $ converges to some $\mu$ in $L^q (0,T)$ for all $1\le q<\infty$ and weakly * in $L^\infty (0,T)$, and $\frac {d\mu}{dt}=-\gamma$ in the sense of distributions.  
We deduce that $\mu$ has bounded variations, so $\mu(T)=\lim_{t\to T} \mu(t)$ exists. 
For a smooth function $\psi$  compactly supported in $(0,T]$,  we get by passing to the limit that 
\begin{equation}
  \label{eq:53}
  \int_0^T \psi(t) d\gamma(t)-\int_0^T \psi'(t) \mu(t)dt =0.
\end{equation}
 Let us set $v_M(t,x)= u_M(t,x)+\mu_M(t) $: we see that $v_M$ tends to $v$: $v(t,x)=u(t,x)+\mu(t)$ in 
 $L^q (Q_T)$ for all $1\le q<\infty$, weakly * in $L^\infty (0,T)$, and in $L^2 (0,T; H^1(\Omega))$ weakly.
The function $v_M$  is a solution of the following boundary value problem:
	\begin{equation}
\label{eq:49}
	\left\{ \quad
	\begin{aligned}
	&-\partial_t v_M(t,x) -\frac {\sigma^2} 2 \Delta v_M(t,x) +\frac { H_M\left(\left(\int_\Omega  p_M(t) \right) Dv_M(t,x)\right) }  {\int_\Omega  p_M(t)}=   \ds \frac{f(x)}{\int_\Omega  p_M(t)}   ,
	 \quad \hbox{ in } Q_T,
	 \\
 	& v_M - \mu_M = 0, \qquad \hbox{ on } (0,T)\times\partial \Omega,
	\\ 
	& v_M(T,x)=  -\epsilon \frac { \mathds{1}_{\Omega}(x)}  {\int_\Omega  p_M(T,y)dy }  ,	\qquad \hbox{ in } \Omega .
     \end{aligned}
	\right.
	\end{equation}
By using stability results for weak solutions of Bellman equations, see \cite{MR766873,MR3452251}, we obtain that $v$ is a distributional solution 
of 
	\begin{equation}
\label{eq:55}
	\left\{ \quad
	\begin{aligned}
	&-\partial_t v(t,x) -\frac {\sigma^2} 2 \Delta v(t,x) +\frac { H\left(\left(\int_\Omega  p(t) \right) Dv(t,x)\right) }  {\int_\Omega  p(t)}=   \ds \frac{f(x)}{\int_\Omega  p(t)}   ,
	 \quad \hbox{ in } Q_T,
	\\ 
	& v(T,x)=  -\epsilon \frac { \mathds{1}_{\Omega}(x)}  {\int_\Omega  p(T,y)dy }  ,	\qquad \hbox{ in } \Omega ,
     \end{aligned}
	\right.
	\end{equation}
i.e. that for all smooth function $\psi: \overline Q_T \to \RR$ with compact support in $(0,T]\times \Omega$,
\begin{displaymath}
  \begin{split}
&  \int_{Q_T} v(t,x)  \partial_t \psi(t,x) dxdt + 
\epsilon \frac {\int_\Omega  \psi(T,x)  dx} {\int_\Omega  p(T,x)dx } 
 \\
& +\frac {\sigma^2} 2 \int_{Q_T}  D v(t,x)\cdot D\psi(t,x) dx dt
 +
\int_{Q_T}\frac { H\left(\left(\int_\Omega  p(t) \right) Dv(t,x) \right) }  {\int_\Omega  p(t)} \psi(t,x) dx  dt
\\= & \int_{Q_T}  \ds \frac{f(x)}{\int_\Omega  p(t)}    \psi(t,x) dx dt .   
  \end{split}
\end{displaymath}
From this and (\ref{eq:53}), we deduce that
\begin{equation}
  \label{eq:54}
  \begin{split}
&  \int_{Q_T} u(t,x)  \partial_t \psi(t,x) dxdt + 
\epsilon \frac {\int_\Omega  \psi(T,x)  dx} {\int_\Omega  p(T,x)dx } 
 \\
& +\frac {\sigma^2} 2 \int_{Q_T}  D u(t,x)\cdot D\psi(t,x) dx dt
 +
\int_{Q_T}\frac { H\left(\left(\int_\Omega  p(t) \right) Du(t,x) \right) }  {\int_\Omega  p(t)} \psi(t,x) dx  dt
\\= & \int_{Q_T}  \ds \frac{f(x)}{\int_\Omega  p(t)}    \psi(t,x) dx dt - \int_{Q_T} \psi(t,x) d\gamma(t)dx    ,
  \end{split}
\end{equation}
which means that $u\in L^2(0,T; H^1_0(\Omega))$ is a distributional solution  of  (\ref{eq:61}), which furthermore satisfies (\ref{eq:65}).\\
  Finally (\ref{eq:66}) is a consequence of  the weak convergence of  $p_M(T)$ to $p(T)$ in $L^1(\Omega)$, the almost everywhere convergence of $p_M$, the almost everywhere convergence of $t\mapsto \int_\Omega p_M(t)$ to $t\mapsto \int_\Omega p(t)$,  the almost everywhere convergence of $p_M |b_M|^2 $ to $p |b|^2 $, using Fatou lemma when necessary.
\end{proof}

\section{Numerical method for the finite horizon problem}
\label{sec:numerics-time}

\subsection{Numerical method}
\label{subsec:finite-hor-num-meth}

\paragraph{\textbf{Discretization.} }
We will use finite differences.
To save notations, we focus on one-dimensional problems, but the same method can be generalized e.g. to dimension $2$. Let us take $\Omega = [0,x_{\mathrm{max}}] \subset \RR$ for some $x_{\mathrm{max}}>0$. 
Let $N_T$ be a positive integer and let $h \in (0,x_{max})$ be such that $x_{max}/h$ is an integer. Let us denote $\Delta t = T/N_T$ and $N_h = x_{max}/h$. We consider uniform grids on $[0,T]$ and $[0,x_{max}]$ with respectively $(N_T+1)$ and $(N_h+1)$ points. Set $t_n = n\Delta t,$ and $x_i = i \,h$ for $(n,i) \in \{0,\dots,N_T\}\times\{0,\dots,N_h\}$.
The values of $u$ and $p$ at $(x_i,t_n)$ are respectively approximated by $U^{n}_{i}$ and $P^{n}_{i}$, for each $(n,i)$. We will use the notations $U_i = (U_i^n)_{n=0,\dots,N_T}$, $U^n = (U_i^n)_{i=0,\dots,N_h}$, and likewise for $P$. 

In the sequel, following~\cite{MR3135339} in the context of MFG, we consider a numerical Hamiltonian corresponding to $\widecheck H$ defined by~\eqref{eq:def-checkH}. More precisely, we consider $\widetilde H: \RR \times \RR_+^*  \times \RR \times \RR \to \RR$, $(x, \mu,\xi_1,\xi_2) \mapsto \widetilde H (x, \mu, \xi_1,\xi_2)$ satisfying the following conditions:
\begin{itemize}
	\item Monotonicity: $\widetilde H$ is non-increasing with respect to $\xi_1$ and nondecreasing with respect to $\xi_2$.
	\item Consistency: $\widetilde H(x, \mu, \xi, \xi) = \widecheck H(x, \mu, \xi)$.
	\item Differentiability: $\widetilde H$ is of class $\mathcal C^1$ in all variables.
	\item Convexity: for all $x \in \RR, \mu \in \RR_+^*$,  $(\xi_1,\xi_2) \mapsto \widetilde H(x, \mu, \xi_1, \xi_2)$ is convex. 
\end{itemize}

\begin{remark}
	For example, if $L(x,b) = f(x) + \frac{1}{2}|b|^2$, then $\widecheck H$ is given in Remark~\ref{eq:Lquadra} and for the numerical Hamiltonian, one can take
	$$
	\widetilde H(x,\mu, \xi_1,\xi_2) = 
	\begin{cases}
		\frac{1}{2} \mu \left( (\xi_1^-)^2 + (\xi_2^+)^2 \right) - \frac{f(x)}{\mu}, &\hbox{ if } \mu\left( (\xi_1^-)^2 + (\xi_2^+)^2 \right)^{1/2} \leq M,
		\\
		\frac{1}{2} M \left( (\xi_1^-)^2 + (\xi_2^+)^2 \right)^{1/2} - \frac{f(x)}{\mu} - \frac{M^2}{2\mu}, &\hbox{ otherwise,}
	\end{cases}
	$$ 
	where $x \in \RR^d$, $\xi_1, \xi_2 \in \RR$, $\mu \in \RR_+^*$.
\end{remark}
We also consider discrete versions of $F,\, G$ that we will be noted  $\widetilde  F, \,\widetilde G: \RR^{N_h+1} \to \RR^{N_h+1}$. 
A typical set of assumptions that can be made on $\widetilde \varphi\in \{ \widetilde F,\widetilde G\} $ is that
\begin{itemize}
\item
  $\widetilde \varphi$ is Lipschitz continuous from $\RR^{N_h+1}$  endowed with the discrete $L^2 $ norm : $\|V\|_{2}=  \left(h \sum_{i} V_i^2\right)^{\frac 1 2}$ to itself, with a Lipschitz constant independent of $h$

\item Let $I_h$ be the  piecewise linear interpolation  at the grid nodes. There exists a continuous and bounded function $\omega: \RR_+ \to \RR_+$ such that $\omega(0) = 0$ and for all $p \in L^2(\Omega)$, for all sequences $( P^{(h)})_h$, 
	$$
		\left\| \varphi [p] -   I_h \left( \widetilde \varphi[P^{(h)}]\right) \right\|_{L^2(\Omega)}
		\leq
		\omega\left( \|p -  I_h P^{(h)}\|_{L^2(\Omega)}\right).
	$$
\end{itemize}

We introduce the finite difference operators: for $W \in \RR^{N_T+1}$,
\begin{align*}
	(D_t W)^n &= \frac{1}{\Delta t}(W^{n+1} - W^n), \qquad n \in \{0, \dots N_T-1\},
\end{align*}
and for $W \in \RR^{N_h+1},$
\begin{align*}
	(D^+ W)_i &= \frac{1}{h} (W_{i+1} - W_{i}), \qquad i  \in \{0, \dots N_h-1\},
	\\
	(\Delta_h W)_i &= -\frac{1}{h^2} \left(2 W_i - W_{i+1} - W_{i-1}\right), \qquad i \in \{ 1, \dots N_h-1\}, 
	\\
	[\grad_h W]_i &= \left( (D^+ W)_{i}, (D^+ W)_{i-1} \right)^T, \qquad i \in \{ 0, \dots N_h-1\}.
\end{align*}

\paragraph{\textbf{Discrete HJB equation.} }
To alleviate the notations, let us introduce, for $P,U \in \RR^{N_h+1}$,
$$
	\cF(P,U)_i = 
	-  {\ds \sum_k} h P_k  \widetilde H_\mu \left(x_k, \sum_j h P_j,  DU_k \right) 
	+ \frac {\widetilde F\left[ \frac {P} {\sum_j h P_j } \right](x_i)} {\sum_j h P_j } 
	-  \frac { {\ds \sum_k} h P_k  
    \widetilde F\left[ \frac {P} {\sum_j h P_j } \right](x_k) }  {\left( \sum_j h P_j\right)^2 }
$$
and
$$
	\cG(P)_i = 
	- \frac { \epsilon}  {\sum_j h P_j } 
	+ \frac 1  {\sum_j h P_j} \widetilde G\left[ \frac {P} {\sum_j h P_j } \right](x_i) 
	- \frac {\sum_k P_k \widetilde G\left[ \frac {P} {\sum_j h P_j } \right](x_k)  h}  {\left(\sum_j h P_j\right)^2}.
$$
Then, for the HJB equation~\eqref{eq:HJB-checkH}, we consider the following finite difference scheme:
\begin{align}
	- (D_t U_{i})^{n} - \frac{\sigma^2}{2} (\Delta_h U^{n})_{i}
	+ \widetilde  H\left(x_i, \sum_j h P_j^{n+1}, [\grad_h U^n]_i\right)
	& = \cF(P^{n+1},U^{n})_i
	\label{eq:HJBcond-H-scheme-discrete-edo}
	\\ &  i \in \{1,\dots,N_h-1\}, n \in \{0, \dots, N_T-1\},
	\notag
	\\
	U^{n}_{i}  = 0, \qquad & i \in \{0, N_h\}, n \in \{0, \dots,  N_T-1\},   
	\label{eq:HJBcond-H-scheme-discrete-bc}
	\\
	\label{eq:HJBcond-H-scheme-discrete-final}
	U^{N_T}_{i} = \cG(P^{N_T})_i \,, \qquad & i \in \{0, \dots, N_h\} .
\end{align}

\paragraph{\textbf{Discrete FP equation.} } 
To define a discretization of the FP equation, we consider the weak form of~\eqref{eq:time-FP}. It involves, for a smooth $w$,
\begin{align*}
	&- \int_{\Omega} \diver\left( p(t,\cdot) \widecheck H_\xi\left(\cdot, \int_\Omega p(t), \grad u(t,\cdot)\right) \right)(x) w(t,x) dx
	\\
	=\, &
	\int_{\Omega} p(t,x) \widecheck H_\xi\left(x, \int_\Omega p(t), \grad u(t,x)\right) \cdot \grad w(t,x) dx,
\end{align*}
where we used integration by parts and the Dirichlet boundary condition.

This leads us to introduce the following discrete operator (see~\cite{MR3135339} for more details), for $\mu\in\RR$, $U, P \in \RR^{N_h+1}$, and for $ i \in \{1,\dots, N_h-1\}$,
\begin{equation}
	\label{eq:def-B-div}
	\begin{aligned}
	&\mathcal B_i^{(\mu,U)}(P) 
	= \frac{1}{h} 
	\left[ P_i  \widetilde H_{\xi_1}\left(x_i, \mu, [\grad_h U]_i\right) 
		- P_{i-1} \widetilde H_{\xi_1}\left(x_{i-1}, \mu, [\grad_h U]_{i-1}\right)
	\right.
		\\
		&\qquad\qquad\qquad\left. +
		P_{i+1} \widetilde H_{\xi_2} \left(x_{i+1}, \mu,  [\grad_h U]_{i+1}\right)
		- P_i \widetilde H_{\xi_2} \left(x_i, \mu, [\grad_h U]_i\right)
	\right].
	\end{aligned}
\end{equation}
Then, for the discrete version of the FP equation~\eqref{eq:time-FP} we consider the following finite difference equation:
\begin{align}
	(D_t P_{i})^{n} - \frac{\sigma^2}{2} (\Delta_h P^{n+1})_{i}
	- \mathcal B_i^{(\sum_j h P^{n+1}_j, U^{n})}(P^{n+1})
	& = 0,
	\label{eq:FPcond-H-scheme-discrete-edo}
	\\
	& \qquad i \in \{1,\dots,N_h-1\}, n \in \{0, \dots, N_T-1\},
	\notag
	\\
	P^{n}_{i} = 0, & \qquad i \in \{0, N_h\}, n \in \{ 1, \dots, N_T\}, 
	\label{eq:FPcond-H-scheme-discrete-bc}
	\\
	P^{0}_{i} = p_0(x_i), & \qquad i \in \{0, \dots, N_h\}.
	\label{eq:FPcond-H-scheme-discrete-initial}
\end{align}

\paragraph{\textbf{Numerical method.} } 

An interesting option for solving~\eqref{eq:time-HJB}--\eqref{eq:time-FP} is to use an Augmented Lagrangian method as was done in \cite{AL2016b} for mean field type control problems, but we have not done this yet.
One could also try to solve directly the whole system~\eqref{eq:time-HJB}--\eqref{eq:time-FP} using for example Newton method, see~\cite{MR3135339} in the context of MFG. However, equations~\eqref{eq:HJBcond-H-scheme-discrete-edo} and~\eqref{eq:FPcond-H-scheme-discrete-edo} involve some non-local terms in $P$ and $U$, and hence one would need to invert full matrices, which would probably be very time consuming. 
For this reason we instead use the following iterative method, where the solution computed at a given iteration is used to compute the non-local terms involved in the next iteration:

\begin{enumerate}
	\item Start with a guess $(P^{(0)},U^{(0)})$; set $k \leftarrow 0$.
	\item Repeat the following steps
		\begin{enumerate}
			\item Compute $U^{(k+1)}$, solution to the following modified version of~\eqref{eq:HJBcond-H-scheme-discrete-edo}--\eqref{eq:HJBcond-H-scheme-discrete-final},
			\begin{equation}
			\label{eq:HJB-discrete-kp1}
			\left\{\quad 
			\begin{aligned}
			- (D_t U_{i})^{n} - \frac{\sigma^2}{2} (\Delta_h U^{n})_{i}
				&+ \widetilde  H\left(x_i, \sum_j h P_j^{(k),n+1}, [\grad_h U^n]_i\right)
				= \cF(P^{(k),n+1},U^{(k),n})_i ,
			\\
			&   i \in \{1,\dots,N_h-1\}, n \in \{0, \dots, N_T-1\},
			\\
			U^{n}_{i} = 0, \qquad &i \in \{0, N_h\}, n \in \{0, \dots,  N_T-1\}, 
			\\
			U^{N_T}_{i} = \cG(P^{(k),N_T})_i, \qquad &i \in \{ 0, \dots, N_h\}.
			\end{aligned}
			\right.
			\end{equation}
			\item Compute $P^{(k+1)}$, solution to the following modified version of~\eqref{eq:FPcond-H-scheme-discrete-edo}--\eqref{eq:FPcond-H-scheme-discrete-initial},
			\begin{equation}
			\label{eq:FP-discrete-kp1}
			\left\{\quad 
			\begin{aligned}
			(D_t P_{i})^{n} - \frac{\sigma^2}{2} (\Delta_h P^{n+1})_{i}
				& - \mathcal B_i^{(\sum_j h P^{(k),n+1},U^{(k+1),n})}(P^{n+1})
				= 0, 
			\\
			&i \in \{1,\dots,N_h-1\}, n \in \{0, \dots, N_T-1\},
			\\
			P^{n}_{i} = 0, \qquad & i \in \{0, N_h\}, n \in \{ 1, \dots, N_T\}, 
			\\
			P^{0}_{i} = p_0(x_i), \qquad & i \in \{0, \dots, N_h\}.
			\end{aligned}
			\right.
			\end{equation}
			\item If $||P^{(k+1)} - P^{(k)} ||_{\ell^2}$ and $||U^{(k+1)} - U^{(k)} ||_{\ell^2}$ are small enough, stop. 
				Otherwise set $k \leftarrow k+1$ and continue.
		\end{enumerate}
\end{enumerate}
To ensure convergence, it is sometimes helpful to use relaxation in the update rule, that is, denoting by $\tilde U^{(k+1)}$ the solution of~\eqref{eq:HJB-discrete-kp1}, we set $U^{(k+1)} = (1-\theta) U^{(k)} + \theta \tilde U^{(k+1)}$, for some $\theta \in (0,1)$, and similarly for $P^{(k+1)}$.

\paragraph{\textbf{Solving the finite difference equations.} } Equation~\eqref{eq:HJB-discrete-kp1} can be solved using backward time marching with Newton method at each time step. Equation~\eqref{eq:FP-discrete-kp1} can be solved using forward time marching and solving a linear system of equations at each time step.

In our numerical implementation, in order to solve the linear systems associated to~\eqref{eq:HJB-discrete-kp1}--\eqref{eq:FP-discrete-kp1}, we have used the open source library \texttt{UMFPACK}~\cite{MR2075981} which contains an Unsymmetric MultiFrontal method for solving linear systems.

\subsection{Numerical results}
\label{sec:numres-time}

\subsubsection{Case 1: a one dimensional problem with $T=0.2$}
\label{sec:time-longT}
We present here results in dimension $1$ with a small time horizon: we take $\Omega = (0,1)$ and $T=0.2$. The cost functional is given by~\eqref{eq:def-time-J}, with 
$$
	L(x,b) = \frac{1}{2}|b|^2, \qquad \Psi[p] = \int_\Omega p(x) g_T(x) dx, \qquad g_T(x) = -\frac{1}{2} \exp\left( \left(x- 0.7\right)^2 / \left(\tfrac{1}{5}\right)^2 \right)
$$
and $\Phi \equiv 0, \epsilon = 0$, for $x \in \RR, b \in \RR, p \in L^2(\Omega; \RR)$. The initial distribution has a density given by
$$
	p_0(x) = \frac{1}{Z_0} \max\left\{0, \exp\left( \left(x - \tfrac{1}{4}\right)^2 / (\tfrac{1}{10})^2 \right) - 0.05 \right\}
$$
where $Z_0$ is the appropriate constant ensuring that $\int p_0 = 1$. In this case, the terminal condition for $u$ in~\eqref{eq:time-HJB} reads
\begin{equation}
\label{eq:terminal-cond-u-g}
	u(T,x) = \frac{g_T(x)} {\int_\Omega p(T,y) dy} - \frac{\int_\Omega g_T(y) p(T,y) dy} {\left(\int_\Omega p(T,y) dy\right)^2}.
\end{equation}
The graphs of $p_0$ and $g_T$ are shown in Figure~\ref{fig:shortT-gTp0}. The evolution of $p$ and $u$ is shown in Figure~\ref{fig:shortT-evol-pu} for $\sigma=0.8$. Since the goal is to minimize the cost functional, the mass moves towards the minimum of $g_T$.

The numerical convergence for this test case is illustrated by Figure~\ref{fig:time-conv-iter}. We have stopped the iterative procedure when the normalized $\ell^2$ distance between two successive approximations of $p$ and $u$ were both smaller than $10^{-6}$, where the (time-space) normalized $\ell^2$ norm is defined, for a vector $V \in \RR^{(N_h+1)\times(N_T+1)}$, by
$$
	\| V \|_{\ell^2} = \left( h \, \Delta t  \sum_{n=0}^{N_T} \sum_{i=0}^{N_h} |V_i^{n}|^2 \right)^{1/2}.
$$

For the results presented in Figures~\ref{fig:shortT-gTp0}--\ref{fig:time-conv-iter}, we have used $h = 5 \times 10^{-4}$ and $\Delta t = 2 \times 10^{-4}$.

\begin{center}
  \includegraphics[width=0.48\linewidth]{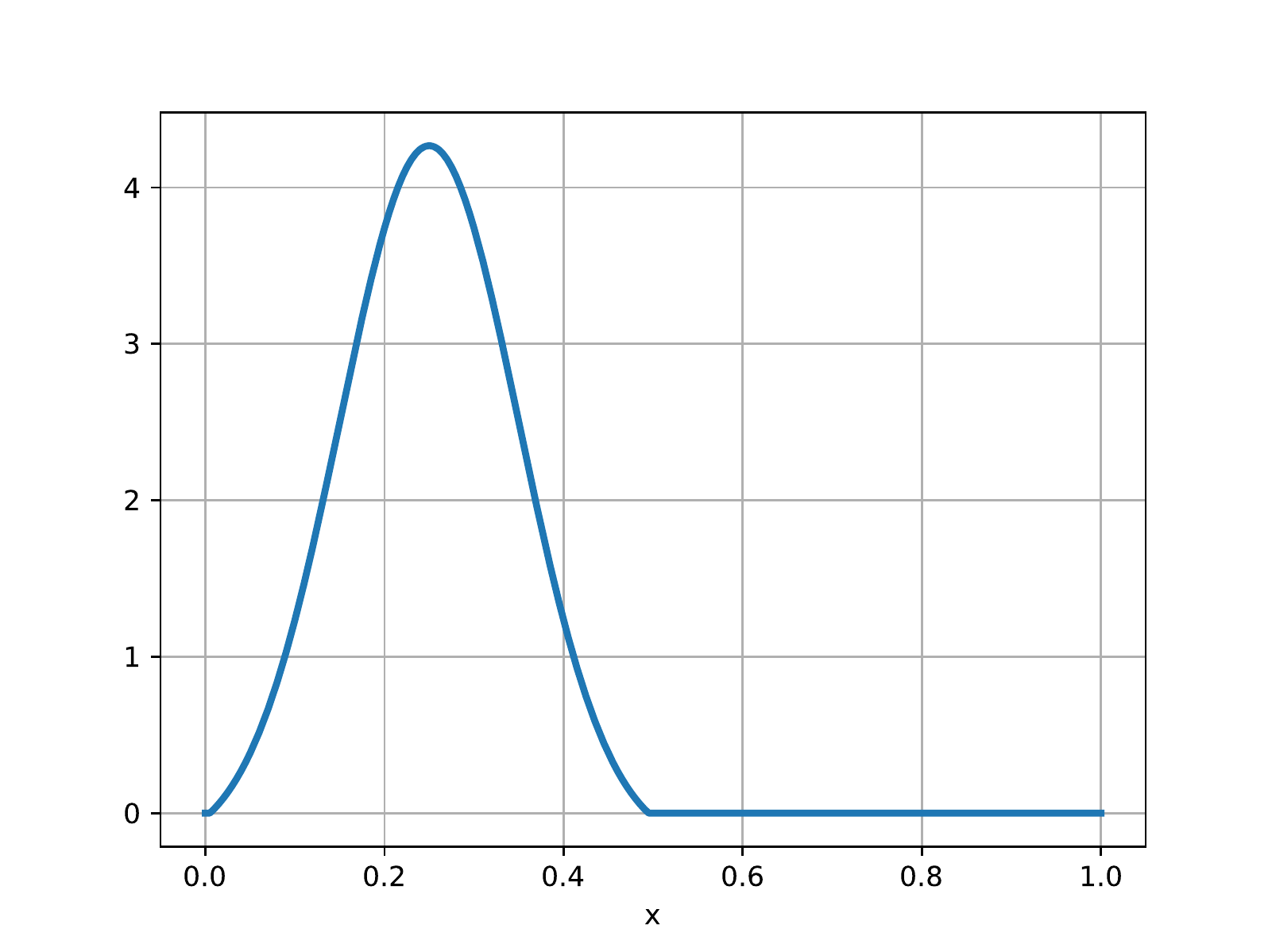}
  \includegraphics[width=0.48\linewidth]{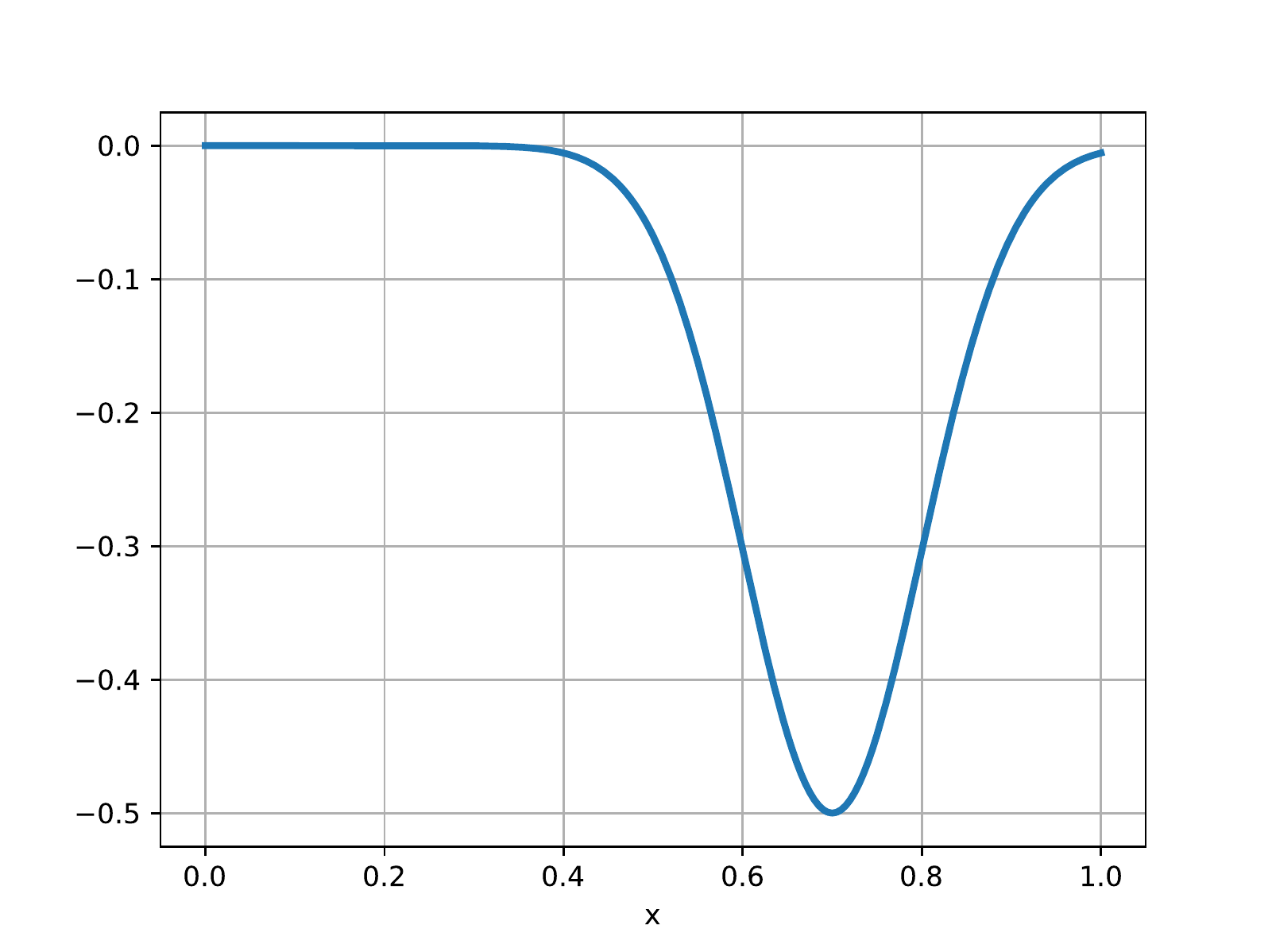}
  \captionof{figure}{Case 1: data of the problem: initial density $p_0$ (left) and terminal cost $g_T$ (right).}
\label{fig:shortT-gTp0}
\end{center}


\begin{center}
  \includegraphics[width=0.48\linewidth]{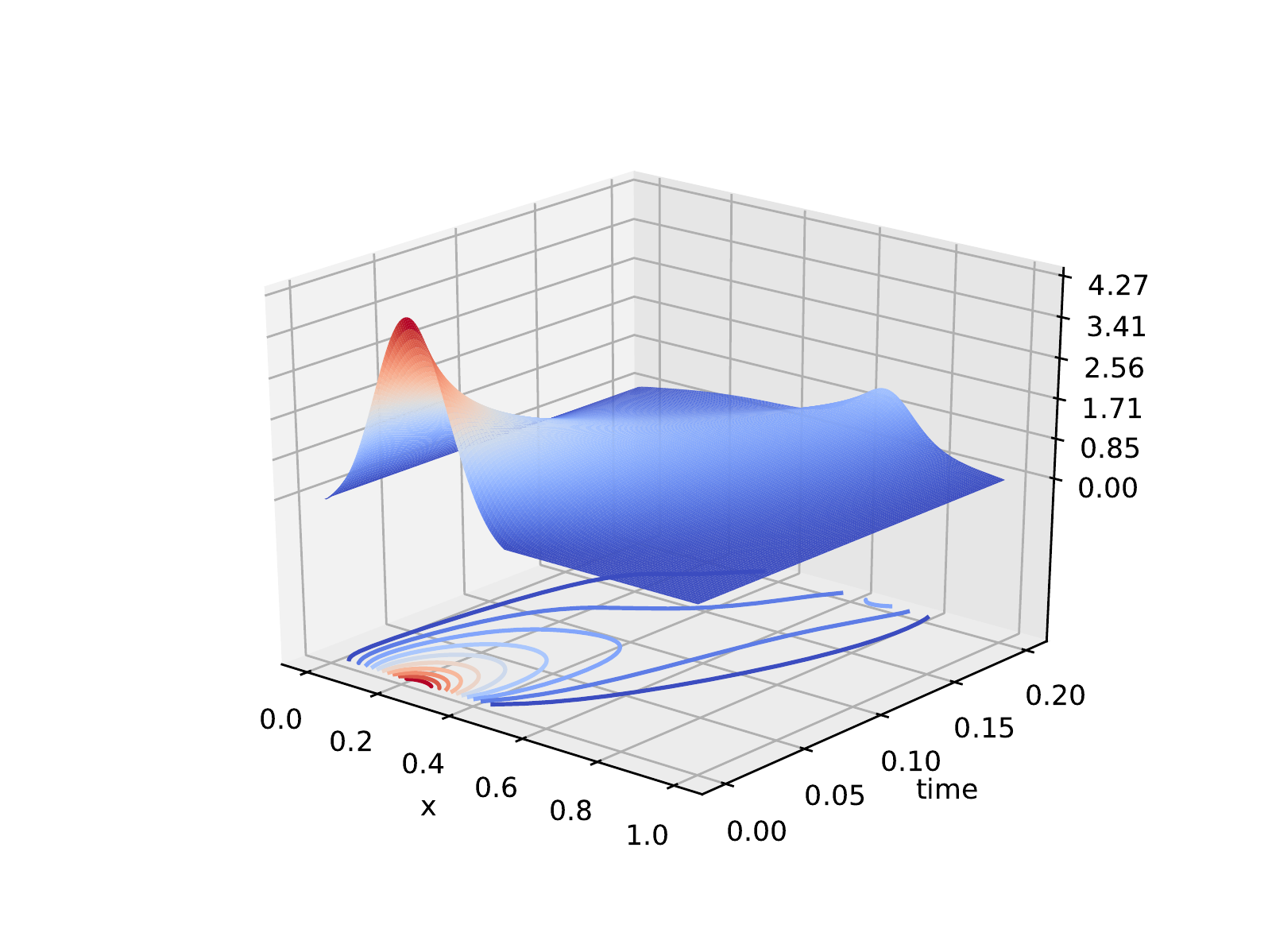}
  \includegraphics[width=0.48\linewidth]{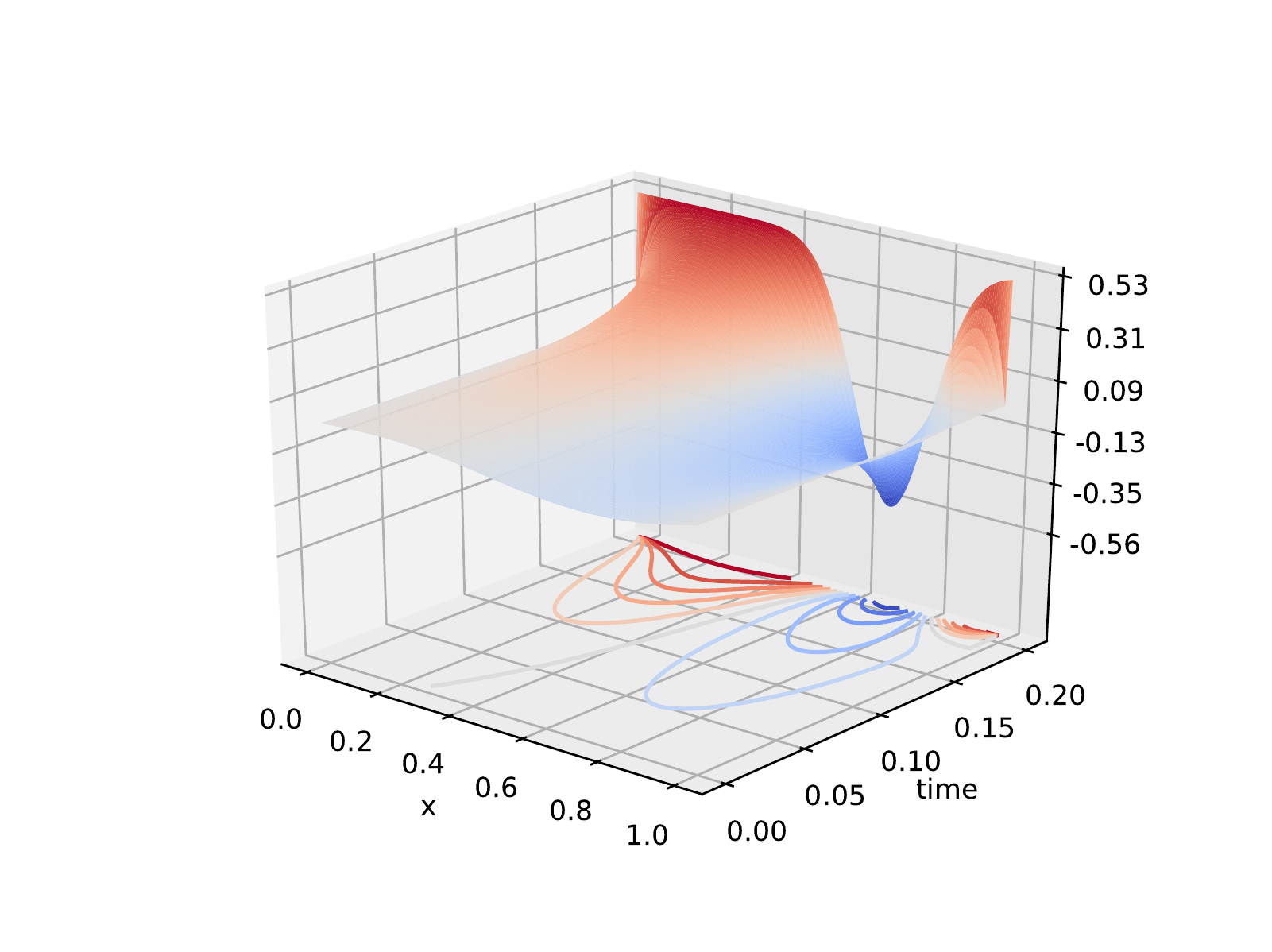}
  \captionof{figure}{Case 1: evolution of $p$ (left) and $u$ (right), with $T=0.2$.}
\label{fig:shortT-evol-pu}
\end{center}


\begin{center}
   \includegraphics[width=0.48\linewidth]{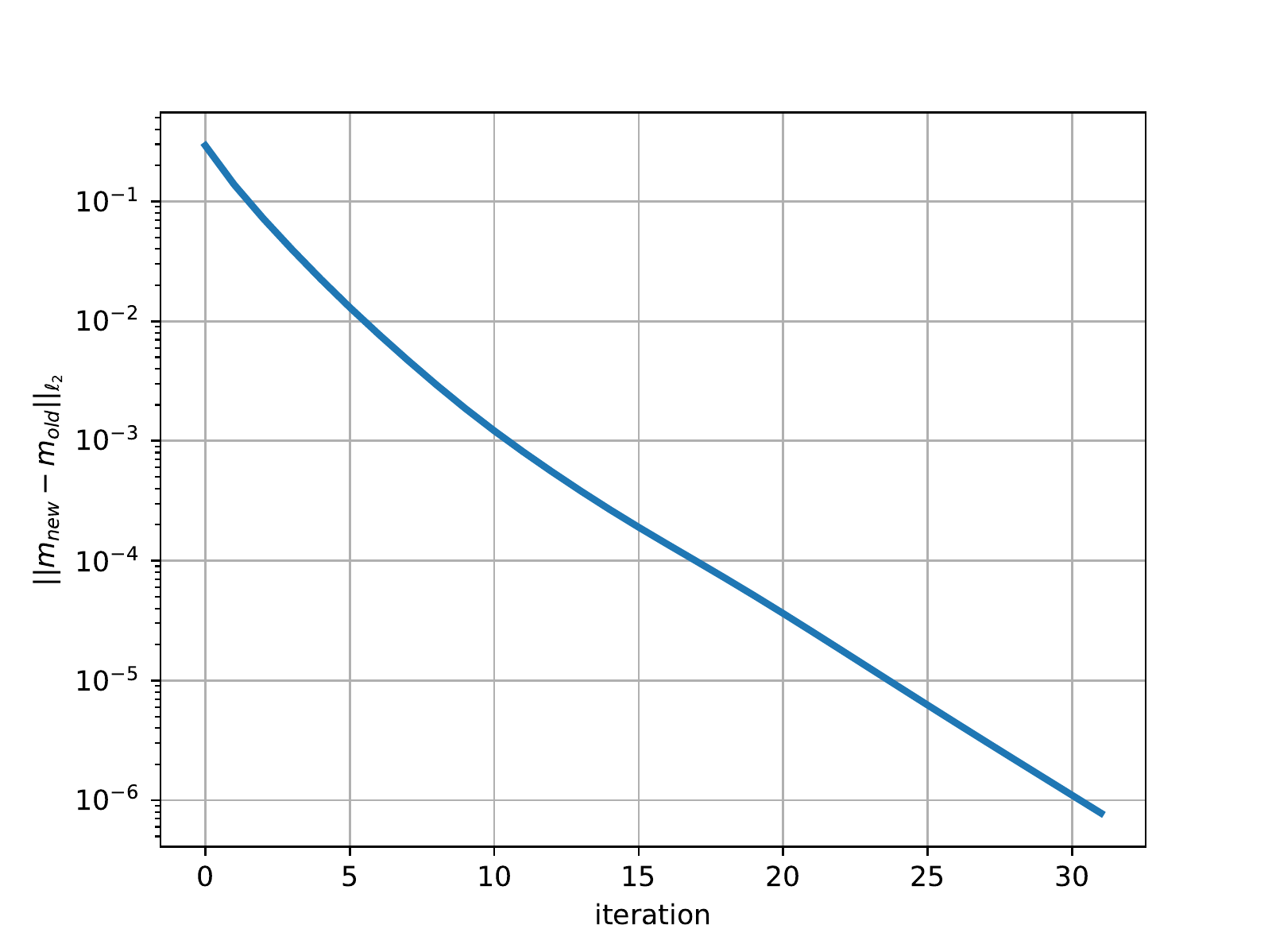}  \includegraphics[width=0.48\linewidth]{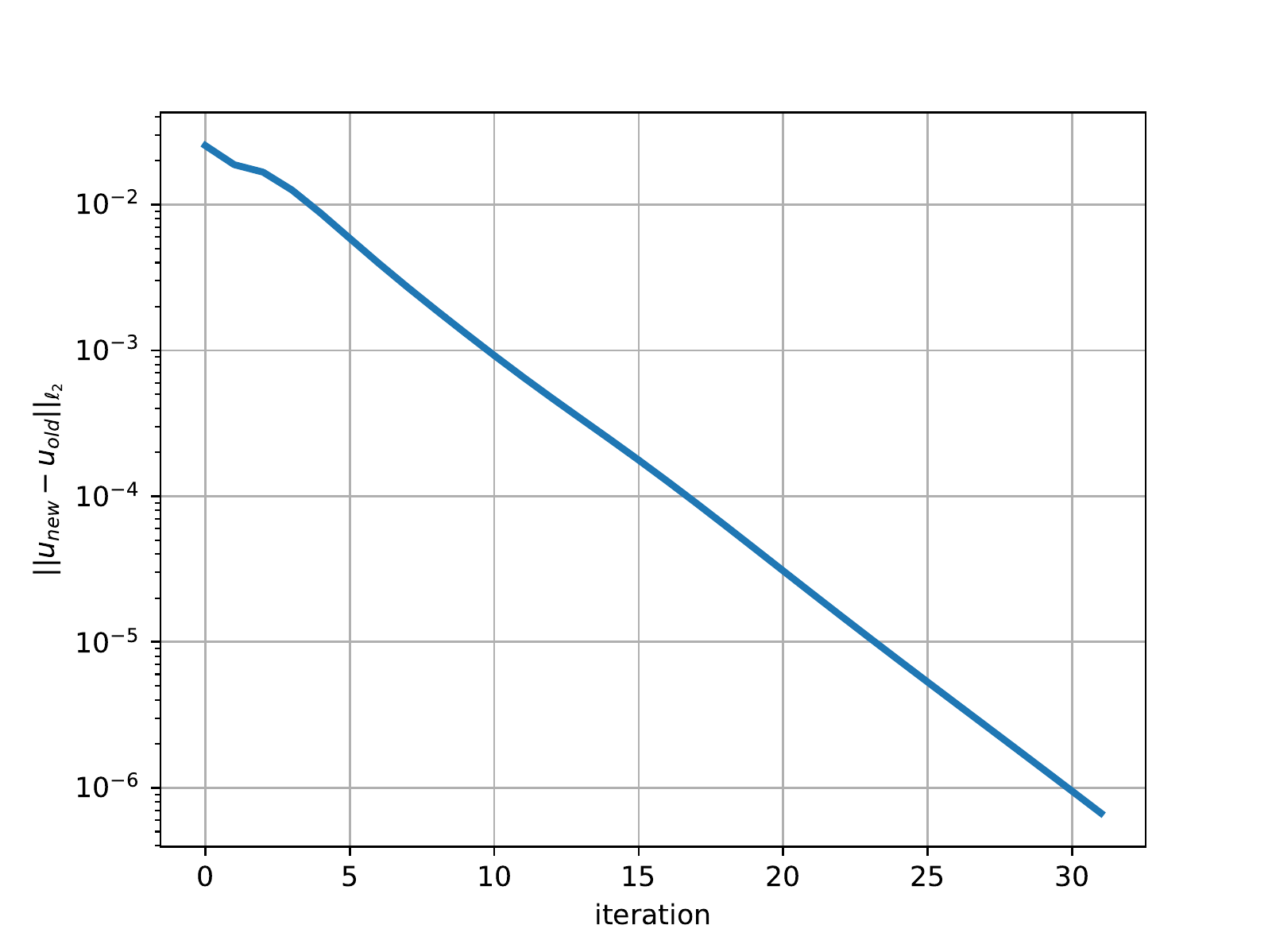}
 \captionof{figure}{Case 1: numerical convergence: normalized $\ell^2$ distance between two iterations, for $p$ (left) and $u$ (right), with respect to the number of iterations.}
\label{fig:time-conv-iter}
\end{center}

\subsubsection{Case 2: same problem with $T=2.0$}

Figure~\ref{fig:longT-evol-pu} displays the evolution of $u$ and $p$ with the same cost functions and initial distribution as above but with a much larger time horizon. The mass decays to $0$ at an exponential rate but has a certain structure even at the final time, as shown in Figure~\ref{fig:longT-mass-final}. The larger $T$, the smaller $\int_\Omega p(T,x) dx$ and hence the larger the final condition~\eqref{eq:terminal-cond-u-g}. This leads to numerical difficulties for very large time horizons. A possible approach is to rely on a continuation method, that is, to use the solution for a shorter time horizon to initialize the iterative procedure for a longer time horizon. 
However this method is time consuming since it requires solving a number of intermediate problems. We will come back to this matter in Section~\ref{sec:alternative-method}. For the results presented in Figures~\ref{fig:longT-evol-pu}--\ref{fig:longT-mass-final}, we have used $h = 5 \times 10^{-4}$ and $\Delta t = 2 \times 10^{-4}$.

\begin{center}
  \includegraphics[width=0.48\linewidth]{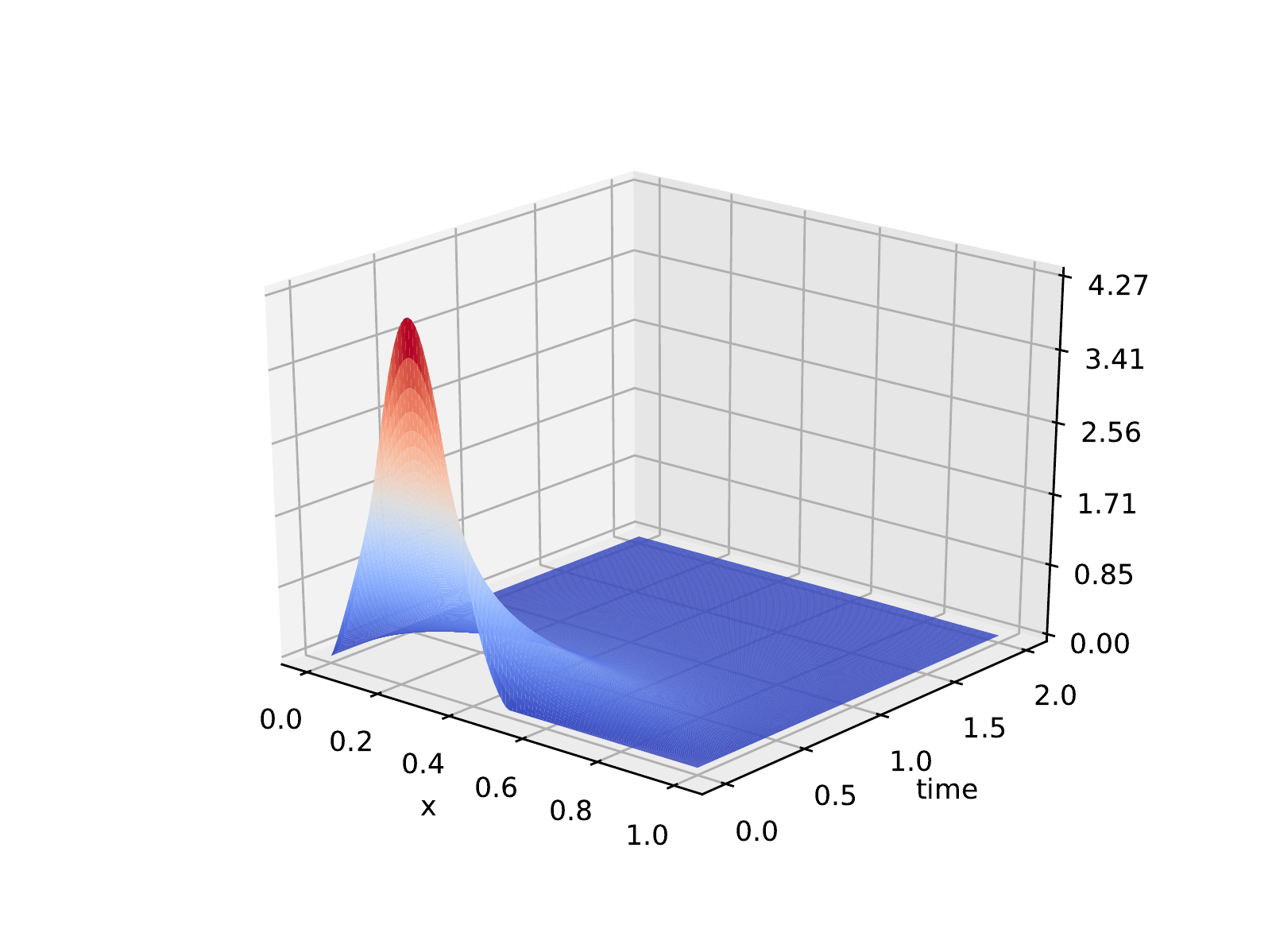}
  \includegraphics[width=0.48\linewidth]{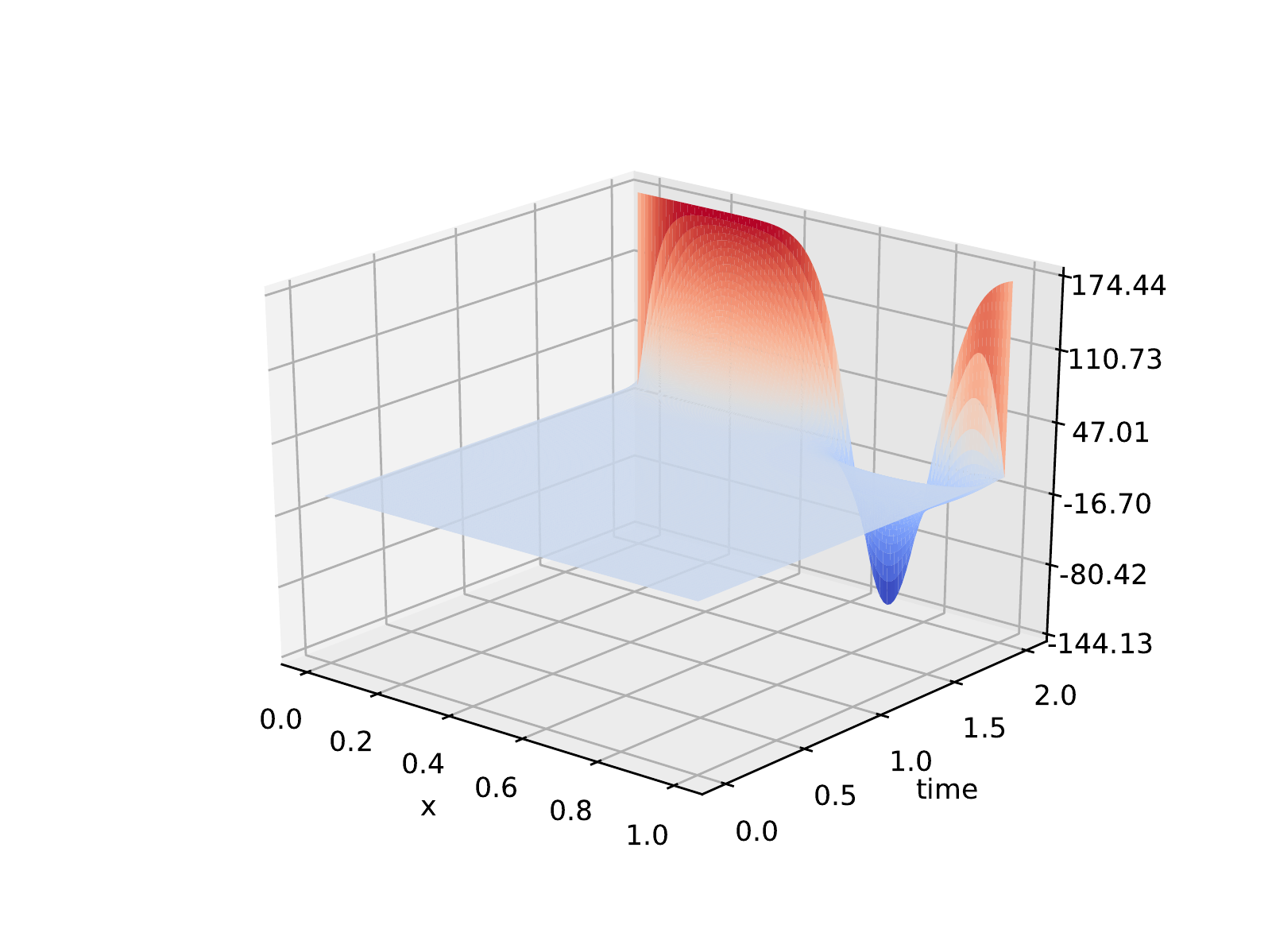}
  \captionof{figure}{Case 2: evolution of $p$ (left) and $u$ (right), with $T=2.0$.}
\label{fig:longT-evol-pu}
\end{center}


\begin{center}
   \includegraphics[width=0.48\linewidth]{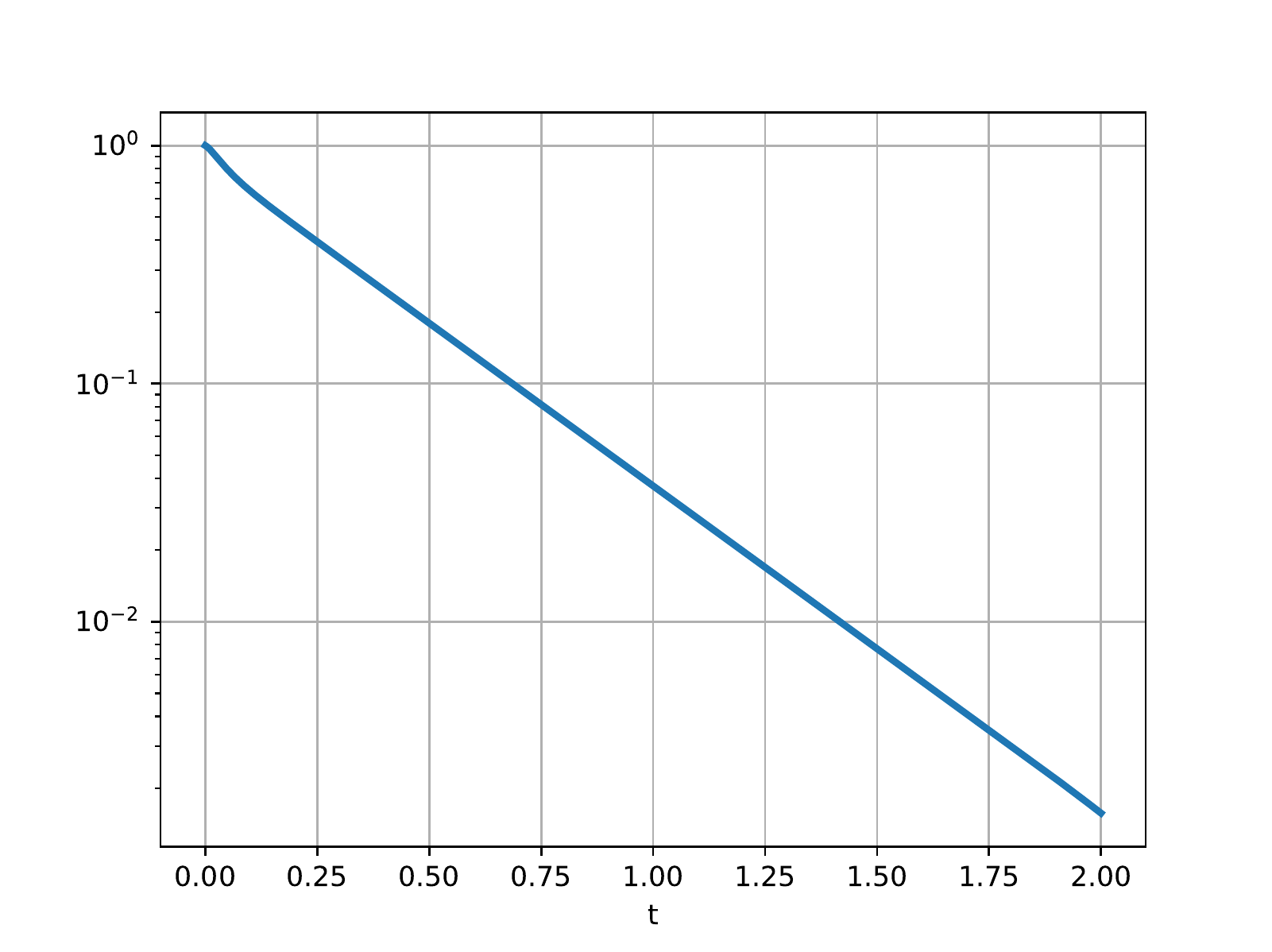}  \includegraphics[width=0.48\linewidth]{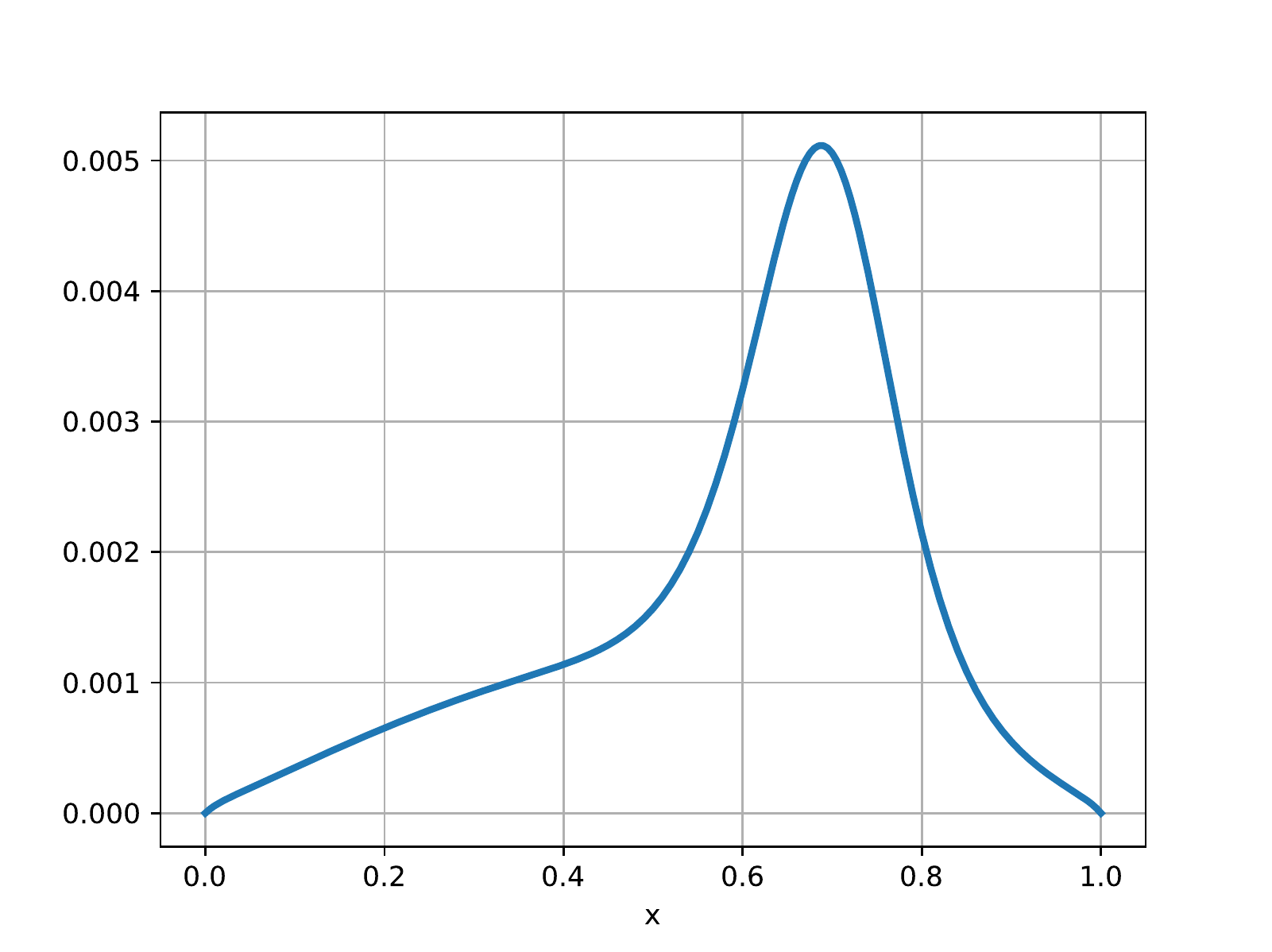}
 \captionof{figure}{Case 2: evolution of $t \mapsto \int_\Omega p(t,x) dx$  (left), and graph of $p(T,\cdot)$ at final time $T=2.0$ (right).}
\label{fig:longT-mass-final}
\end{center}

\subsubsection{Results in 2D}

The numerical method described in Section~\ref{sec:numerics-time} can be applied beyond dimension $1$. Here, we provide results for two test cases in dimension $2$. We take $\Omega = (0,1)^2$ and $T = 0.2$.  For the results presented in this paragraph, we have used $h = 1/80 = 0.0125$ in both dimensions of space, and $\Delta t = 5 \times 10^{-3}$.

\noindent\textbf{Case 3: First test case in dimension $2$.} The running cost is quadratic in the control. The initial distribution and the terminal cost are displayed in Figure~\ref{fig:2D-shortT-gTp0}: at the beginning the mass is concentrated around the center of the domain, but the final cost discourages to stay there. The evolution of the density is shown in Figure~\ref{fig:2D-shortT-evolp}.

\begin{figure}
\centering
\begin{subfigure}{.45\textwidth}
  \centering
  \includegraphics[width=\linewidth]{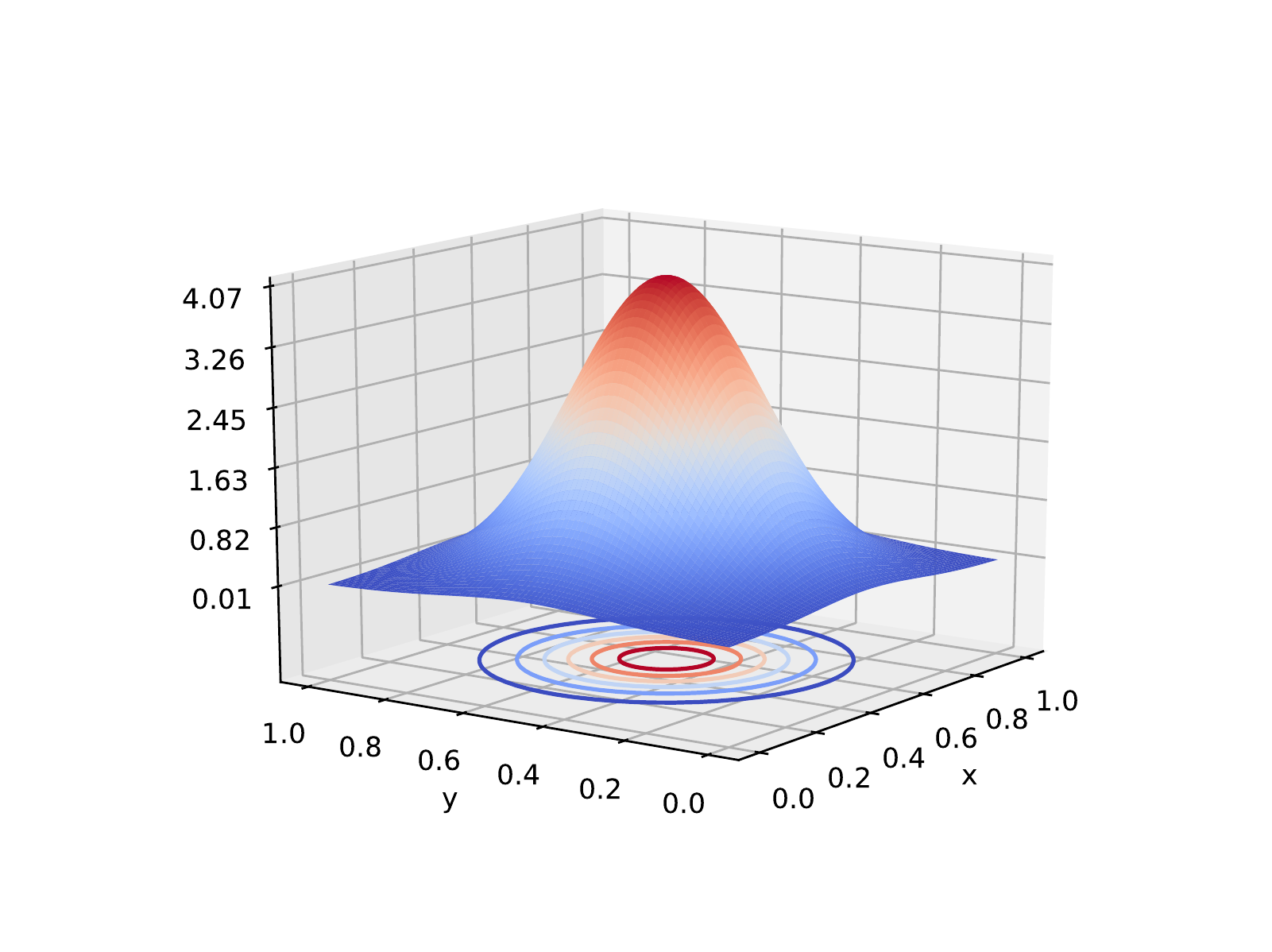}
\end{subfigure}%
\begin{subfigure}{.45\textwidth}
  \centering
  \includegraphics[width=\linewidth]{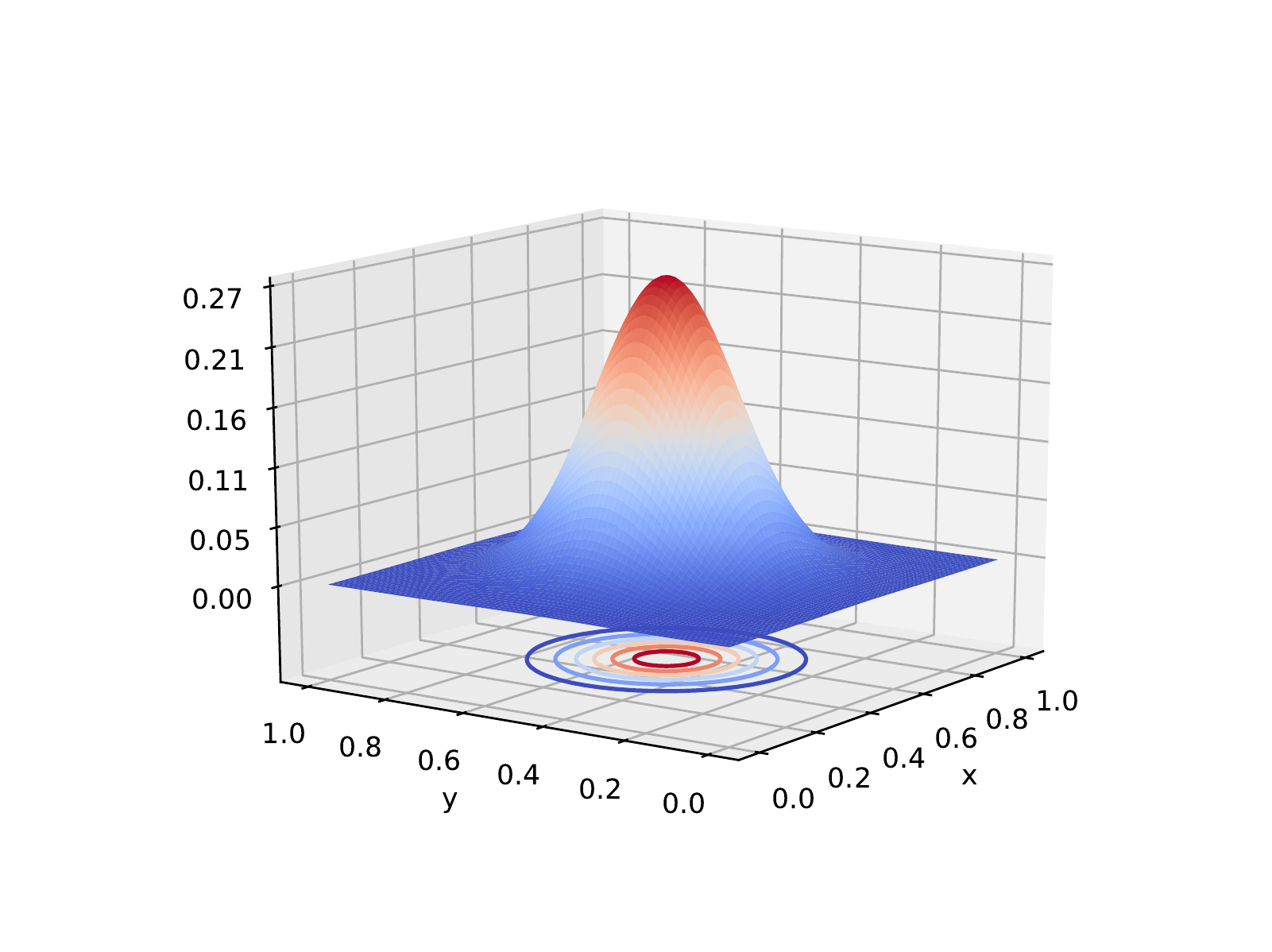}
\end{subfigure}
\caption{Case 3: data of the problem: initial density $p_0$ (left) and terminal cost $g_T$ (right).}
\label{fig:2D-shortT-gTp0}
\end{figure}

\begin{figure}
\centering
\begin{subfigure}{.45\textwidth}
  \centering
  \includegraphics[width=\linewidth]{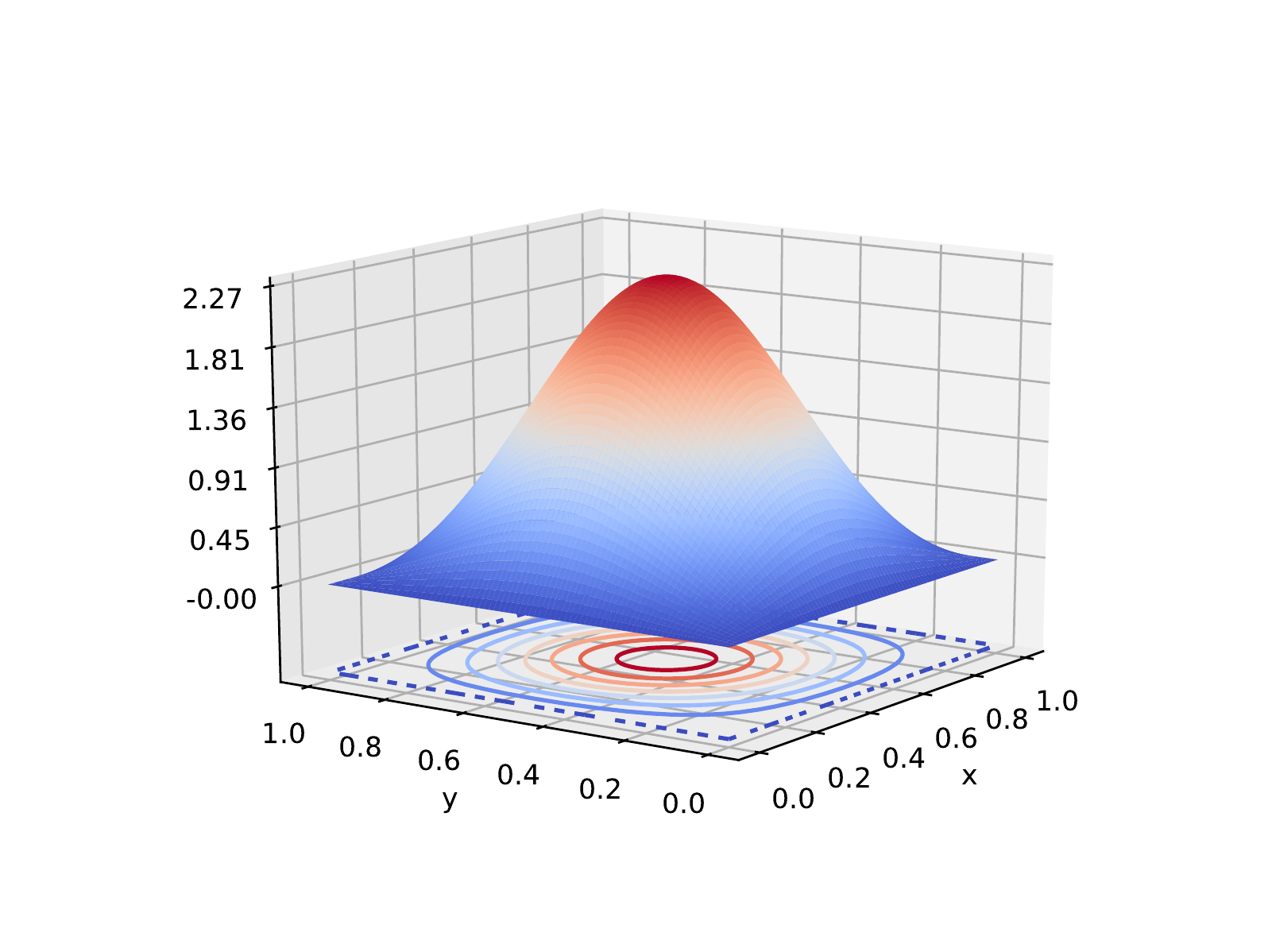}
  \caption*{$t = 0.05$}
\end{subfigure}%
\begin{subfigure}{.45\textwidth}
  \centering
  \includegraphics[width=\linewidth]{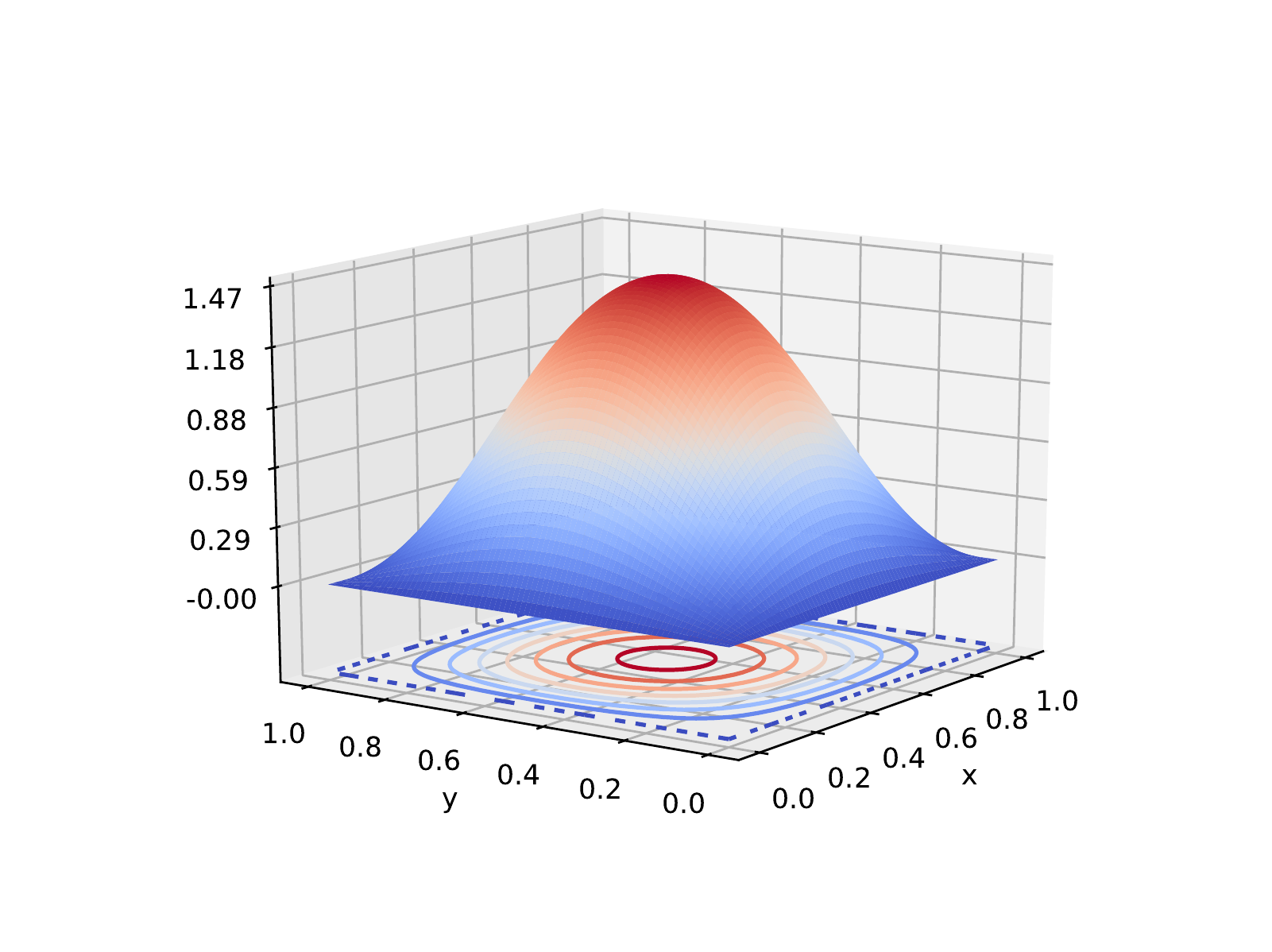}
  \caption*{$t = 0.1$}
\end{subfigure}
\\
\begin{subfigure}{.45\textwidth}
  \centering
  \includegraphics[width=\linewidth]{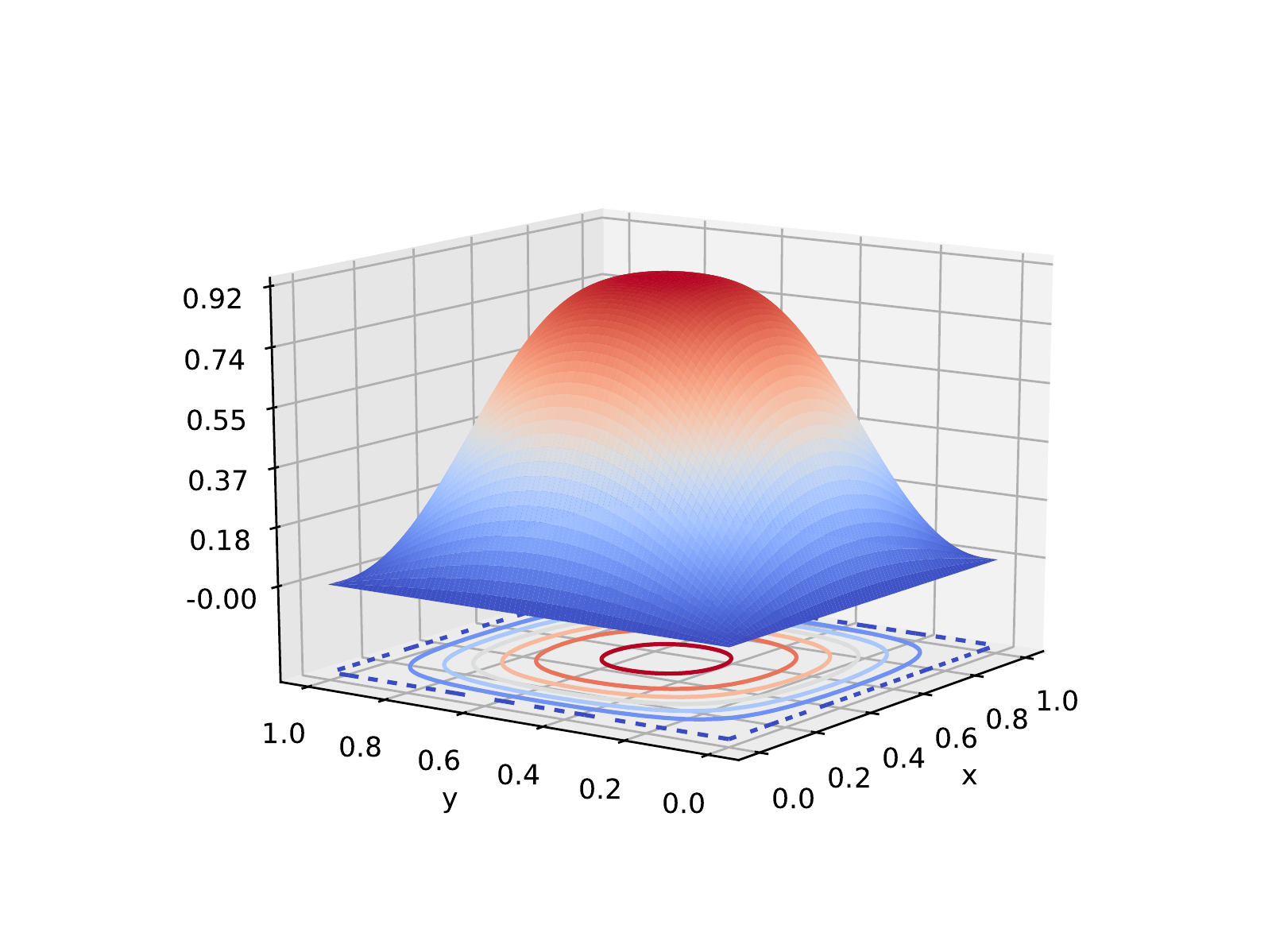}
  \caption*{$t = 0.15$}
\end{subfigure}%
\begin{subfigure}{.45\textwidth}
  \centering
  \includegraphics[width=\linewidth]{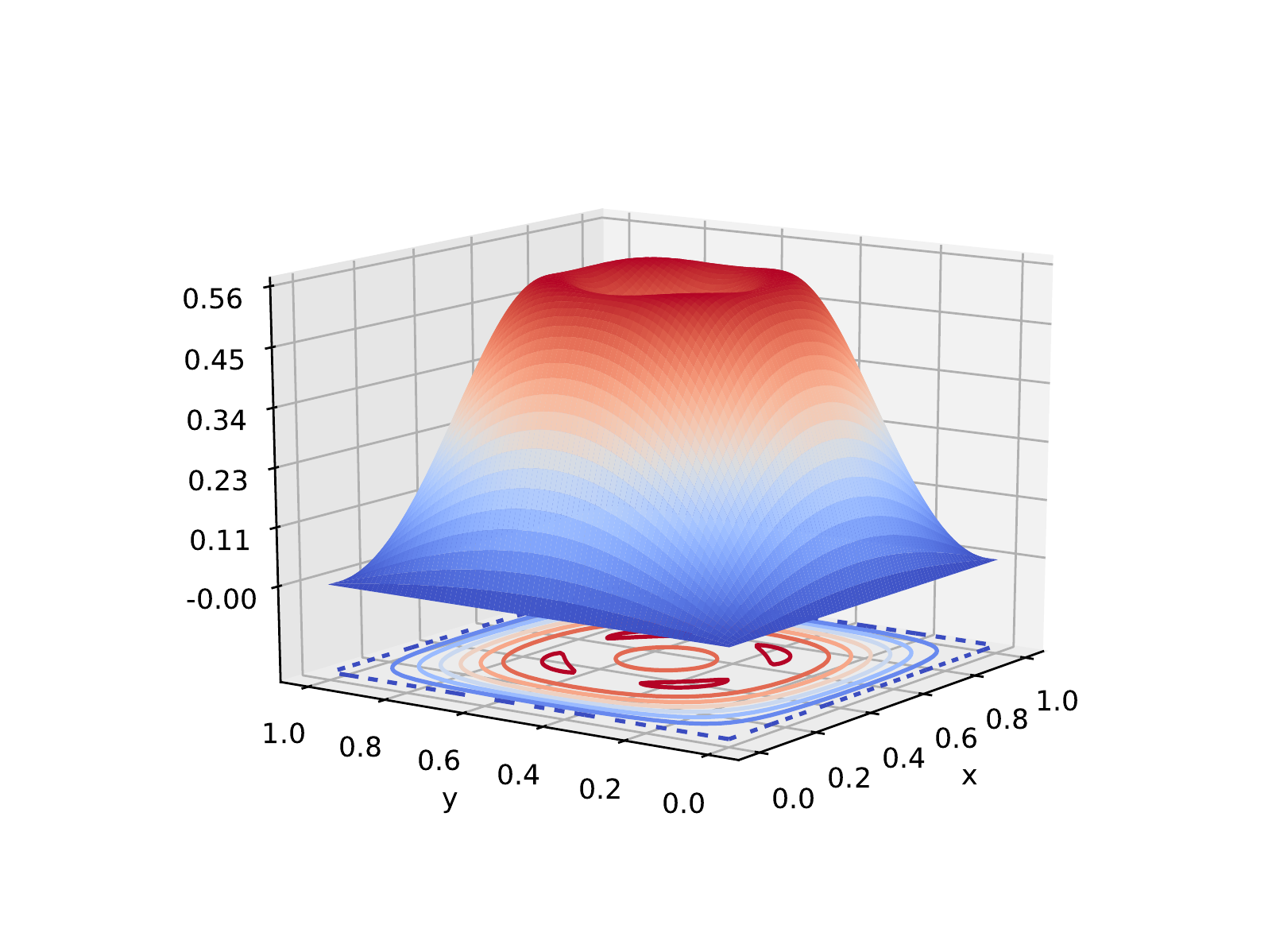}
  \caption*{$t = 0.2$}
\end{subfigure}
\caption{Case 3: evolution of the density.}
\label{fig:2D-shortT-evolp}
\end{figure}

\noindent\textbf{Case 4: Second test case in dimension $2$.} The running cost is quadratic in the control. The initial distribution and the terminal cost are displayed in Figure~\ref{fig:2D-B-shortT-gTp0}: at the beginning the mass is concentrated around the center of the domain, but the final cost attracts it towards two points. The evolution of the density is shown in Figure~\ref{fig:2D-B-shortT-evolp}. 

\begin{figure}
\centering
\begin{subfigure}{.45\textwidth}
  \centering
  \includegraphics[width=\linewidth]{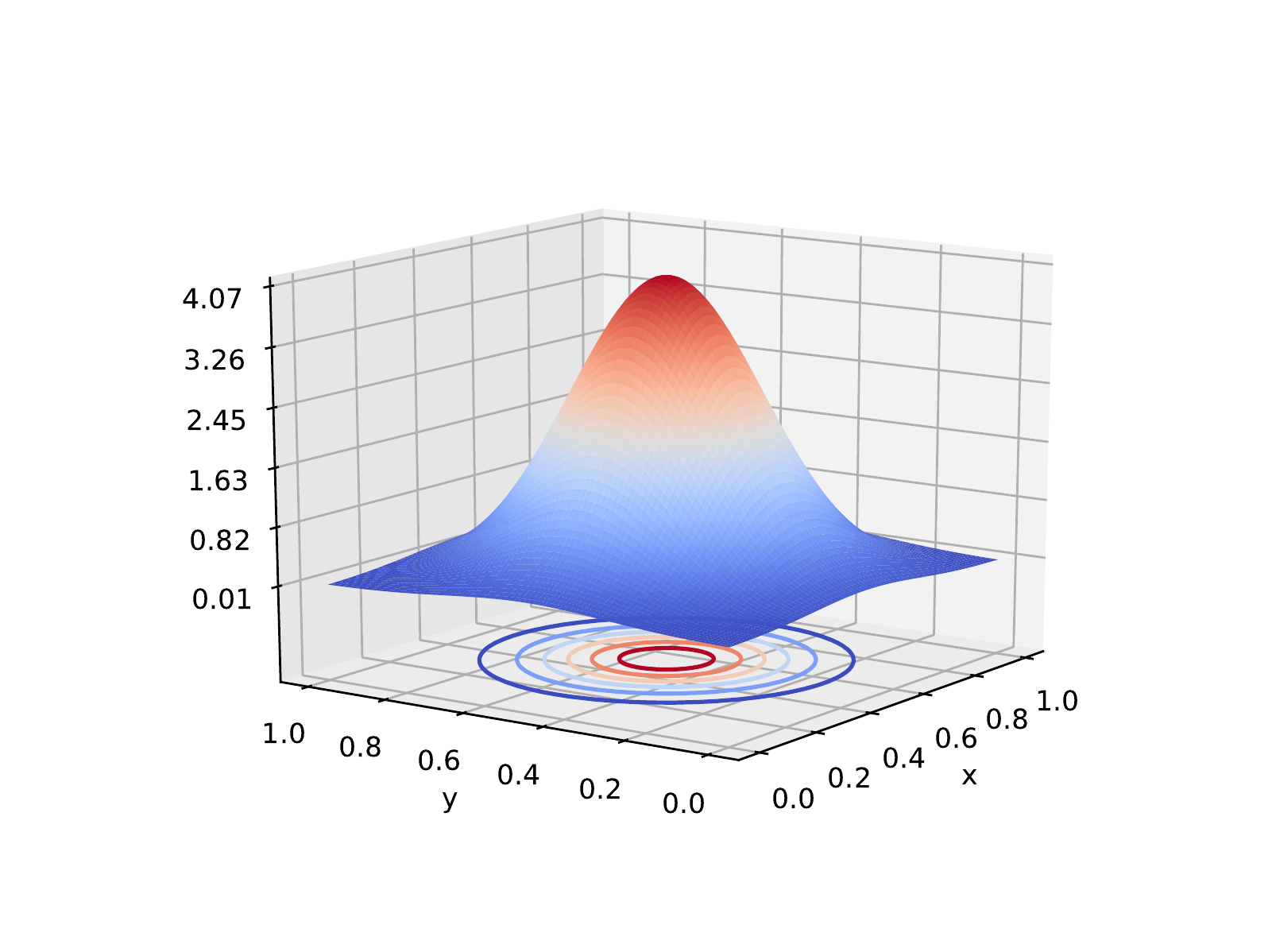}
\end{subfigure}%
\begin{subfigure}{.45\textwidth}
  \centering
  \includegraphics[width=\linewidth]{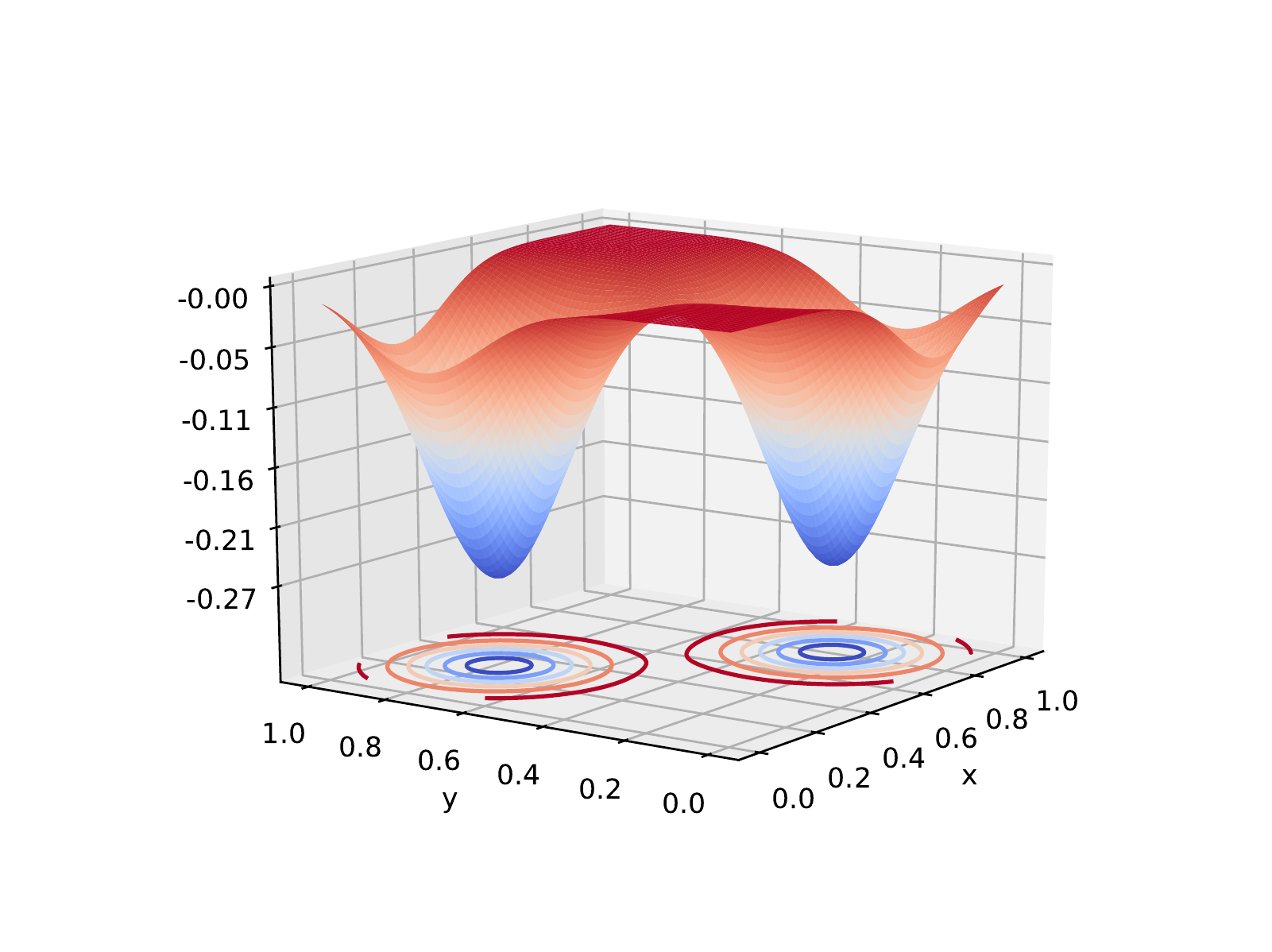}
\end{subfigure}
\caption{Case 4: Second test case in dimension $2$. Data of the problem: initial density $p_0$ (left) and terminal cost $g_T$ (right).}
\label{fig:2D-B-shortT-gTp0}
\end{figure}

\begin{figure}
\centering
\begin{subfigure}{.45\textwidth}
  \centering
  \includegraphics[width=\linewidth]{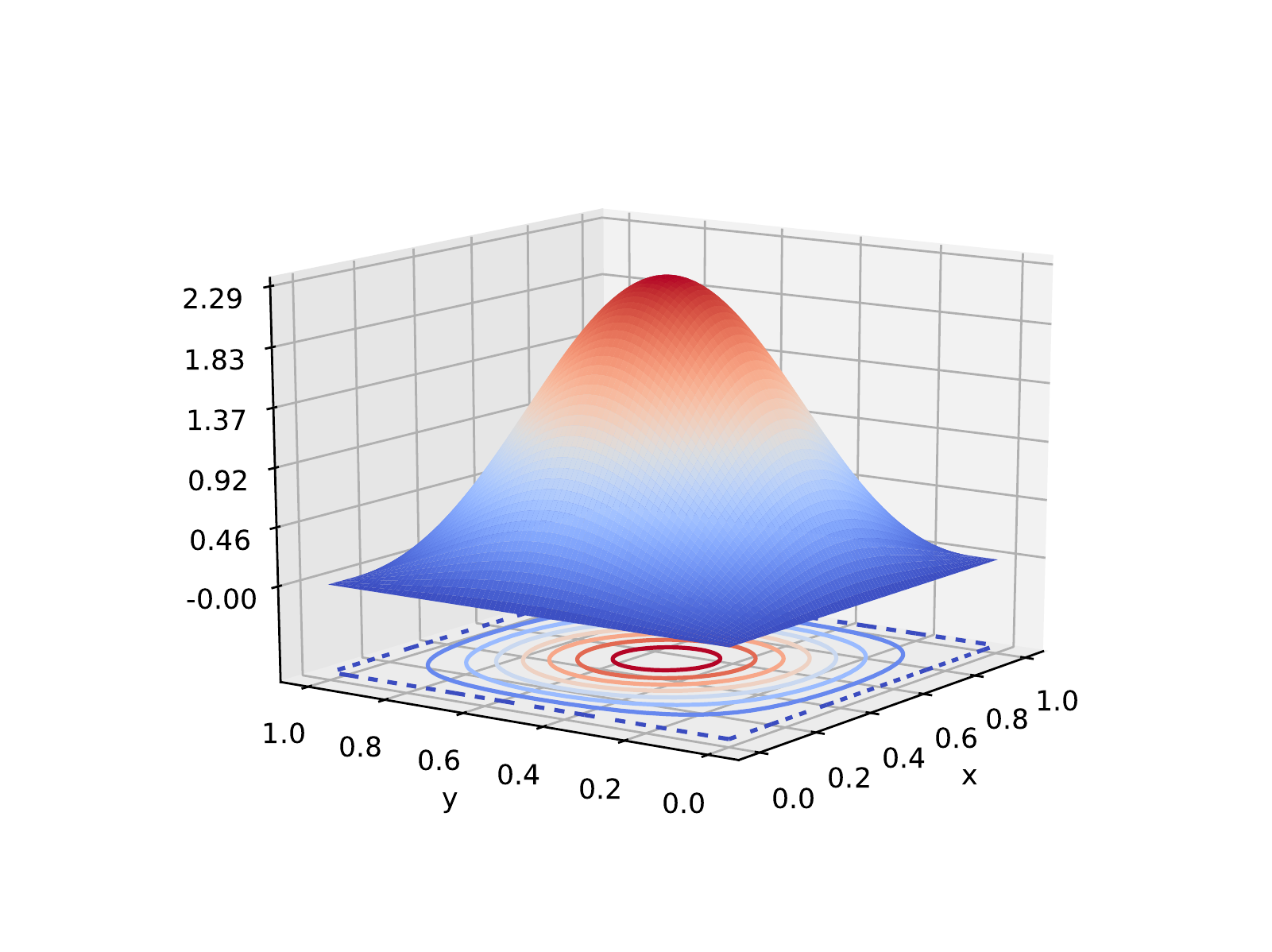}
  \caption*{$t = 0.05$}
\end{subfigure}%
\begin{subfigure}{.45\textwidth}
  \centering
  \includegraphics[width=\linewidth]{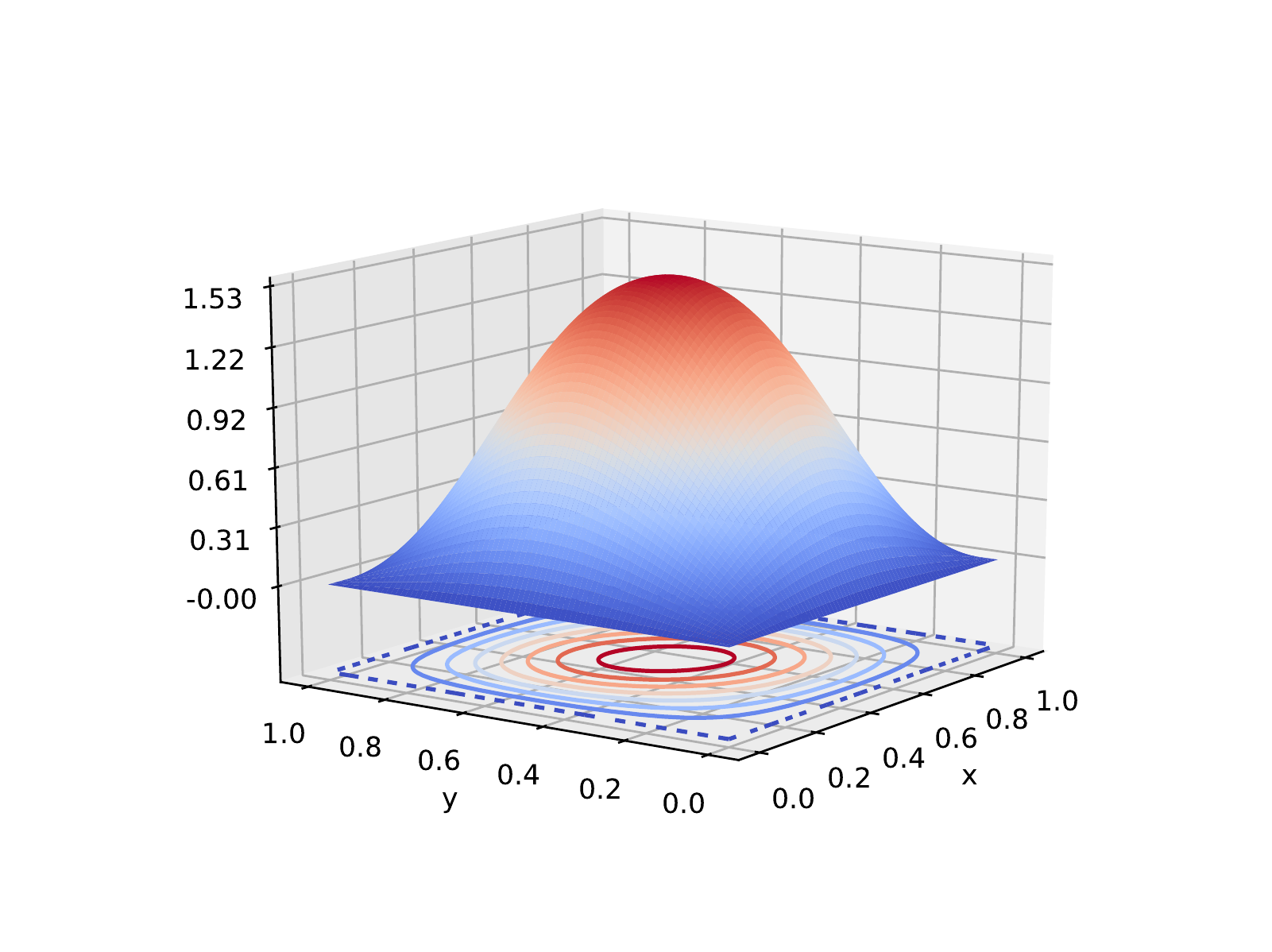}
  \caption*{$t = 0.1$}
\end{subfigure}
\\
\begin{subfigure}{.45\textwidth}
  \centering
  \includegraphics[width=\linewidth]{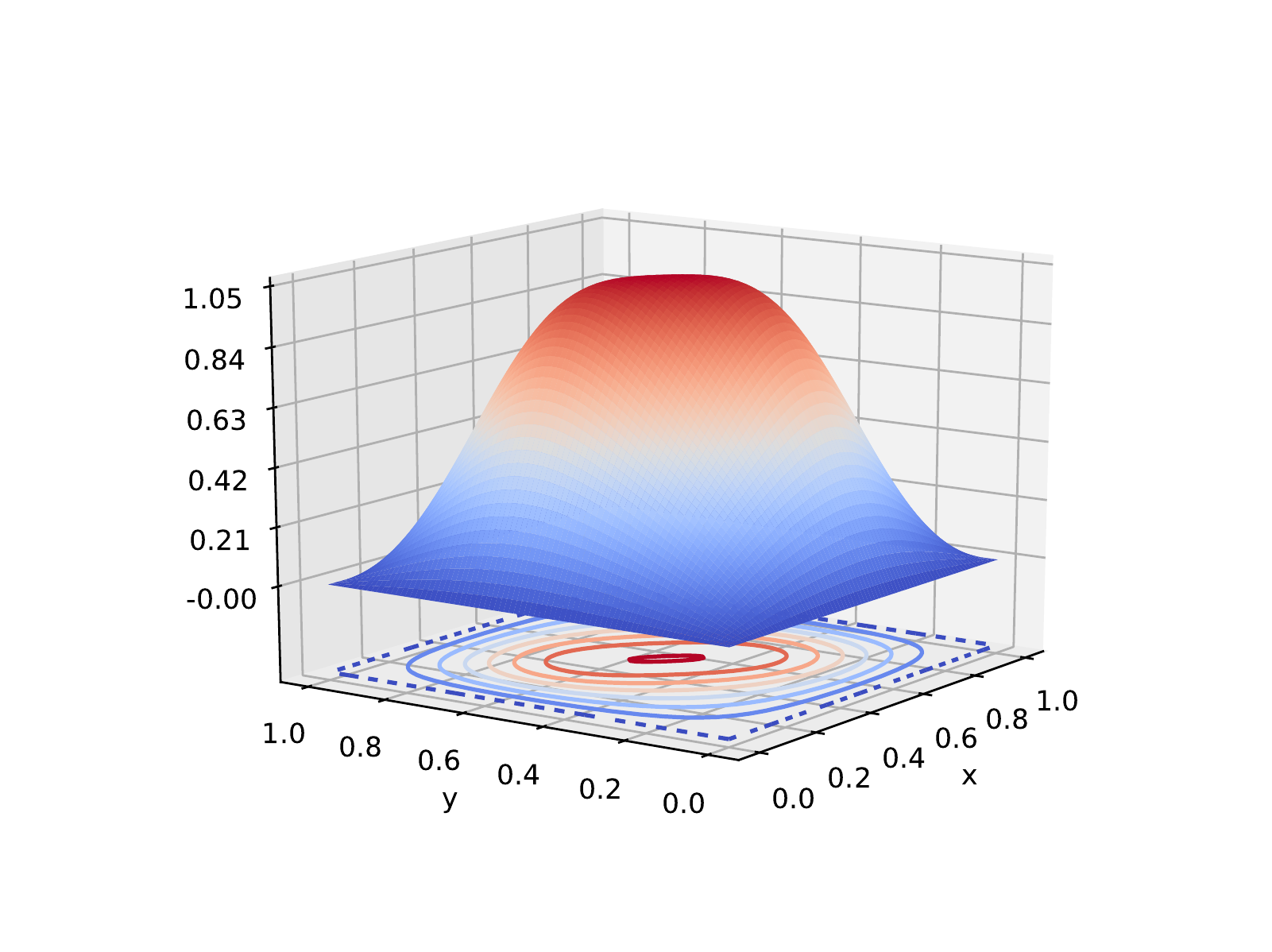}
  \caption*{$t = 0.15$}
\end{subfigure}%
\begin{subfigure}{.45\textwidth}
  \centering
  \includegraphics[width=\linewidth]{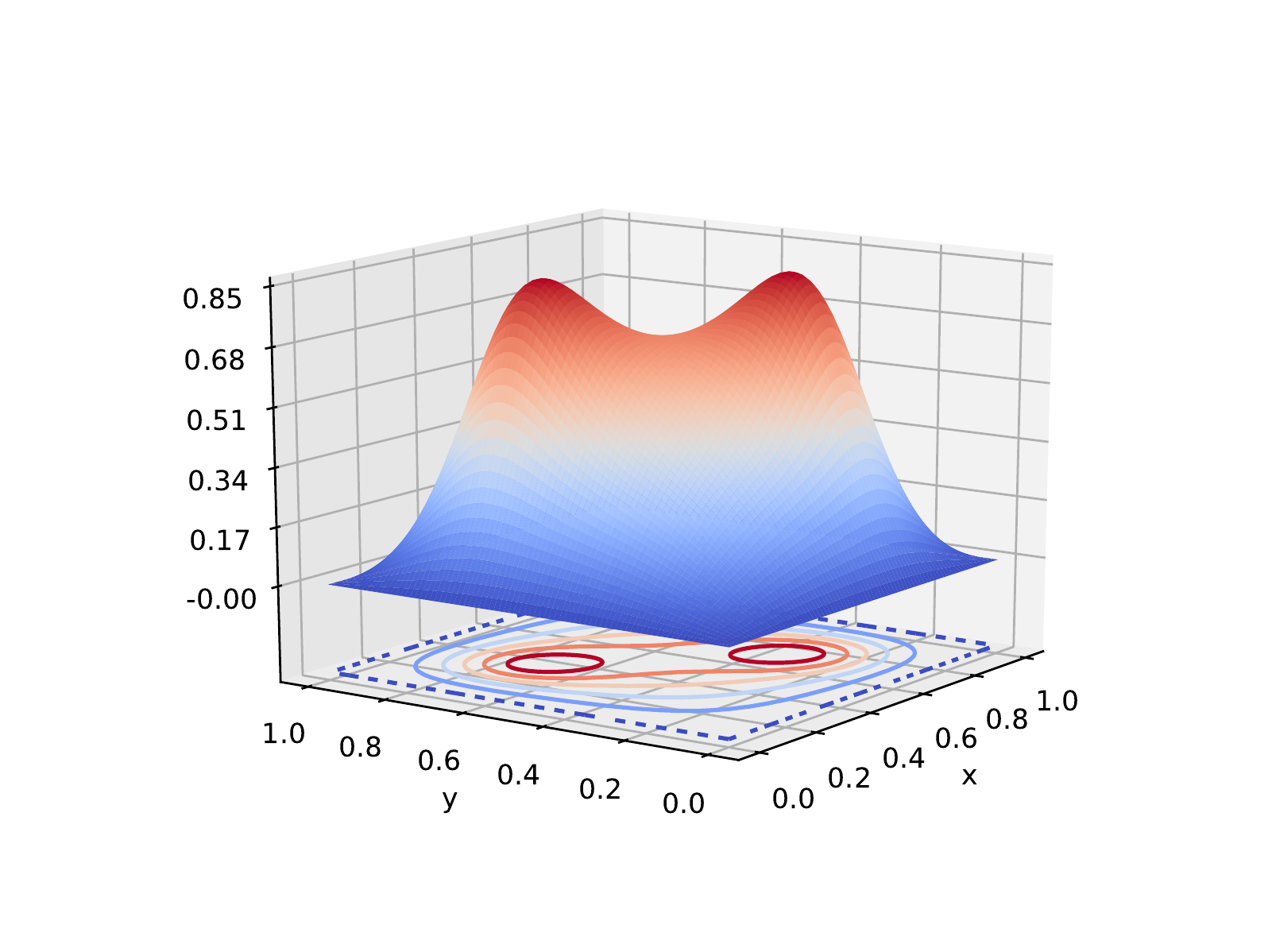}
  \caption*{$t = 0.2$}
\end{subfigure}
\caption{Case 4: Second test case in dimension $2$. Evolution of the density.}
\label{fig:2D-B-shortT-evolp}
\end{figure}

\section{Numerical method for the stationary problem}
\label{sec:numer-meth-stat}
\subsection{Numerical method}

To solve~\eqref{eq:S-FP}--\eqref{eq:S-HJB}, we first notice that we can replace this system by the system composed of~\eqref{eq:S-FP} which we rewrite for convenience
\begin{equation}
  \label{eq:S-FP-equiv} 
  \left\{
    \begin{array}[c]{rcll}
 -\frac{\sigma^2 } 2 \Delta  \widecheck p -\diver\left( \widecheck p    H_\xi\left (\cdot ,  D\widecheck u \right)   \right)   &=&\lambda \widecheck p,    \quad  \quad  \quad & \hbox{ in } \Omega,\\
\widecheck p&=&0    \quad  & \hbox{ on } \partial \Omega,\\
 \widecheck p &\ge& 0   \quad &\hbox{ in } \Omega,\\
\ds \int_\Omega \widecheck p &=&1.&
    \end{array}
\right.
\end{equation}
and the following modified HJB equation:
\begin{equation}
\label{eq:S-HJB-equiv}
 \left\{
    \begin{array}[c]{rcl}
 -\frac{\sigma^2 } 2 \Delta  \widecheck u + H(\cdot, D\widecheck u)
 &=&  \ds  \lambda \widecheck u +F[\widecheck p]    +c_1 \mathds{1}_{\Omega} - \epsilon - \int_\Omega \widecheck u \widecheck p, \hfill \hbox{in } \Omega,\\
\widecheck u&=&0, \hfill \hbox{on }\partial \Omega,
\\
  c_1 &  = & -   \left( \int_\Omega \widecheck p   \left( L\left (\cdot,  H_\xi\left (\cdot,   D\widecheck u  \right)  \right)  +F[\widecheck p]\right) dx\right).
    \end{array}
    \right.
  \end{equation}
  
  Indeed~\eqref{eq:S-FP}--\eqref{eq:S-HJB} is equivalent to~\eqref{eq:S-FP-equiv}--\eqref{eq:S-HJB-equiv}: it is straightforward to see that~\eqref{eq:S-FP}--\eqref{eq:S-HJB} implies~\eqref{eq:S-FP-equiv}--\eqref{eq:S-HJB-equiv}; conversely multiplying the HJB equation~\eqref{eq:S-HJB-equiv} by $\widecheck p$ and the Fokker-Planck equation~\eqref{eq:S-FP-equiv} by $\widecheck u$, integrating on $\Omega$, summing the resulting equations and using the definition of $c_1$, we see that most of the terms cancel and we obtain 
\begin{equation*}
	\int_\Omega \widecheck u(x) \widecheck p(x) dx = -\epsilon
\end{equation*}
and finally~\eqref{eq:S-FP}--\eqref{eq:S-HJB}.

This allows us to use the same kind of iterative algorithm as for the non stationary problem (see Section~\ref{sec:numerics-time}), except that, in order to solve~\eqref{eq:S-FP-equiv}, the linear solver for the FP equation is replaced by an inverse power method.
More precisely, we use the following iterative method, where the solution computed at a given iteration is used to compute the non-local terms involved in the next iteration. 
We will use the notations $\widetilde H$ and $\cF$ introduced in Section~\ref{subsec:finite-hor-num-meth} for the finite time horizon problem. The iterative procedure we consider is:

\begin{enumerate}
	\item Start with a guess $(\widecheck P^{(0)}, \widecheck U^{(0)}, \lambda^{(0)})$; set $k \leftarrow 0$.
	\item Repeat the following steps
		\begin{enumerate}
			\item Compute $\widecheck U^{(k+1)}$, solution to the following modified version of~\eqref{eq:S-HJB-equiv},
			\begin{equation}
			\label{eq:S-HJB-equiv-discrete-kp1}
			\left\{\quad
    			\begin{array}[c]{rcll}
			- \frac{\sigma^2}{2} (\Delta_h \widecheck U)_{i}
				+ \widetilde  H\left(x_i, 1, [\grad_h \widecheck U]_i\right)
				&=& \lambda^{(k)} \widecheck U^{(k)}_i + \cF(\widecheck P^{(k)},\widecheck U^{(k)})_i - \epsilon - h \sum_j  \widecheck U^{(k)}_j \widecheck P^{(k)}_j,
			\\
			&& \hfill \hbox{$i \in \{1,\dots,N_h-1\},$}
			\\
			\widecheck U_{i} &=& 0, \hfill \hbox{$i \in \{0, N_h\}.$}
			\end{array}
    			\right.
 			\end{equation}
			\item Compute $(\widecheck P^{(k+1)}, \lambda^{(k+1)})$, solution to the following modified version of~\eqref{eq:S-FP-equiv},
			\begin{equation}
			\label{eq:S-FP-equiv-discrete-kp1}
			\left\{\quad
    			\begin{array}[c]{rcll}
			- \frac{\sigma^2}{2} (\Delta_h \widecheck P)_{i}
				 - \mathcal B_i^{(1, \widecheck U^{(k+1)})}(\widecheck P)
				&=& \lambda \widecheck P_i, 
			\quad &i \in \{1,\dots,N_h-1\}, 
			\\
			\widecheck P_{i} &=& 0, \quad & i \in \{0, N_h\},
			\\
			\widecheck P_{i} &\geq& 0, \quad & i \in \{1,\dots,N_h-1\},
			\\
			\sum_j h \widecheck P_{j} &=& 1.
			\end{array}
    			\right.
 			\end{equation}
			\item If $||\widecheck P^{(k+1)} - \widecheck P^{(k)} ||_{\ell^2}$ and $||\widecheck U^{(k+1)} - \widecheck U^{(k)} ||_{\ell^2}$ are small enough, stop. 
				Otherwise set $k \leftarrow k+1$ and continue.
		\end{enumerate}
\end{enumerate}
To solve~\eqref{eq:S-HJB-equiv-discrete-kp1}, here again we use \texttt{UMFPACK}.
For the second step, namely finding the eigenvalue and eigenvector solving~\eqref{eq:S-FP-equiv-discrete-kp1}, our implementation relies on the \texttt{ARPACK} routines~\cite{MR1621681}, which are based on the Implicitly Restarted Arnoldi Method.

\subsection{Numerical results}

\subsubsection{Stationary solution}
We consider again the one-dimensional test case discussed in Section~\ref{sec:numres-time}, except that we consider a running cost instead of a final cost.
This example will be referred to as Case 5 in what follows.  More precisely, the cost functional is given by~\eqref{eq:56}, with 
$$
	L(x,b) = \frac{1}{2}|b|^2, \qquad \Phi[p] = \int_\Omega p(x) f(x) dx, \qquad f(x) = -\frac{1}{2} \exp\left( \left(x- 0.7\right)^2 / \left(\tfrac{1}{5}\right)^2 \right)
$$
and $\epsilon = 0$, for $x \in \RR, b \in \RR, p \in L^2(\Omega; \RR)$. 
The solution $(p,u)$ is shown in Figure~\ref{fig:statio-pu}. It is compared with the solution at time $T/2$ of the problem with a time horizon $T$, rescaled by the remaining mass (namely, $u$ and $p$ are respectively multiplied and divided by $\int_\Omega p(T/2)$).
 The numerical convergence of the iterative procedure is illustrated by Figure~\ref{fig:statio-conv-iter}. The stopping criterion is based on the normalized $\ell^2$ distance between two successive approximations of $p$ and $u$, where the normalized $\ell^2$ norm is defined, for a vector $V \in \RR^{N_h+1}$, by
$$
	\| V \|_{\ell^2} = \left( h \sum_{i=0}^{N_h} |V^i|^2 \right)^{1/2}.
$$
For the results presented in Figures~\ref{fig:statio-pu}--\ref{fig:statio-conv-iter}, we have used $h = 10^{-3}$.

\begin{figure}
\centering
\begin{subfigure}{.45\textwidth}
  \centering
  \includegraphics[width=\linewidth]{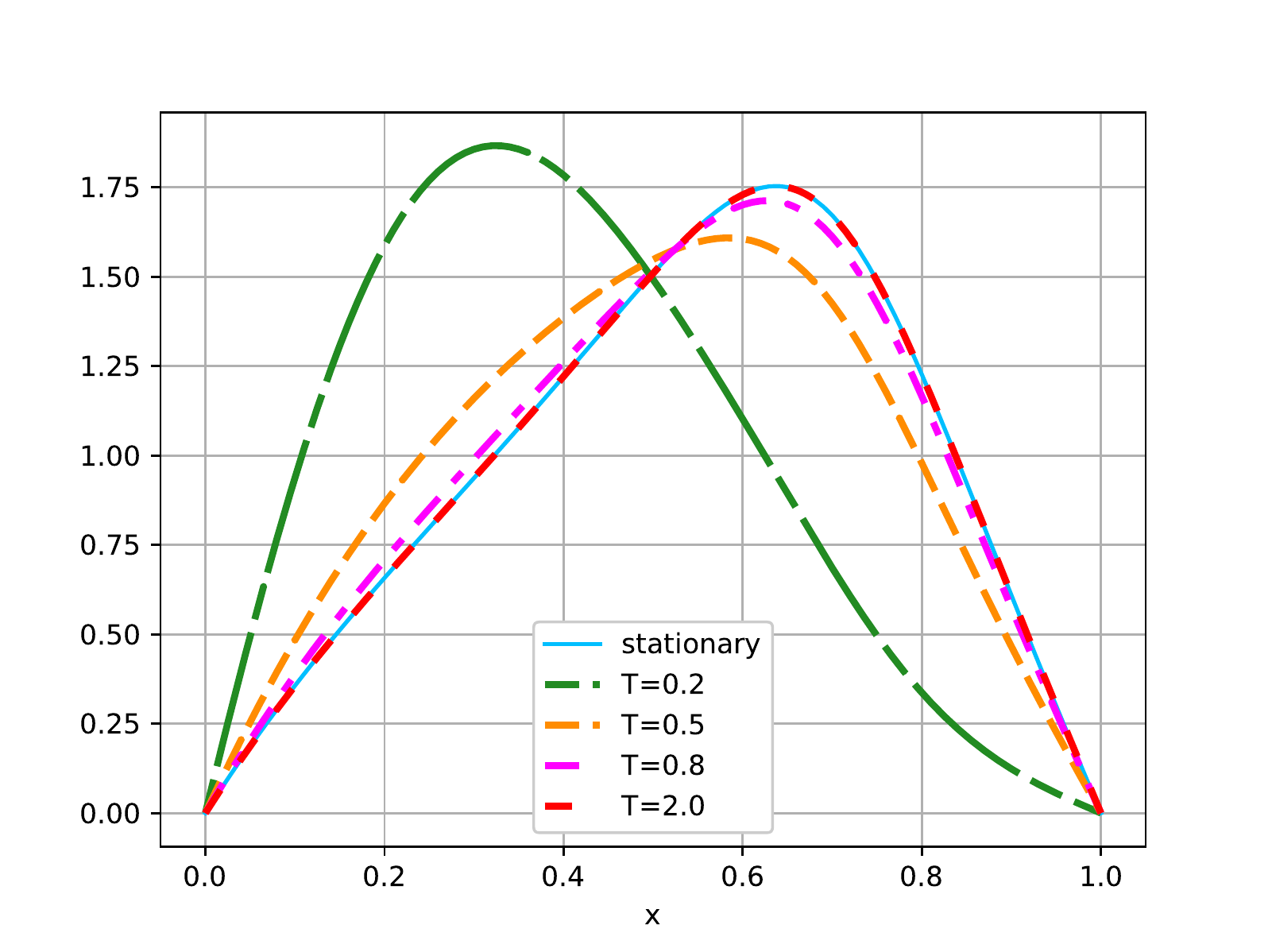}
\end{subfigure}%
\begin{subfigure}{.45\textwidth}
  \centering
  \includegraphics[width=\linewidth]{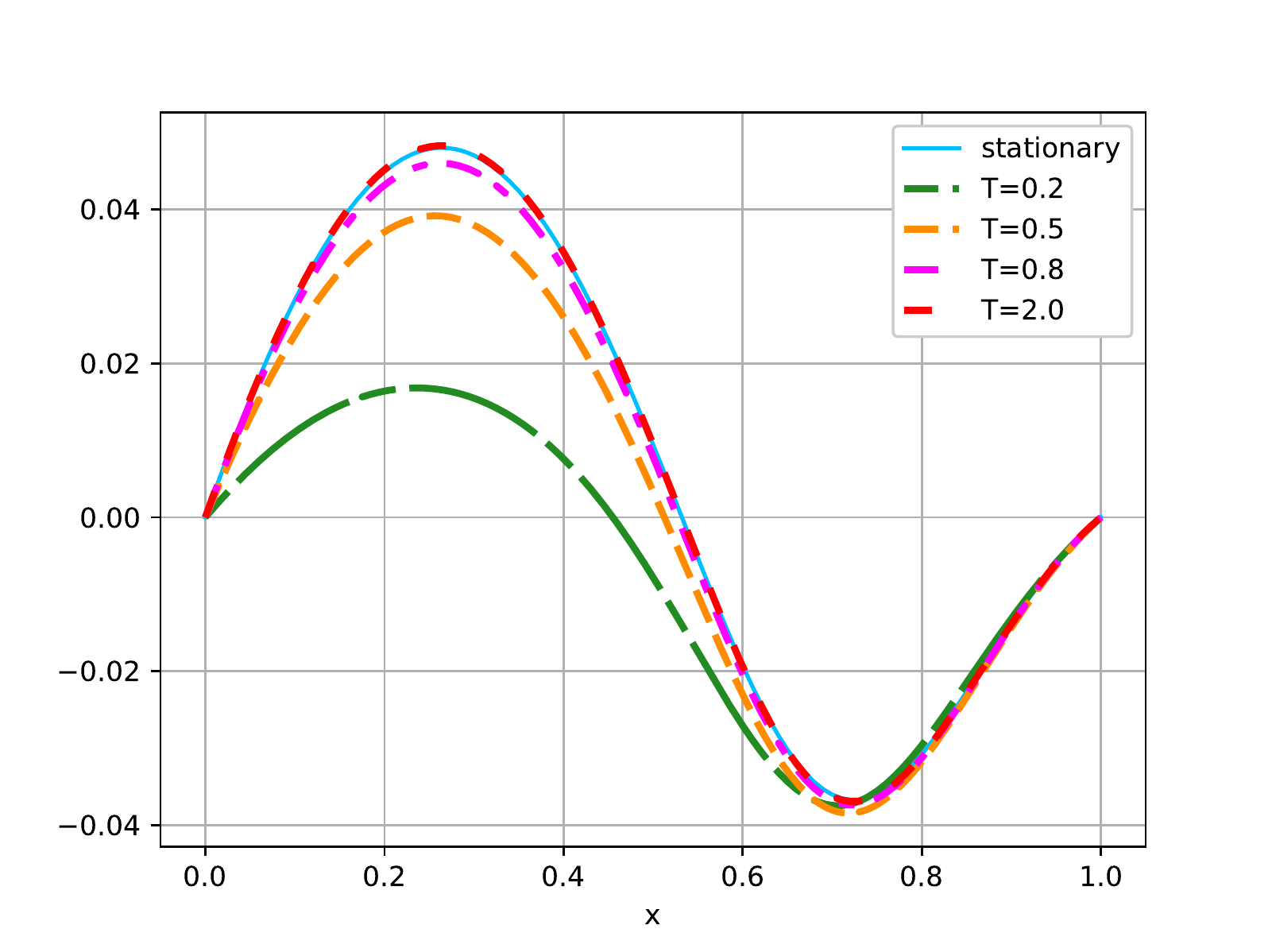}
\end{subfigure}
\caption{Case 5: Solution to the stationary problem: $p$ (left), $u$ (right) and $\lambda \approx 3.15$. }
\label{fig:statio-pu}
\end{figure}

\begin{figure}
\centering
\begin{subfigure}{.45\textwidth}
  \centering
  \includegraphics[width=\linewidth]{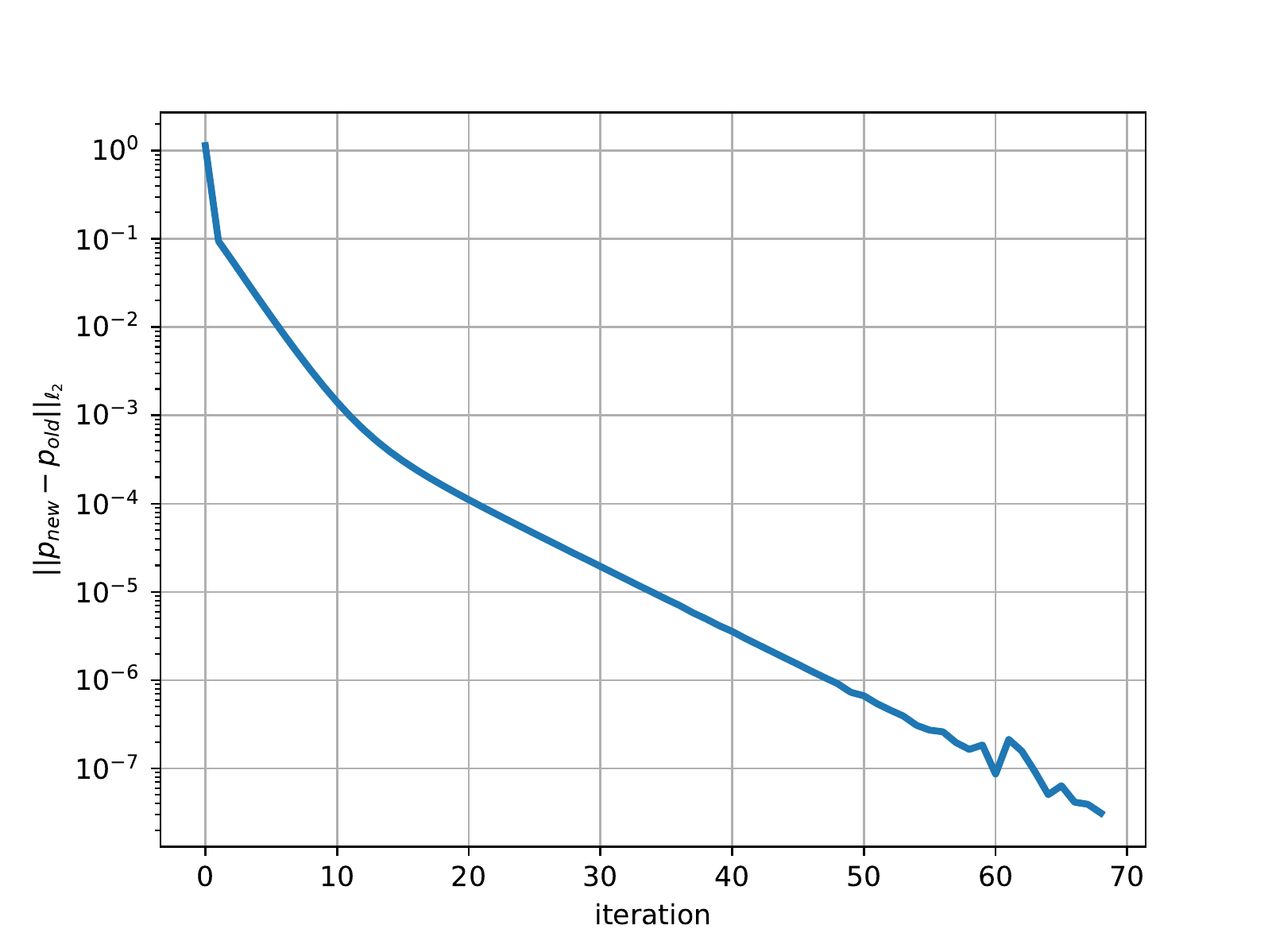}
\end{subfigure}%
\begin{subfigure}{.45\textwidth}
  \centering
  \includegraphics[width=\linewidth]{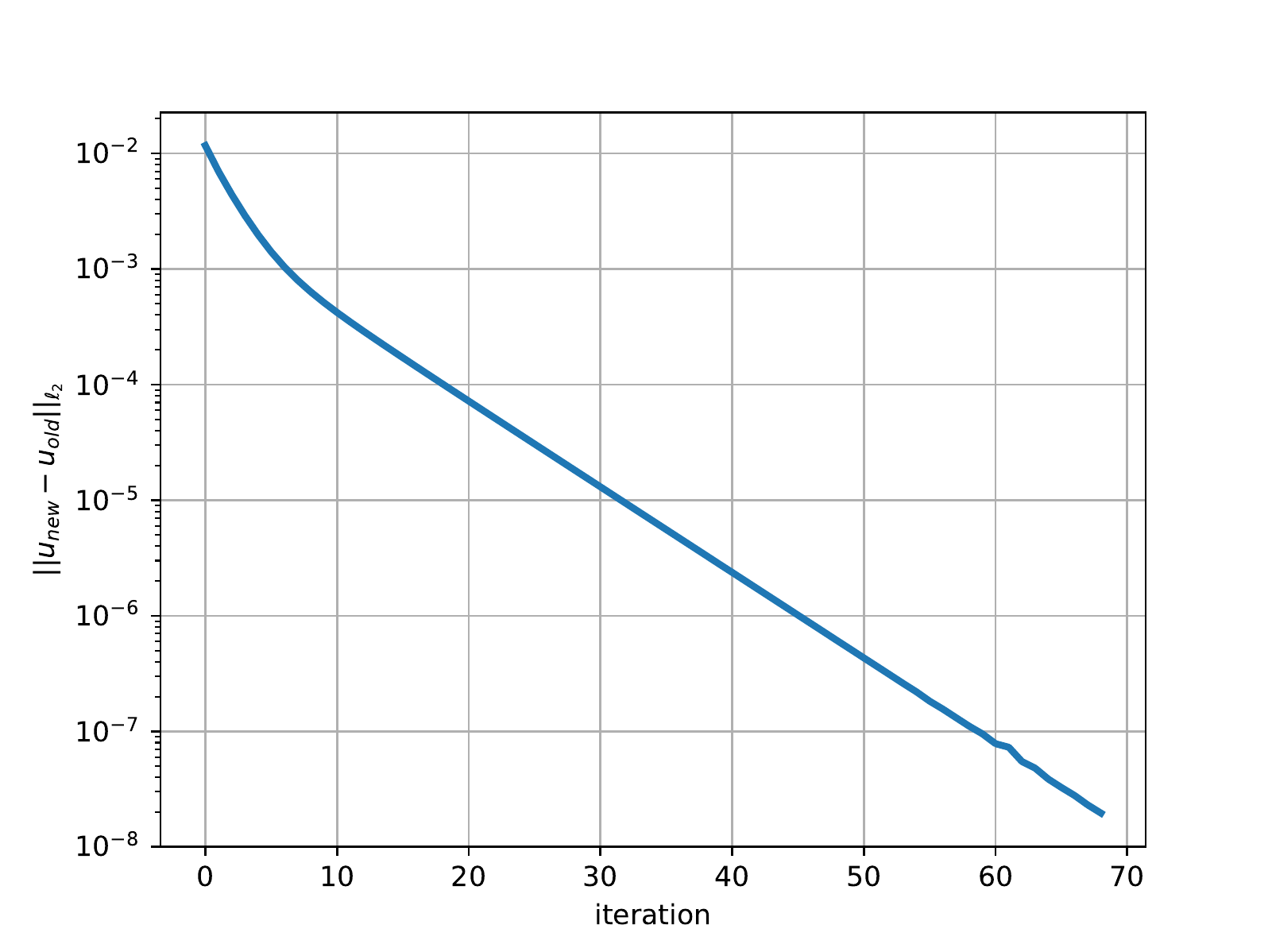}
\end{subfigure}
\\
\begin{subfigure}{.45\textwidth}
  \centering
  \includegraphics[width=\linewidth]{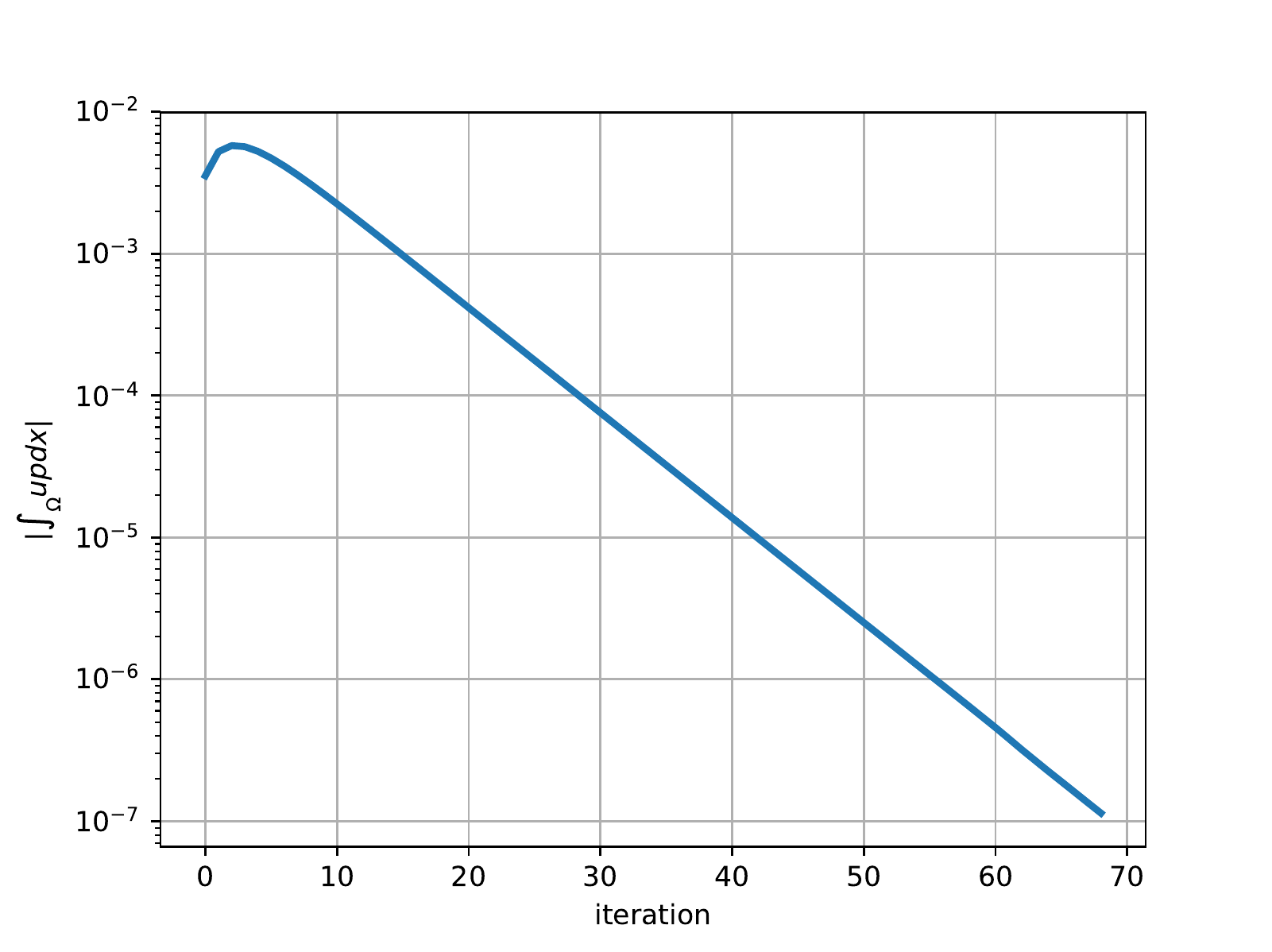}
\end{subfigure}
\caption{Case 5: Numerical convergence: normalized $\ell^2$ distance between two iterations, for $p$ (top left) and $u$ (top right), and evolution of $\int_\Omega p u$ (bottom) with respect to the number of iterations.}
\label{fig:statio-conv-iter}
\end{figure}

\subsubsection{Relation between non-stationary and stationary solutions}
\label{sec:statio-nonstatio-turnpike}

Figure~\ref{fig:ergo-time-turnpike} displays the evolution of (spatial) normalized $\ell^2$ distance between the solution to the stationary problem and the solution to the time-dependent problem, with respect to time. The graphs are shown for several values of the time horizon $T$. One can see that, provided $T$ is large enough, the distance tends to $0$ for $t$ bounded away from $0$ and $T$. This is reminiscent of the so-called turnpike phenomenon: starting from the initial condition, the solution quickly evolves to a stationary regime, which it leaves around the end of the time interval in order to satisfy the final condition.

\begin{figure}
\centering
\begin{subfigure}{.45\textwidth}
  \centering
  \includegraphics[width=\linewidth]{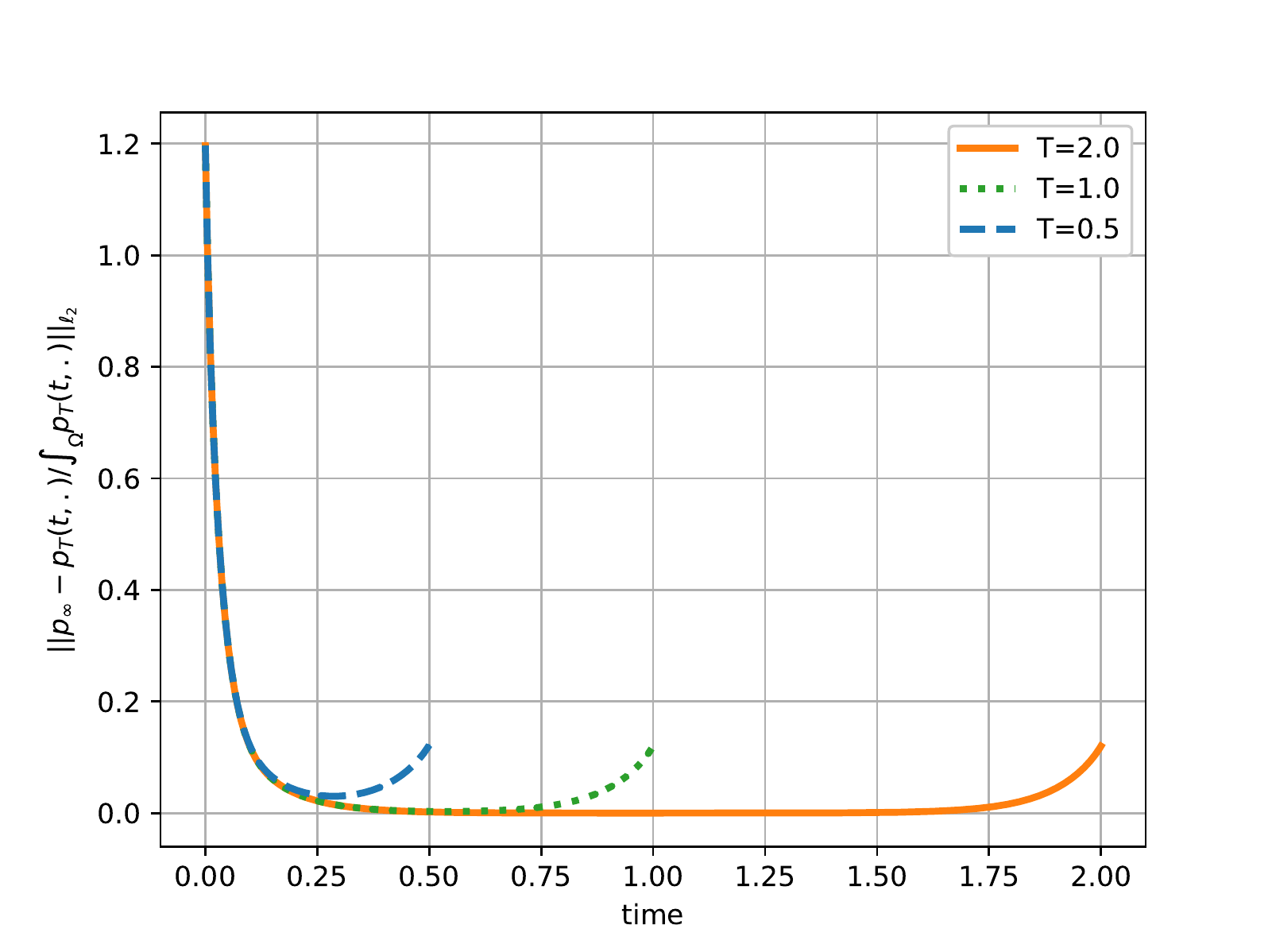}
\end{subfigure}%
\begin{subfigure}{.45\textwidth}
  \centering
  \includegraphics[width=\linewidth]{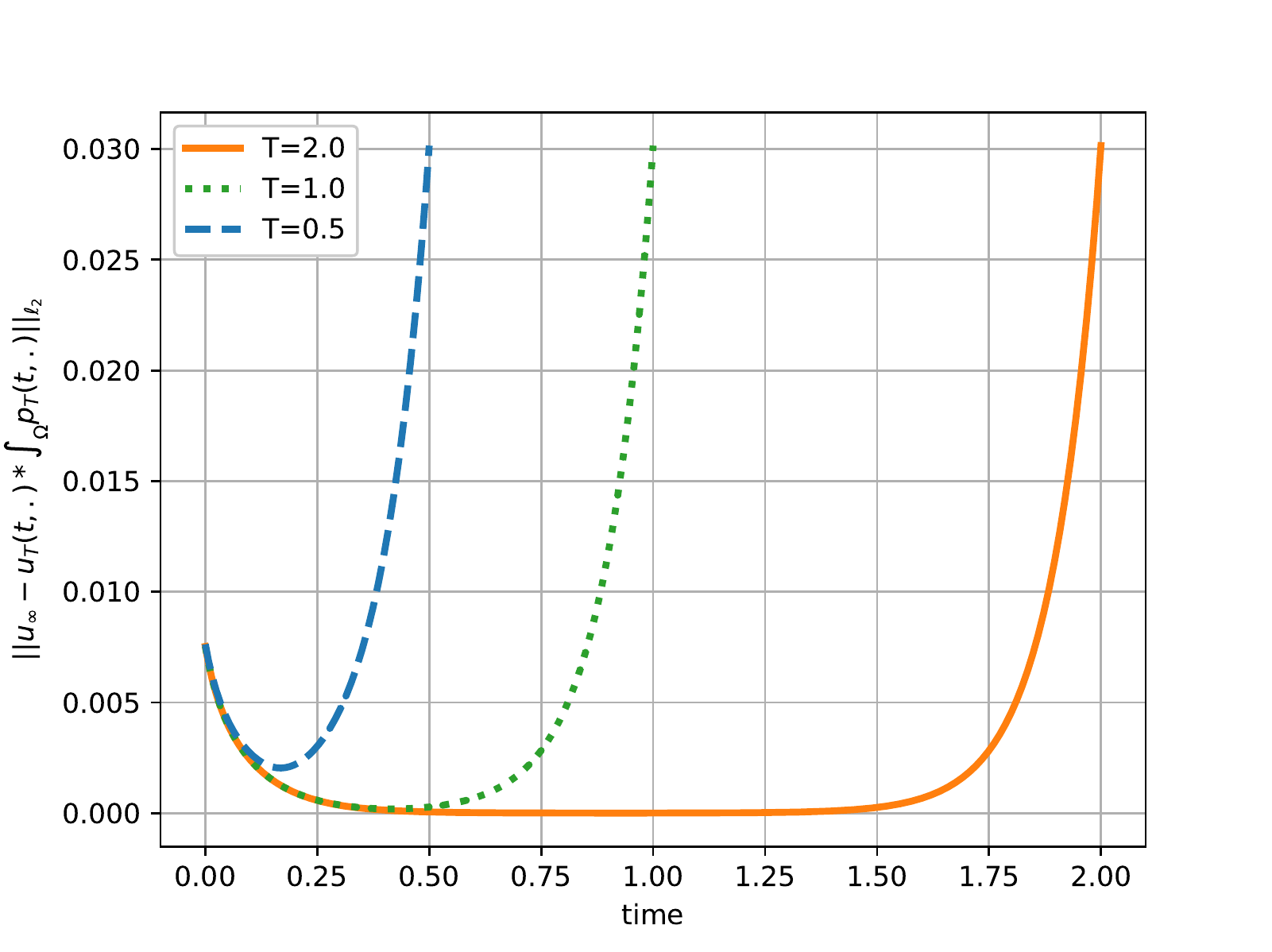}
\end{subfigure}
\caption{Case 5: Comparison between the stationary and non-stationary solutions for $3$ different time horizons ($T=0.5, 1, 2$): normalized $\ell^2$ distance between stationary and non-stationary $p$ (left) and $u$ (right).}
\label{fig:ergo-time-turnpike}
\end{figure}

\subsection{Alternative method for the finite time horizon problem}
\label{sec:alternative-method}

As described above (see Section~\ref{sec:time-longT}), the fact that the total mass tends to $0$ as time increases leads to numerical difficulties for large time horizons. However, the long time behavior suggests that a properly scaled version of the time dependent version should be more amenable to numerical treatment. As evidenced by Figure~\ref{fig:ergo-time-turnpike}, the proper scaling should come from the total remaining mass, but this is unknown until we solve the non-stationary problem. Let us instead rescale the functions using the ergodic constant coming from the non-stationary problem. More precisely, for $\gamma >0$, let $(p^{(T,\gamma)}, u^{(T,\gamma)})$ solve the following scaled system
\begin{equation}
	\label{eq:time-HJB-scaled}
	\left\{ \quad
	\begin{aligned}
	&-\partial_t u(t,x) -\frac {\sigma^2} 2 \Delta u(t,x) +\frac { H\left(x, \left(\int_\Omega  p(t) \right) Du(t,x)\right) }  {\int_\Omega  p(t)}
	=
	 \gamma u(t,x) +
	   \ds \frac {F\left[ \frac {p(t,\cdot)} {\int_\Omega p(t) } \right](x)} {\int_\Omega  p(t)}  +c_1(t)  ,
	 \quad \hbox{ in } Q_T,
	 \\
 	& u = 0, \qquad \hbox{ on } (0,T)\times\partial \Omega,
	\\ 
	& u(T,x)=  \frac 1  {\int_\Omega  p(T)} G\left[ \frac {p(T,\cdot)} {\int_\Omega p(T) } \right](x)  +  c_2(T),
	\qquad \hbox{ in } \Omega ,
    \end{aligned}
	\right.
	\end{equation}
\begin{equation}
\label{eq:time-FP-scaled}
 \left\{ \quad 
 \begin{aligned}
	& \partial_t p(t,x) -\frac {\sigma^2} 2 \Delta p(t,x) - \diver\left( p(t,\cdot) H_\xi \left(\cdot,  \left(\int_\Omega  p(t)\right)   Du(t,\cdot) \right) \right)(x)
	=
	\gamma p(t,x), \quad  \hbox{in } Q_T,
	\\
	& p = 0,\quad\quad \qquad \hbox{ on } (0,T)\times\partial \Omega,
	\\
	& p(0,\cdot) = p_0, \qquad \hbox{ in } \Omega, 
  \end{aligned}
\right.
\end{equation}
and $c_1,c_2: [0,T] \to \RR$ are defined by~(\ref{eq:62}).

In order to solve the original system~\eqref{eq:time-HJB}--\eqref{eq:time-FP}, we propose the following method:
\begin{enumerate}
	\item Solve the non-stationary problem~\eqref{eq:S-FP}--\eqref{eq:S-HJB} and obtain $(\tilde p, \tilde u, \lambda)$.
	\item Solve the scaled system with parameter $\gamma = \lambda$ and obtain $(p^{(T,\lambda)}, u^{(T,\lambda)})$.
	\item Un-scale the solution by letting $u^{(T)}(t,x) = e^{\lambda t} u^{(T,\lambda)}(t,x)$ and $p^{(T)}(t,x) = e^{-\lambda t} p^{(T,\lambda)}(t,x)$.
\end{enumerate}
One can check that $(p^{(T)}, u^{(T)})$ indeed solves~\eqref{eq:time-HJB}--\eqref{eq:time-FP}.

\bigskip

\noindent {\bf Acknowledgment.}
  The research of the first author was partially supported by the ANR (Agence Nationale de la Recherche) through
MFG project ANR-16-CE40-0015-01.

\bibliographystyle{plain}
\bibliography{controlcond-bib}

\end{document}